\documentclass[11pt]{amsart}

\usepackage[english]{babel}

\usepackage{txfonts}

\usepackage{amsmath}
\usepackage{amssymb}
\usepackage{mathrsfs}

\usepackage{color}
\usepackage{esint}
\usepackage{enumerate}

\usepackage[colorlinks=true,
  linkcolor=blue,
  citecolor=red,
  urlcolor=magenta,
  backref=page]{hyperref}

\usepackage[T1]{fontenc}
\usepackage{lmodern}     

\usepackage{anysize}

\allowdisplaybreaks
\newtheorem{theorem}{Theorem}[section]
\newtheorem{proposition}[theorem]{Proposition}
\newtheorem{corollary}[theorem]{Corollary}
\newtheorem{lemma}[theorem]{Lemma}

\theoremstyle{definition}
\newtheorem{definition}[theorem]{Definition}
\newtheorem{example}[theorem]{Example}

\newtheorem{convention}[theorem]{Convention}

\theoremstyle{remark}
\newtheorem{remark}[theorem]{Remark}

\numberwithin{equation}{section}

\newcommand{\ave}[1]{\langle #1\rangle}

\newcommand{\abs}[1]{|#1|}
\newcommand{\Babs}[1]{\Big|#1\Big|}
\newcommand{\Norm}[2]{\|#1\|_{#2}}
\newcommand{\BNorm}[2]{\Big\|#1\Big\|_{#2}}
\def\gfz{\genfrac{}{}{0pt}{}}

\begin{document}

\title[Sobolev and BV spaces on metric measure spaces]
{Weak-type characterizations of Sobolev and bounded variation spaces on metric measure spaces}

\author[T. P. Hyt\"onen]{Tuomas P.\ Hyt\"onen$^*$}
\address{Department of Mathematics and Systems Analysis, Aalto University, P.O. Box 11100
(Otakaari~1), \mbox{FI-00076} Aalto, Finland}
\curraddr{}
\email{tuomas.hytonen@aalto.fi}

\thanks{$^*$Corresponding author.}

\author[D. Yang]{Dachun Yang}
\address{Laboratory of Mathematics and Complex Systems (Ministry of Education of China),
School of Mathematical Sciences, Institute for Advanced Study,
Beijing Normal University, Beijing 100875, The People's Republic of China}
\curraddr{}
\email{dcyang@bnu.edu.cn}

\author[W. Yuan]{Wen Yuan}
\address{Laboratory of Mathematics and Complex Systems (Ministry of Education of China),
School of Mathematical Sciences, Beijing Normal University, Beijing 100875, The People's
Republic of China}
\curraddr{}
\email{wenyuan@bnu.edu.cn}

\author[Y. Zhao]{Yirui Zhao}
\address{Laboratory of Mathematics and Complex Systems (Ministry of Education of China),
School of Mathematical Sciences, Beijing Normal University, Beijing 100875, The People's
Republic of China}
\curraddr{}
\email{yiruizhao@mail.bnu.edu.cn}

\thanks{This project was supported by
the National Natural Science Foundation of China
(Grant Nos. 12431006 and 12371093), the Beijing Natural Science
Foundation (Grant No. 1262011), the Fundamental Research Funds
for the Central Universities (Grant No. 2253200028),
and the Research Council of Finland (Grant Nos. 364208 and 371637).}

\date{\today.}

\renewcommand{\subjclassname}{\textup{2020} Mathematics Subject Classification}


\begin{abstract}
Given a complete doubling metric measure space $(X,\rho,\mu)$ supporting a Poincar\'e inequality, 
we prove weak-type characterizations of the Sobolev space $\dot{W}^{1,p}(\mu)$ and the space of functions of bounded variation, achieving a full analogy in general Poincar\'e spaces with the Euclidean results of Brezis et al. [Anal. PDE 17 (2024), 943--979]. The main novelty is that the finiteness of a weak-type norm, which only refers to differences or mean oscillations of $f$ without assuming any smoothness a priori, already guarantees the membership of $f$ in the relevant Sobolev or BV space. This distinguishes our contribution from the recent work of F.~Dai et al. [Adv.\ Math.\ 502 (2026), Paper No.\ 111153], where the related norm-equivalence was obtained under the a priori Lipschitz assumption on $f$. A key intermediate step in our approach is a new localized Bourgain--Brezis--Mironescu type characterization.

More precisely, we prove that, if
$p\in(1,\infty)$ and $\gamma\in\mathbb R\setminus\{0\}$, then, for any $f\in
L^1_{\mathrm{loc}}(\mu)$,
\begin{equation*}\tag{$*$}
 \|f\|_{\dot W^{1,p}(\mu)}
 \sim
 \|\rho^{-1}\phi^{-\gamma}F\|_{L^{p,\infty}(\phi^{\gamma p}V^{-1})},
 \qquad F\in\{\Delta f,m_f\},\quad \phi\in\{\rho,V\},
\end{equation*}
where the homogeneous Sobolev space $\dot{W}^{1,p}(\mu)$ is defined by the
minimal $p$-weak upper gradient and,
for any $x,y\in X$, we denote $V(x,y):=\mu(B(x,\rho(x,y)))$ and
$\Delta f(x,y):=|f(x) - f(y)|$, and $m_f(x,y)$ is the mean
oscillation of $f$ on the ball $B(x,\rho(x,y))$.
For $p=1$, the equivalence $(*)$ holds after replacing $\|f\|_{\dot W^{1,1}(\mu)}$ by
a bounded variation norm and restricting the parameters to the optimal ranges
$\gamma\in(-\infty,-1)\cup(0,\infty)$ for $\phi=\rho$ or
$\gamma\in (-\infty,-\frac1d)\cup(0,\infty)$ for $\phi=V$,
where $d\in(0,\infty)$ is the lower dimension of~$X$.
\end{abstract}

\maketitle

\tableofcontents

\section{Introduction}

\subsection{Background}

One of the recurring themes in the theory of Sobolev spaces is to recover
first-order smoothness without explicitly differentiating the function.  Such
characterizations are especially useful on spaces where a linear structure,
convolution, or distributional derivatives are unavailable.

The use of difference quotients and moduli of smoothness
to describe first-order Sobolev smoothness has a long history; see, for
example, DeVore and Sharpley
\cite{DS84} and the references therein.  In the more specific setting of
nonlocal functionals, an influential breakthrough was due to Bourgain,
Brezis, and Mironescu \cite{BBM}, which is called the
Bourgain--Brezis--Mironescu formula or the BBM formula for short. Precisely speaking,
they proved that, for
$p\in(1,\infty)$ and
$f\in L^p(\mathbb R^d)$,
\begin{equation}\label{eq:intro-BBM}
 \lim_{s\nearrow 1}(1-s)
 \iint_{\mathbb R^d\times\mathbb R^d}
 \frac{|f(x)-f(y)|^p}{|x-y|^{d+sp}}\,dx\,dy
 = \kappa_{d,p}\|\nabla f\|_{L^p(\mathbb R^d)}^p,
\end{equation}
where, for any $e\in\mathbb S^{d-1}$,
\begin{equation*}
 \kappa_{d,p}:=\frac{1}{p}\int_{\mathbb S^{d-1}}
 |\omega\cdot e|^p\,d\sigma(\omega),
\end{equation*}
which is independent of the choice of $e$.  In the endpoint case $p=1$,
Bourgain et al. obtained the corresponding liminf and limsup
estimates, while D\'avila \cite{Davila02} proved the exact identity
\eqref{eq:intro-BBM} for any $f\in \operatorname{BV}(\mathbb R^d)$, with its
right-hand side replaced by $\kappa_{d,1}|Df|(\mathbb R^d)$.
This formula repairs the well-known defect that the Gagliardo seminorm at
$s=1$ is infinite unless $f$ is essentially constant.  It also initiated a
large literature on nonlocal approximations of Sobolev and BV energies.
Inspired by the BBM formula, Nguyen \cite{N:06} and Bourgain and Nguyen
\cite{BN06} established some characterizations of Sobolev spaces via limits of a
class of nonlocal nonconvex functionals.

A different resolution of the same endpoint defect of
the Gagliardo seminorm was discovered by Brezis,
Van Schaftingen, and Yung \cite{BVY21}.  Instead of renormalizing a family of
strong $L^p$ quantities, they replaced the Gagliardo seminorm at
$s=1$ by a weak $L^p$ norm and obtained a direct representation of
$\|\nabla f\|_{L^p(\mathbb R^d)}$.  Brezis et al.\ \cite{BSVY} then placed this phenomenon
in a one-parameter family and
obtained some new characterizations of Sobolev and BV spaces, which we call
the BSVY formulae for short.
In a normalization
adapted to the present paper, they considered the following functionals
\begin{equation*}
 \mathcal E_{\lambda,\gamma}(f)
 :=\lambda^p
 \iint_{\gfz{(x,y)\in\mathbb R^d\times\mathbb R^d}
 {|f(x)-f(y)|>\lambda|x-y|^{1+\gamma}}}
 |x-y|^{\gamma p-d}\,dx\,dy,
 \qquad \lambda\in(0,\infty),\ \gamma\in\mathbb R.
\end{equation*}
For any $p\in(1,\infty)$, $\gamma\in\mathbb R\setminus\{0\}$, and
$f\in L^1_{\mathrm{loc}}(\mathbb R^d)$, Brezis et al.\ \cite{BSVY} proved that
$f\in\dot W^{1,p}(\mathbb R^d)$
if and only if
$\sup_{\lambda\in(0,\infty)}\mathcal E_{\lambda,\gamma}(f)<\infty$.
Moreover,
\begin{equation}\label{eq:intro-BSVY}
 \sup_{\lambda\in(0,\infty)}\mathcal E_{\lambda,\gamma}(f)
 \sim\|\nabla f\|_{L^p(\mathbb R^d)}^p.
\end{equation}
When $p=1$ and $\gamma\in(-\infty,-1)\cup(0,\infty)$, the characterization
\eqref{eq:intro-BSVY} remains valid with
$\dot W^{1,1}(\mathbb R^d)$ and $\|\nabla f\|_{L^1(\mathbb R^d)}$ replaced by
$\dot{\rm BV}(\mathbb R^d)$ and $|Df|(\mathbb R^d)$, respectively.  The role
of $\gamma$ is transparent: when $\gamma=-1$, the functional $\mathcal E_{\lambda,\gamma}$
reduces to the
one considered by Bourgain and Nguyen \cite{BN06}, whereas, when $\gamma=\frac dp$, the
weight $|x-y|^{\gamma p-d}$ disappears and the functional $\mathcal E_{\lambda,\gamma}$
reduces to the one
considered in \cite{BVY21}.  The parameter $\gamma$ thus places the work of
Bourgain--Nguyen \cite{BN06} and the work of Brezis et al.\ \cite{BVY21}
into a uniform framework. In \cite{BSVY}, they also established
some exact limit formulae.  Both BBM and BSVY formulae have been extensively
developed.  For developments of the BBM formula, we
refer to \cite{BBM02,Davila02,DominguezMilman23,Ponce04} in the  Euclidean space;
to \cite{DrelichmanDuran22,Mohanta24} on domains; and to \cite{DMS,Han24}
on metric measure spaces.  We next highlight two lines of development
of the BSVY formula that are particularly relevant to the present paper.

The first line replaces the difference quotients in BSVY formulae by
quantities associated with other operators.  A particularly relevant example
is the mean oscillation characterization of Frank \cite{Frank:W1p},
motivated by the study of trace ideal properties of commutators.  For a ball
$B\subset\mathbb R^d$, write
\begin{equation*}
 \langle f\rangle_B:=\fint_B f(y)\,dy
 :=\frac{1}{|B|}\int_B f(y)\,dy.
\end{equation*}
For any $f\in L^1_{\rm loc}(\mathbb R^d)$, define
\begin{equation*}
 m_f(x,t):=\fint_{B(x,t)}|f(y)-\langle f\rangle_{B(x,t)}|\,dy,
\qquad x\in\mathbb R^d\ \text{and}\ t\in(0,\infty),
\end{equation*}
where $B(x,t):=\{y\in\mathbb R^d:\ |y-x|<t\}$.
Frank \cite{Frank:W1p} proved that, for $p\in(1,\infty)$,
\begin{equation}\label{eq:intro-Frank}
 \|m_f\|_{L^{p,\infty}(t^{-p-1}\,dt\,dx)}
 \sim \|\nabla f\|_{L^p(\mathbb R^d)}.
\end{equation}
Formula \eqref{eq:intro-Frank} may be viewed as the case $\gamma=-1$ of the
BSVY formula, with the difference $|f(x)-f(y)|$ replaced by the mean
oscillation $m_f(x,t)$; it also provides a new perspective on Sobolev theory
in metric measure spaces.
We refer to \cite{DominguezMilman22,DSSVY23,
DLYYZ:LiftedHL} for further developments along this
line.

The second line goes beyond Euclidean spaces to complete doubling metric measure
spaces supporting a Poincar\'e inequality.  In this setting, Di Marino and
Squassina \cite{DMS} established BBM-type and Nguyen-type characterizations.  Dai et
al.\ \cite{DLYYZ:JFA} obtained a volume version of the BSVY formula with all positive
$\gamma$ and also some negative $\gamma$ under some stronger assumptions, while
Hyt\"onen and Korte \cite{HK:W1p} extended Frank's mean oscillation
characterization.  More recently, Dai et al.\ \cite{DLYYZ} introduced families
of lifted sharp maximal operators that unify and extend \eqref{eq:intro-BSVY} and
\eqref{eq:intro-Frank} and further established both formulae
over the optimal ranges of $\gamma$, including the endpoint
$p=1$.  Other related extensions in weighted and metric settings can be found
in \cite{Munnier,DLYYZ:CVPDE,ZYY23,ZLYYZ24}.

The conclusions in \cite{DLYYZ}, however, are formulated as equivalent
seminorms on Lipschitz functions.  A norm comparison on a test class does
not by itself show that the finiteness of the nonlocal weak-type quantity
characterizes the entire Sobolev or BV space.  One of the main motivations of
the present paper is to remove this a priori Lipschitz assumption and obtain
full-space characterizations.

\subsection{Main results}

Throughout the introduction, $(X,\rho,\mu)$ is a metric measure space.
For any $x,y\in X$ and $f\in L^1_{\rm loc}(\mu)$, we
write
\begin{align*}
 V(x,r):=\mu(B(x,r)),\qquad
 V(x,y):=V(x,\rho(x,y)),
\end{align*}
$\Delta f(x,y):=|f(x)-f(y)|$ and, for any $r\in(0,\infty)$, let
\begin{align*}
\langle f\rangle_{B(x,r)}:=\fint_{B(x,r)}f\,d\mu:=\frac{1}{\mu(B(x,r))}\int_{B(x,r)}
 f\,d\mu,
\end{align*}
and
\begin{align}\label{eq-def:mf}
 m_f(x,r):=\fint_{B(x,r)}|f-\langle f\rangle_{B(x,r)}|\,d\mu,
 \qquad m_f(x,y):=m_f(x,\rho(x,y)).
\end{align}
The spaces $\dot W^{1,p}(\mu)$ and $\dot M^{1,p}(\mu)$ denote, respectively,
the homogeneous Sobolev space defined by the minimal $p$-weak upper gradient in the sense of
\cite{AGS:13} and the
homogeneous Haj\l{}asz--Sobolev space of \cite{Haj:96}.  Their precise definitions are
recalled
in Section~\ref{sec:preliminaries} below.  In the displayed weak norms, the
expression inside the parentheses specifies the underlying measure.  Thus,
for example,
\begin{equation*}
 \|F\|_{L^{p,\infty}(\rho^{p\gamma}V^{-1})}
 :=\sup_{\lambda\in(0,\infty)}\lambda
 \left[\iint_{\{|F(x,y)|>\lambda\}}
 \frac{\rho(x,y)^{p\gamma}}{V(x,y)}\,d\mu(y)\,d\mu(x)\right]^{\frac1p}.
\end{equation*}

Our first group of results gives the characterizations of Sobolev spaces.  The
following statements reorganize Theorems~\ref{thm:Frank-type},
\ref{thm:Nguyen-type}, \ref{thm-Nguyen-type-v2}, and \ref{thm:pd-Nguyen-type}.

\begin{theorem}[Sobolev characterizations]\label{thm:intro-Sobolev}
Let $(X,\rho,\mu)$ be complete and doubling, let $p\in(1,\infty)$ and
$\gamma\in\mathbb R\setminus\{0\}$, and let $f\in L^1_{\mathrm{loc}}(\mu)$.
\begin{enumerate}[\rm(i)]
\item For $F=\Delta f$ and any $\phi\in\{\rho,V\}$,
\begin{equation}\label{eq:intro-Sobolev-sandwich}
 \|f\|_{\dot W^{1,p}(\mu)}
 \lesssim
 \|\rho^{-1}\phi^{-\gamma}F\|_
 {L^{p,\infty}(\phi^{\gamma p}V^{-1})}
 \lesssim \|f\|_{\dot M^{1,p}(\mu)}.
 \end{equation}
\item If $X$ also supports a $p$-Poincar\'e inequality, then, for any
$F\in\{m_f,\Delta f\}$ and $\phi\in\{\rho,V\}$,
\begin{equation}\label{eq:intro-Sobolev-equivalence}
 \|f\|_{\dot W^{1,p}(\mu)}
 \sim
 \|\rho^{-1}\phi^{-\gamma}F\|_
 {L^{p,\infty}(\phi^{\gamma p}V^{-1})}.
\end{equation}
Consequently, the weak-type quantity in
\eqref{eq:intro-Sobolev-equivalence} is finite if and only if $f\in
\dot W^{1,p}(\mu)$.
\end{enumerate}
\end{theorem}

\begin{remark}
\begin{enumerate}[\rm(i)]
\item If $X$ is additionally reverse doubling, then
\eqref{eq:intro-Sobolev-sandwich} also holds with $F=m_f$ and both choices of
$\phi$; see Corollary~\ref{cor:mf-reverse-doubling}.
\item The range $\gamma\in\mathbb R\setminus\{0\}$ in Theorem
\ref{thm:intro-Sobolev} is optimal. Indeed, when $X=\mathbb R^d$, the upper
estimate in \eqref{eq:intro-Sobolev-sandwich}
fails for $\gamma=0$ (see Brezis et al.\
\cite[Theorem 1.3(i)]{BSVY} for $F=\Delta f$
and \cite[Remark 1.4(ii)]{ZLYYZ24} for $F=m_f$).
\end{enumerate}
\end{remark}

Under the Poincar\'e assumption, the choice $(F,\phi)=(m_f,\rho)$ also has the
equivalent formulation
\begin{equation}\label{eq:intro-Frank-main}
 \|f\|_{\dot M^{1,p}(\mu)}
 \sim
 \|t^{-1-\gamma}m_f\|_{L^{p,\infty}(t^{\gamma p-1}\,dt\,d\mu)},
\end{equation}
which is closest to Frank's original formula \eqref{eq:intro-Frank}.
The full statements in the body
of the paper also contain intermediate versions involving the lifted sharp
maximal functions $\mathcal M_\theta^\#$ and $f_\theta^*$ introduced in
Section \ref{s:Frank}.
The rightmost estimates in \eqref{eq:intro-Sobolev-sandwich} are based
primarily on the lifted sharp maximal function bounds of Dai et al.\ \cite{DLYYZ}.
The main contribution of the
present paper is the opposite estimate
\begin{equation}\label{eq:lower-qu}
 \|f\|_{\dot W^{1,p}(\mu)}\lesssim \text{nonlocal weak-type quantity},
\end{equation}
which removes the Lipschitz assumption of
\cite{DLYYZ} and turns an equivalence of seminorms on
test functions into a characterization of the full homogeneous Sobolev space.

The strategy of deriving the lower estimate \eqref{eq:lower-qu}
by reducing it to
a BBM lower bound is already present in the Euclidean argument of Brezis et
al.\ \cite[Section~4]{BSVY}; see also
Poliakovsky \cite[Theorem~1.3 and Section~3]{Poliakovsky22} for a related use
of the BBM formula on Euclidean domains.  In contrast, our lower estimate rests on a
localized homogeneous BBM characterization.  Denote the truncation
$f_N:=\max\{-N,\min\{f,N\}\}$ for any $N\in\mathbb N$. The characterization may be
summarized as
follows.

\begin{proposition}[Localized BBM characterization]\label{prop:intro-BBM}
Let $p\in(1,\infty)$, let $(X,\rho,\mu)$ be complete and doubling, and assume
that $X$ supports a $p$-Poincar\'e inequality.  Then, for any
$f\in L^1_{\mathrm{loc}}(\mu)$,
\begin{equation}\label{eq:intro-local-BBM}
 \|f\|_{\dot W^{1,p}(\mu)}^p
 \sim
 \sup_{E,N}\liminf_{s\nearrow 1}(1-s)
 \|\mathbf{1}_{E\times E}\Delta f_N\|_{L^p(\rho^{-ps}V^{-1})}^p
\end{equation}
with the same conclusion when $\liminf$ is replaced by $\limsup$.  Here the
supremum is taken over all bounded measurable subsets $E\subset X$ and
$N\in\mathbb N$.
\end{proposition}

The localization in \eqref{eq:intro-local-BBM} is essential for homogeneous
spaces: a general $f\in\dot W^{1,p}(\mu)$ need not belong to any fixed
fractional Sobolev space globally (see Remark \ref{rem:loc-global}).  Bounded exhaustion,
truncation, and stability of
weak upper gradients allow the localized fractional information to be glued
into a global Sobolev upper gradient.  Combining this lower estimate of
\eqref{eq:intro-local-BBM}
with Lorentz space duality as used in Brezis et al.\ \cite{BSVY} yields the first
inequality in \eqref{eq:intro-Sobolev-sandwich} with $F=\Delta f$ and either
choice of $\phi$.

We also obtain endpoint BV characterizations.  The total variation
$|Df|(X)$ and the homogeneous space $\dot{\rm BV}(\mu)$ are defined in
\eqref{eq:Df(A)} and \eqref{BV}, respectively.  A complete doubling
$1$-Poincar\'e space is automatically reverse doubling and hence
there exists the lower homogeneous dimension
$d\in(0,\infty)$, which may be viewed as the largest exponent appearing in
the reverse doubling condition; see Section~\ref{sec:BV} and see also
\cite[(1.14)]{DLYYZ}.

\begin{theorem}[BV characterizations]\label{thm:intro-BV}
Let $(X,\rho,\mu)$ be a complete doubling $1$-Poincar\'e space
with the lower homogeneous dimension $d\in(0,\infty)$ and let
$f\in L^1_{\mathrm{loc}}(\mu)$.  For $F\in\{\Delta f,m_f\}$, the following statements hold.
\begin{enumerate}[\rm(i)]
\item If $\gamma\in(-\infty,-1)\cup(0,\infty)$, then
\begin{equation}
 |Df|(X)
 \sim \|\rho^{-1-\gamma}F\|_{L^{1,\infty}(\rho^\gamma V^{-1})}.
\end{equation}
\item If $\gamma\in(-\infty,-\frac1d)\cup(0,\infty)$, then
\begin{equation}\label{eq:intro-BV-V}
 |Df|(X)
 \sim \|\rho^{-1}V^{-\gamma}F\|_{L^{1,\infty}(V^{\gamma-1})}.
\end{equation}
\end{enumerate}
In either case, the finiteness of the corresponding weak-type quantity is
equivalent to $f\in\dot{\rm BV}(\mu)$.
\end{theorem}

\begin{remark}
The ranges of $\gamma$ in Theorem \ref{thm:intro-BV} are optimal. When
$X=\mathbb R^d$, this follows from Brezis et al.\
\cite[Theorem~1.4(i)]{BSVY} for $F=\Delta f$ and \cite[Remark 1.4(ii)]{ZLYYZ24}
for $F=m_f$. Moreover, the range
$\gamma\in(-\infty,-\frac1d)\cup(0,\infty)$ in
\eqref{eq:intro-BV-V} remains optimal on general complete Poincar\'e spaces,
as shown by the Grushin-type examples of Dai et al.\ \cite[Section~7]{DLYYZ}.
\end{remark}

\subsection{Relation with known results}

On $\mathbb R^d$, Theorems \ref{thm:intro-Sobolev} and \ref{thm:intro-BV} with
$F=\Delta f$ and either
choice of $\phi$ coincide with BSVY characterizations \eqref{eq:intro-BSVY}, while
Theorem \ref{thm:intro-Sobolev} with
$(F,\phi)=(m_f,\rho)$ and $\gamma=-1$ or
$(F,\phi)=(m_f,V)$ and $\gamma=-\frac1d$
coincides with Frank characterizations \eqref{eq:intro-Frank}.
On complete doubling Poincar\'e spaces, the
case $(F,\phi)=(m_f,\rho)$ of Theorem \ref{thm:intro-Sobolev},
equivalently \eqref{eq:intro-Frank-main}, coincides with
the characterization of \cite[Theorem~1.6]{HK:W1p}.

It is worth pointing out that, without a Poincar\'e inequality,
\eqref{eq:intro-Sobolev-sandwich} places the weak-type quantity between
$\|f\|_{\dot W^{1,p}(\mu)}$ and $\|f\|_{\dot M^{1,p}(\mu)}$ for
$F=\Delta f$.  Han et al.\
\cite[Theorem~1.3]{HXZ} obtained an analogous sandwich between two
Sobolev-type seminorms on Lipschitz differentiability spaces, and
\cite[Corollary~1.5]{HXZ} obtained an equivalence under additional geometric
assumptions.  While their results are restricted to compactly supported Lipschitz functions,
\eqref{eq:intro-Sobolev-sandwich} holds for any $f\in L^1_{\rm loc}(\mu)$ and,
under a Poincar\'e inequality, yields the characterization
\eqref{eq:intro-Sobolev-equivalence} of Sobolev spaces.

As mentioned before, the equivalences in Theorems \ref{thm:intro-Sobolev} and
\ref{thm:intro-BV} with $F=\Delta f$ and
$(F,\phi)=(m_f,\rho)$ have been proved by \cite[Corollaries 1.9, 1.13, and~1.14]{DLYYZ} for
all
Lipschitz functions.  Theorems
\ref{thm:intro-Sobolev} and \ref{thm:intro-BV} remove this a priori
Lipschitz assumption and obtain full characterizations of the Sobolev and BV spaces.

Besides their intrinsic interest in the theory of function spaces, such characterizations
are
useful in applications to commutators $[f,T]$ of pointwise multipliers $f$ and
certain singular integrals $T$; see \cite{Hyt:Sp} for a recent general theory in
metric measure spaces (under assumptions similar to those of the present paper) and
references
therein for earlier developments in $\mathbb R^d$ and other concrete situations. As shown in
\cite{RS:NWO}
in the Euclidean case and extended to $(X,\rho,\mu)$ in \cite{Hyt:Sp}, there are close links
between various mapping properties of $[f,T]$ and the size of the mean oscillation function
$m_f$.
In particular, certain decay of the singular values of $[f,T]$, which is relevant in the
{\em quantized calculus} of \cite{Connes:book}, is characterized by the finiteness of the
right-hand side of \eqref{eq:intro-Frank-main} with $\gamma=-1$. By itself, this would be a
slightly  obscure characterization. But, in combination with \eqref{eq:intro-Frank-main},
it gives a much more illuminating characterization in terms of the condition
$f\in\dot M^{1,p}(\mu)$, allowing \cite{Hyt:Sp} to recover concrete earlier results of
\cite{CST,LMSZ} and others within the general framework of metric measure spaces.
For the case $\gamma=-1$ relevant to \cite{Hyt:Sp}, the characterization
\eqref{eq:intro-Frank-main}
was already known from \cite{HK:W1p}, but its present extension to other values
may play a similar role in further developments of the commutator theory.

\subsection{Organization and notation}

Section~\ref{sec:preliminaries} develops the required notation
and conclusions of Sobolev and BV spaces on metric measure spaces, including
(weak) upper gradients, the approximation by bounded subsets, and the
equivalence of different Sobolev or BV spaces.  The target of
Section~\ref{s:Frank} is to prove the Frank-type
characterization \eqref{eq:intro-Frank-main}.
In Section~\ref{sec:DMS}, we establish the localized BBM
principle of Proposition \ref{prop:intro-BBM} and prove its lower estimate.
The proof of the upper bound is postponed to Section \ref{sec:Nguyen}.
The aim of Section~\ref{sec:Nguyen} is to show Theorem \ref{thm:intro-Sobolev}
with $F=\Delta f$.
Section~\ref{sec:further-variants} treats further variants formulated with
mean oscillations, namely Theorem \ref{thm:intro-Sobolev} with
$(F,\phi)=(m_f,V)$.
In Section~\ref{sec:BV}, we prove Theorem \ref{thm:intro-BV}.

We end this section by introducing several notational conventions.  Let
$\mathbb N:=\{1,2,\ldots\}$ and $\mathbb Z_+:=\mathbb N\cup\{0\}$.
The symbol $C$ always denotes a
positive constant independent of the main parameters involved, but may vary
from line to line.   The notation $f\lesssim g$ indicates that $f\leq Cg$
for some positive constant
$C$.  When both $f\lesssim g$ and $g\lesssim f$ hold, we write $f\sim g$.
For a measurable set $E$ in the
underlying space, we write $\mathbf{1}_E$ for its characteristic function. The cardinality
of a
finite set $A$ is denoted by $\#A$.
Finally, in all proofs we
consistently retain the notation introduced in the original theorem (or
related statement).

\section{Preliminaries}\label{sec:preliminaries}

Throughout this paper, $(X,\rho,\mu)$ denotes a metric measure space, where
$(X,\rho)$ is a metric space and $\mu$ is a nontrivial Borel measure on $X$.
For $x\in X$ and $r\in(0,\infty)$, we write
\begin{equation*}
  V(x,r):=\mu(B(x,r)).
\end{equation*}
For $x,y\in X$ with $x\neq y$, we further abbreviate
\begin{equation*}
  V(x,y):=V(x,\rho(x,y))=\mu(B(x,\rho(x,y))).
\end{equation*}
We say that $\mu$ is \emph{locally finite} if every point has a neighbourhood of
finite measure, and that $(X,\rho,\mu)$ is \emph{doubling} if there is a constant
$C_\mu\in[1,\infty)$ such that
\begin{equation*}
  0<V(x,2r)\leq C_\mu V(x,r)<\infty
  \quad\text{for any }x\in X\text{ and }r\in(0,\infty).
\end{equation*}
We record some standard consequences of these assumptions. A doubling metric
measure space is separable. Indeed, after fixing $x_0\in X$, for any
$n\in\mathbb N$ choose a maximal $2^{-n}$-separated subset $D_n$ of
$B(x_0,n)$. The doubling condition implies that $D_n$ is finite, while
$\bigcup_{n=1}^\infty D_n$ is dense in $X$. Moreover, if $X$ is separable and
$\mu$ is locally finite, then $\mu$ is $\sigma$-finite: the finite-measure
neighbourhoods form an open cover of $X$, and separability gives a countable
subcover. Finally, a complete doubling metric measure space is proper. Indeed,
the doubling condition implies that every bounded subset is totally
bounded, so every closed bounded subset is compact by completeness.

In the rest of these preliminaries, we collect and complement a rather extensive set of
definitions
and results about the basic properties of Sobolev and BV spaces over a
metric measure space $(X,\rho,\mu)$, so as to streamline the discussion of their
new characterizations in the subsequent sections. The definition of these spaces depends
on the notions of curves and curve families and, in particular, negligible families of
curves
that can be discarded. Depending on the details of defining the latter, at least two
prominent
approaches have been developed in the literature: one based on the {\em modulus} of curve
families,
going back to \cite{HK:96,KM:98} and used, e.g., in the monographs
\cite{Heinonen:book,HKST:book},
and a more recent one based on so-called {\em test plans} introduced in
\cite{AGS:13,AGS:14},
with further variants e.g. in \cite{DMS}. We will borrow some results from both approaches,
and hence need some elements of both as preliminaries. Although the equivalence of these two
approaches is basically known since \cite[Theorem 7.4]{AGS:13},
there is a devil in the details: much of the existing literature concentrates on
inhomogeneous
Sobolev and BV spaces, imposing matching integrability on the function and its (weak upper)
gradient (or the total variation measure), while it is the homogeneous versions that feature
in
our main Theorems \ref{thm:intro-Sobolev} and \ref{thm:intro-BV}. In order to close this
gap, and
to make the exact auxiliary results that we need easily available, it seems that a somewhat
lengthy
section on preliminaries is unavoidable. On the other hand, these results may have some
independent
interest and applications in other questions dealing with homogeneous Sobolev spaces over
metric
measure spaces.

\subsection{Curves and curve families}\label{ss:curves}

Following \cite[Section 2.1]{AGS:13}, we say that
$\gamma:[0,1]\to X$ is \emph{absolutely continuous} if there is a function $h\in L^1(0,1)$
such that
\begin{equation}\label{eq:gamma-AC}
  \rho(\gamma(s),\gamma(t))\leq\int_s^t h(u)\,du\quad\text{for all }0\leq s<t\leq 1.
\end{equation}
The minimal such $h$ is called the \emph{metric speed} of $\gamma$ and given by
\begin{equation}\label{eq:metric-speed}
   \abs{\gamma'}(t):=\lim_{s\to t}\frac{\rho(\gamma(s),\gamma(t))}{\abs{s-t}}\quad\text{for
   a.e. }t\in[0,1];
\end{equation}
see \cite[Theorem 1.1.2]{AGS:book}. The length of $\gamma$ is
\begin{equation*}
  \ell(\gamma):=\int_0^1\abs{\gamma'}(t)\,dt.
\end{equation*}

Still following \cite[Section 2.1]{AGS:13}, let
\begin{equation*}
\begin{split}
  C([0,1];X) &:=\{\text{continuous curves $\gamma:[0,1]\to X$ with $\sup$ distance}\}, \\
  AC([0,1];X) &:=\{\gamma\in C([0,1];X) \text{ absolutely continuous}\}, \\
  AC^p([0,1];X) &:=\{\gamma\in AC([0,1];X):  \abs{\gamma'}\in L^p(0,1)\}, \\
  \mathscr P(C([0,1];X)) &:=
  \{\text{Borel probability measures on }C([0,1];X)\}.
\end{split}
\end{equation*}
It is observed in \cite[page 972]{AGS:13} that $AC^p([0,1];X)$ is a Borel subset for
$p\in(1,\infty)$. As pointed out in \cite[page 4154]{AD:14},
$AC^\infty([0,1];X)$ coincides with the set of all Lipschitz curves and, in this case,
we denote by ${\operatorname{Lip}}(\gamma)$ the Lipschitz constant of
$\gamma\in AC^\infty([0,1];X)$.

Given a Borel function $f:\mathcal{X}\to\mathcal{Y}$ and a Borel measure $m$
on $\mathcal{X}$, define the \emph{push-forward} $f_\#m$ as the measure on $\mathcal{Y}$
such that $f_\#m(A)=m(f^{-1}(A))$ for any Borel set $A\subset \mathcal Y$. Moreover, define
$e_t:C([0,1];X)\to X$ as the \emph{evaluation} of $\gamma$ at the time $t$, namely
$e_t(\gamma):=\gamma(t)$.

Given $\pi\in \mathscr P(C([0,1];X))$,
we say $\pi$ has the \emph{bounded compression property} if
there is a positive constant $C(\pi)$ such that, for any
$t\in[0,1]$,
$(e_t)_\#\pi\leq C(\pi)\mu$, i.e.,
\begin{equation}\label{it:C(pi)}
   \pi(\{\gamma\in C([0,1];X):\gamma(t)\in A\})\leq C(\pi)\mu(A)
\end{equation}
for all Borel sets $A\subset X$.

\begin{definition}[\cite{AGS:13}, Definition 4.6 and \cite{AD:14}, Definition 5.5]\label{def:fzplan}
Let $q\in [1,\infty]$ and let
$$\pi\in\mathscr P(C([0,1];X)).$$
We say that $\pi$ is
a \emph{$q$-test plan} if
\begin{enumerate}[\rm(i)]
  \item $\pi$ has the bounded compression property;
  \item\label{it:pi(AC)=1} $\pi(AC^q([0,1];X))=1$;
  \item\label{it:doubleInt} if $q\in[1,\infty)$, then
  $$
  \int_{C([0,1];X)}\int_0^1\abs{\gamma'(t)}^q\,dt\,d\pi(\gamma)<\infty,
  $$
  while, if $q=\infty$, then
  $$
  \Norm{\gamma\mapsto\operatorname{Lip}(\gamma)}{L^\infty(\pi)}<\infty.
  $$
\end{enumerate}
A Borel set $\Gamma\subset C([0,1];X)$ is said to be {\em $p$-negligible}
if $\pi(\Gamma)=0$ for every $q$-test plan $\pi$,
where $q=p'$ is the H\"older conjugate.
A property is said to \emph{hold for $p$-a.e.\ curve},
if it holds for all curves outside a $p$-negligible set.
\end{definition}

\begin{lemma}\label{lem:negligible}
Let $(X,\rho,\mu)$ be a metric measure space and $Y\subset X$ a closed subset.
Let $A\subset C([0,1];Y)\subset C([0,1];X)$ be a Borel set and $p\in[1,\infty]$.
Then $A$ is $p$-negligible as a subset of $C([0,1];Y)$ if and only if
it is $p$-negligible as a subset of $C([0,1];X)$.
\end{lemma}

\begin{proof}
Let first $A$ be $p$-negligible in $C([0,1];Y)$, and let $\pi$ be a $q$-test plan for $X$
with $q=p'$. We need to show that $\pi(A)=0$. If $\pi(C([0,1];Y))=0$, this is clear, so we
assume that $\pi_0:=\pi(C([0,1];Y))>0$. Let
\begin{equation*}
  \tilde\pi(B):=\pi_0^{-1}\pi(B)\quad\text{for all Borel sets }B\subset C([0,1];Y).
\end{equation*}
It is immediate that $\tilde\pi\in\mathscr P(C([0,1];Y))$; we will check that it is a
$q$-test plan for $(Y,\rho|_Y,\mu|_Y)$. Once this is verified, the $p$-negligibility of $A$
in $C([0,1];Y)$ shows that $\tilde\pi(A)=0$, and hence $\pi(A)=\pi_0\tilde\pi(A)=0$.

We turn to the verification of the $q$-test plan properties of $\tilde\pi$. By
Definition \ref{def:fzplan} for $\pi$ it follows that
\begin{equation*}
  \pi_0=\pi(C([0,1];Y)\cap AC^q([0,1];X))=\pi(AC^q([0,1];Y))
\end{equation*}
and hence $\tilde\pi(AC^q([0,1];Y))=1$, which verifies
Definition~\ref{def:fzplan}\eqref{it:pi(AC)=1}
for $(\tilde\pi,Y)$ in place of $(\pi,X)$.

For $q<\infty$, the integral in
Definition~\ref{def:fzplan}\eqref{it:doubleInt} for $\tilde\pi$ is the integral for
$\pi$ restricted to a subset and divided by the positive number $\pi_0$,
hence finite again. Similarly, for $q=\infty$, the restriction of an $L^\infty$-function to
the subset $C([0,1];Y)\subset C([0,1];X)$ remains an $L^\infty$-function over that subset,
also after dividing the measure by $\pi_0$.

Finally, if $F$ is a Borel subset of $Y$ and hence of $X$,
then, for any $t\in[0,1]$,
\begin{equation*}
\begin{split}
  \tilde\pi(\{\gamma\in C([0,1];Y):\gamma(t)\in F\})
  &=\frac{1}{\pi_0}\pi(\{\gamma\in C([0,1];Y):\gamma(t)\in F\}) \\
  &\leq\frac{1}{\pi_0}C(\pi)\mu(F)
  =\frac{1}{\pi_0}C(\pi)\mu|_Y(F),
\end{split}
\end{equation*}
so $C(\tilde\pi):=\pi_0^{-1}C(\pi)<\infty$ qualifies as the constant for $\tilde\pi$. This
completes the verification that $\tilde\pi$ is a $q$-test plan for $Y$, and hence of the
``only if'' part of the lemma.

Let then $A$ be $p$-negligible in $C([0,1];X)$, and let $\pi$ be a $q$-test plan for
$(Y,\rho|_Y,\mu|_Y)$ with $q=p'$. We can then define
\begin{equation*}
  \hat\pi(B):=\pi(B\cap C([0,1];Y))\quad\text{for all Borel sets }B\subset C([0,1];X).
\end{equation*}
It is straightforward to check that $\hat\pi$ is a $q$-test plan for $(X,\rho,\mu)$.
Since $A$ is $p$-negligible in $C([0,1];X)$, it follows that $\hat\pi(A)=0$.
Moreover, by the assumption $A\subset C([0,1];Y)$, we also obtain
\begin{equation*}
  \pi(A)=\pi(A\cap C([0,1];Y))=\hat\pi(A)=0.
\end{equation*}
As $\pi$ was an arbitrary $q$-test plan for $(Y,\rho|_Y,\mu|_Y)$,
this shows that $A$ is $p$-negligible in $C([0,1];Y)$.
This completes the proof of the ``if'' part, and hence of the lemma.
\end{proof}

Following \cite[Section 4.1]{AGS:13}, for $\gamma\in AC([0,1];X)$
and Borel functions $f$ and $g$, we denote
\begin{equation*}
  \int_{\partial\gamma}f:=f(\gamma(1))-f(\gamma(0)),\qquad
  \int_{\gamma}g:=\int_0^1 g(\gamma(t))\abs{\gamma'}(t)\,dt
\end{equation*}
whenever the right-hand sides are well defined.

\begin{lemma}\label{lem:ACDM4.13}
Let $p\in[1,\infty]$ and $0\leq g\in L^p(\mu)$, and $\pi$ be a $p'$-test plan for
$(X,\rho,\mu)$. Then
\begin{equation*}
\begin{split}
&\int_{C([0,1];X)}\Big(\int_{\gamma}g\Big)\,d\pi(\gamma) \\
&\quad\leq
C(\pi)^{\frac1p}\Norm{g}{L^p(\mu)}\Big[\int_{C([0,1];X)}\int_0^1\abs{\gamma'}^{p'}(t)
\,dt\,d\pi(\gamma)\Big]^{\frac{1}{p'}}
<\infty
\end{split}
\end{equation*}
with the usual modification made when $p'=\infty$.
\end{lemma}

\begin{proof}
The first estimate is a rather direct concatenation of definitions and H\"older's
inequality, as detailed in \cite[(4.13)--(4.14)]{ACDM}. The finiteness follows from
Definition~\ref{def:fzplan}\eqref{it:doubleInt} for $p'$-test plans.
This completes the proof of Lemma \ref{lem:ACDM4.13}.
\end{proof}

\begin{lemma}\label{lem:int-gamma-finite}
If $p\in[1,\infty]$ and $0\leq g\in L^p(\mu)$, then
\begin{equation*}
  \int_{\gamma}g<\infty\quad\text{for $p$-a.e.\ curve }\gamma.
\end{equation*}
\end{lemma}

This is mentioned in \cite[page 981]{AGS:13}. We give a proof for completeness.

\begin{proof}
Let $\Gamma:=\{\gamma\in C([0,1];X):\int_\gamma g=\infty\}$. If $\pi$ is a $p'$-test plan,
then Lemma \ref{lem:ACDM4.13} shows that $\pi(\Gamma)=0$. Since this holds for every
$p'$-test plan $\pi$, it means that
$\Gamma$ is $p$-negligible, and hence $\int_{\gamma}g<\infty$ for $p$-a.e.~$\gamma$ by
definition.
This completes the proof of Lemma \ref{lem:int-gamma-finite}.
\end{proof}

Another way of measuring the size of a curve family is via the modulus:

\begin{definition}
Let $p\in[1,\infty)$.
The {\em $p$-modulus} of a curve family $\Gamma\subset AC([0,1];X)$ is
\begin{equation*}
  \operatorname{Mod}_p(\Gamma)
  :=\inf\Big\{\Norm{h}{L^p(\mu)}^p:\int_\gamma h\geq 1\text{ for all }\gamma\in\Gamma\Big\}.
\end{equation*}
A set $\Gamma\subset AC([0,1];X)$ is said to be
\emph{$\operatorname{Mod}_p$-negligible} if $\operatorname{Mod}_p(\Gamma)=0$.
A property is said to \emph{hold for $\operatorname{Mod}_p$-a.e.\ curve} if it
holds for all $\gamma\in AC([0,1];X)\setminus\Gamma$ for some
$\operatorname{Mod}_p$-negligible set $\Gamma$.
\end{definition}

The $\operatorname{Mod}_p$ analogue of Lemma \ref{lem:negligible} is easy:

\begin{lemma}\label{lem:mod-negli}
Let $(X,\rho,\mu)$ be a metric measure space and $Y\subset X$ a closed subset.
Let $\Gamma\subset AC([0,1];Y)\subset AC([0,1];X)$.
Then $\Gamma$ is $\operatorname{Mod}_p$-negligible as a subset of $AC([0,1];X)$ if and only
if it is $\operatorname{Mod}_p$-negligible as a subset of $AC([0,1];Y)$.
\end{lemma}

\begin{proof}
We use the notation $\operatorname{Mod}_p(\Gamma;Z)$ for
the $p$-modulus of $\Gamma$ as a subset of $AC([0,1];Z)$,
where $Z\in\{X,Y\}$.

Let first $h\in L^p(\mu|_Y)$ be such that $\int_\gamma h\geq 1$ for all $\gamma\in\Gamma$.
If $\hat h\in L^p(\mu)$ is the trivial extension of $h$, i.e., $\hat h=h$ on $Y$ and $\hat
h=0$ on $X\setminus Y$,
then $\Norm{\hat h}{L^p(\mu)}=\Norm{h}{L^p(\mu|_Y)}$ and $\int_\gamma \hat h=\int_\gamma
h\geq 1$ for all $\gamma\in\Gamma$, since $\gamma([0,1])\subset Y$ and $h=\hat h$ on $Y$.
Taking the infimum over all such $h$, we find that
$$\operatorname{Mod}_p(\Gamma;X)\leq\operatorname{Mod}_p(\Gamma;Y).$$

Let then $h\in L^p(\mu)$ be such that $\int_\gamma h\geq 1$ for all $\gamma\in\Gamma$.
Taking $\tilde h:=h|_Y\in L^p(\mu|_Y)$ it follows that $\Norm{\tilde
h}{L^p(\mu|_Y)}\leq\Norm{h}{L^p(\mu)}$ and $\int_\gamma \tilde h=\int_\gamma h\geq 1$ for
all $\gamma\in\Gamma$. Taking the infimum over all such $h$, we find that
$$\operatorname{Mod}_p(\Gamma;Y)\leq\operatorname{Mod}_p(\Gamma;X).$$

Hence $\operatorname{Mod}_p(\Gamma;X)=\operatorname{Mod}_p(\Gamma;Y)$, and in particular the
two moduli are both zero or both non-zero simultaneously. This completes the proof.
\end{proof}

\begin{lemma}\label{lem:subadditive}
Let $(X,\rho,\mu)$ be a metric measure space, and let
\begin{equation*}
   \Gamma_n\subset AC([0,1];X),\qquad
   \Gamma:=\bigcup_{n=1}^\infty\Gamma_n.
\end{equation*}
If every $\Gamma_n$ is $p$-negligible (respectively, $\operatorname{Mod}_p$-negligible),
then so is $\Gamma$.
\end{lemma}

\begin{proof}
Suppose first that each $\Gamma_n$ is $p$-negligible. Hence, for every $p'$-test plan for
$X$, it has measure $\pi(\Gamma_n)=0$.  By $\sigma$-additivity (recalling that $p'$-test
plans are in particular probability measures), it follows that
$\Gamma:=\bigcup_{n=1}^\infty\Gamma_n\subset C([0,1];X)$ satisfies $\pi(\Gamma)=0$. Valid
for every $p'$-test plan, this shows that $\Gamma$ is $p$-negligible.

Suppose then that each $\Gamma_n$ is $\operatorname{Mod}_p$-negligible, i.e.,
$\operatorname{Mod}_p(\Gamma_n)=0$. It follows from the subadditivity of
$\operatorname{Mod}_p$, namely \cite[(5.2.6)]{HKST:book}
\begin{equation*}
  \operatorname{Mod}_p\Big(\bigcup_{n=1}^\infty\Gamma_n\Big)
  \leq\sum_{n=1}^\infty\operatorname{Mod}_p(\Gamma_n),
\end{equation*}
that $\operatorname{Mod}_p(\Gamma)=0$ as well. This completes the proof.
\end{proof}

\subsection{Upper gradients}\label{ss:ug}

The notions of negligible curve families from Section \ref{ss:curves} give rise to different
notions of upper gradients as follows:

\begin{definition}\label{def:uppergrad}
Let $(X,\rho,\mu)$ be a metric measure space and $p\in[1,\infty)$.
Let $f:X\to\mathbb R$ and $g,h:X\to[0,\infty]$.
Suppose that
\begin{equation}\label{eq:uppergrad}
  \Babs{\int_{\partial\gamma}f}\leq\int_\gamma g
\end{equation}
holds for all $\gamma\in AC([0,1];X)\setminus\Gamma$, for some $\Gamma\subset C([0,1];X)$.
We say that $g$ is
\begin{enumerate}[\rm(i)]
  \item an {\em upper gradient of $f$} if $\Gamma=\varnothing$;
  \item a {\em $p$-weak upper gradient of $f$ in the sense of modulus} if $\Gamma$ is
  $\operatorname{Mod}_p$-negligible;
  \item a {\em $p$-weak upper gradient of $f$ in the sense of test plans} if $\Gamma$ is
  $p$-negligible.
\end{enumerate}
If one of these properties holds for all $\gamma\in AC([0,1];X)\setminus\Gamma$ of length
$\ell(\gamma)>\delta$, then we say that $g$ is a ($p$-weak) upper gradient of $f$ (in one or
the other sense) {\em up to scale $\delta$}.

Suppose that
\begin{equation*}
   \abs{f(x)-f(y)}\leq \rho(x,y)[h(x)+h(y)]
\end{equation*}
for $\mu$-a.e.\ $x,y\in X$. Then $h$ is called a {\em Haj\l{}asz upper gradient of $f$}.
\end{definition}

\begin{remark}
The notion of upper gradients as above was introduced in \cite{HK:96} under the name
``very weak gradient''; according to \cite[p.~55 fn.]{Heinonen:book},
the name ``upper gradient'' was suggested by John Garnett, and this is now widely adopted.
The Haj\l{}asz upper gradient was (implicitly) introduced in \cite{Haj:96},
the $p$-weak upper gradient in the sense of modulus in \cite[(7)]{KM:98}, and the one in the
sense of test plans in \cite[Definition 4.8]{AGS:13}. Both latter ones are simply called
``$p$-weak upper gradients'' in these sources and other works of the respective authors and
their collaborators; notably, the terminology of \cite{KM:98} is also used in the monograph
\cite[page 152]{HKST:book}. On the other hand, \cite{AGS:13} refers to $p$-weak upper
gradients in the sense of modulus briefly as $p$-upper gradients. Since we will make use of
both definitions, we adopt the somewhat pedantic terminology of Definition
\ref{def:uppergrad} to avoid possible confusions.

The notion of {\em up to scale $\delta$} was introduced for $p$-weak upper gradients
in the sense of test plans  in \cite[Definition 26]{ACDM};
the same definition naturally extends to other versions of ($p$-weak or not) upper
gradients.
The version for plain upper gradients appears in \cite[Definition 2.1]{DMS}.
\end{remark}

\begin{convention}\label{con:uppergrad}
When making a statement about $p$-weak upper gradients without specifying the sense, it is
understood that such a statement applies to both $p$-weak upper gradients in the sense of
modulus, and those in the sense of test plans; however, the same sense must be applied
consistently throughout the statement.
\end{convention}

\begin{lemma}\label{lem:minimal-wug}
Let $(X,\rho,\mu)$ be separable and locally finite.
Using Convention \ref{con:uppergrad}, let $p\in[1,\infty)$
and $f:X\to\mathbb R$ be $\mu$-measurable with a $p$-weak upper gradient $g\in L^p(\mu)$.
\begin{enumerate}[\rm(i)]
  \item\label{it:ae-wug} If $\tilde g=g$ $\mu$-a.e., then $\tilde g$ is also a $p$-weak
  upper gradient of $f$.
  \item\label{it:min-wug} There is a minimal $p$-weak upper gradient $g_0\in L^p(\mu)$ of
  $f$, i.e., whenever $g_1\in L^p(\mu)$ is a $p$-weak upper gradient of $f$, then $g_0\leq
  g_1$ pointwise a.e.
\end{enumerate}
\end{lemma}

\begin{proof}
In the modulus case, claim \eqref{it:ae-wug}
is \cite[Lemma 6.2.8]{HKST:book}
(where assumptions on $X$ are stated in the beginning of \cite[Section 6.2]{HKST:book}),
while claim \eqref{it:min-wug} is contained in \cite[Theorem 6.3.20 and Remark
6.3.21]{HKST:book}
(which refers to the notion of minimal $p$-weak upper gradient defined on the bottom of
\cite[page 161]{HKST:book}).

We turn to the test plans case, based on results of
\cite{AGS:13} for $p\in(1,\infty)$ and \cite{CKR:24} for $p=1$. Since some of these use the
$\sigma$-finiteness of $(X,\mu)$,
we note that this follows from the assumptions of the
present lemma by \cite[Lemma 3.3.28]{HKST:book}.
If $p\in(1,\infty)$, then, by Lemma \ref{lem:int-gamma-finite}, the assumption $g\in
L^p(\mu)$ implies that
$\int_{\gamma}g<\infty$ for $p$-a.e.\ curve $\gamma$. When $f$ has a $p$-weak upper gradient
with this property, \cite[Remark 4.10]{AGS:13} shows that $f$ is {\em Sobolev along
$p$-a.e.\ curve}
in the sense of \cite[Definition 4.9]{AGS:13}, i.e., for $p$-a.e.\ $\gamma$, the function
$f\circ\gamma$ coincides at the end-points $\{0,1\}$ and a.e.\ on the interval $(0,1)$ with
an absolutely continuous function. For such functions (namely, Sobolev along $p$-a.e.\
curve),
the existence of a minimal $p$-weak upper gradient as in
\eqref{it:min-wug} is explained in
\cite[Definition 4.11 et seq.]{AGS:13}. The case $p=1$ of
\eqref{it:min-wug} is
proved in \cite[page 141]{CKR:24}.

To prove \eqref{it:ae-wug}, let $N\subset X$ be a Borel set such that
$\mu(N)=0$ and $g=\widetilde g$ on $X\setminus N$. Fix an arbitrary
$q$-test plan $\pi$ with $q=p'$. By the bounded compression property \eqref{it:C(pi)} of
$\pi$,
for every $t\in[0,1]$, we have
$$
 \pi\bigl(\{\gamma:\gamma(t)\in N\}\bigr)
 =(e_t)_\#\pi(N)
 \leq C(\pi)\mu(N)=0.
$$
Hence, by Fubini's theorem,
$$
 \int \mathcal L^1\bigl(\{t\in[0,1]:\gamma(t)\in N\}\bigr)
 \,d\pi(\gamma)
 =
 \int_0^1\pi\bigl(\{\gamma:\gamma(t)\in N\}\bigr)\,dt=0,
$$
where $\mathcal{L}^1$ denotes the Lebesgue measure on $\mathbb R$.
It follows that, for $\pi$-a.e.\ $\gamma$,
$$
 g(\gamma(t))=\widetilde g(\gamma(t))
 \quad\text{for a.e. }t\in[0,1],
$$
and consequently
$$
 \int_0^1 g(\gamma(t))|\gamma'(t)|\,dt
 =
 \int_0^1 \widetilde g(\gamma(t))|\gamma'(t)|\,dt.
$$
Thus, the inequality \eqref{eq:uppergrad} for $g$ also holds with
$\widetilde g$ in place of $g$. Since $\pi$ is arbitrary,
$\widetilde g$ is also a $p$-weak upper gradient of $f$.
This completes the proof of Lemma \ref{lem:minimal-wug}.
\end{proof}

\begin{remark}\label{rem:ug-relations}Let $p\in[1,\infty)$.
\begin{enumerate}[\rm(i)]
\item An upper gradient is obviously a $p$-weak upper gradient in either sense.
  \item\label{it:mod>testplans} A $p$-weak upper gradient in the sense of modulus is a
  $p$-weak upper gradient in the sense of test plans; see \cite[Remark 4.12]{AGS:13}.
\end{enumerate}
\end{remark}

As a partial converse to Remark \ref{rem:ug-relations}\eqref{it:mod>testplans}, we have the
following:

\begin{theorem}[\cite{AGS:13}, Theorem 7.4]\label{thm:AGS7.4}
Let $p\in(1,\infty)$ and $(X,\rho,\mu)$ be a complete, separable, and locally finite metric
measure space.
Then $f\in L^p(\mu)$ has a $p$-weak upper gradient in $L^p(\mu)$ in the sense of modulus if
and only if
it has a $p$-weak upper gradient in $L^p(\mu)$ in the sense of test plans. Moreover,
the two minimal $p$-weak upper gradients are equal $\mu$-a.e.
\end{theorem}

\begin{remark}
Theorem \ref{thm:AGS7.4} is proved in \cite[Section 7]{AGS:13} under the stronger
assumptions that $(X,\rho)$ is a compact metric space and $\mu$ is a finite Borel measure,
but an extension to the version formulated in Theorem \ref{thm:AGS7.4} is indicated in
\cite[Section 8.2]{AGS:13}.
 An equivalence with two further notions of weak gradients is also contained in
 \cite[Theorem 7.4]{AGS:13}, but we do not need them here.
\end{remark}

\begin{corollary}\label{cor:AGS7.4}
Theorem \ref{thm:AGS7.4} remains valid for $f\in L^1(\mu)$.
\end{corollary}

\begin{proof}
Suppose that $f\in L^1(\mu)$ has a $p$-weak upper gradient $g\in L^p(\mu)$ in the sense of
test plans, and hence a minimal such object, $\abs{\nabla f}_{w,p}\in L^p(\mu)$. We consider
the usual truncations
\begin{equation}\label{eq:fN-def}
  f_N(x):=\begin{cases} f(x) & \text{if }\abs{f(x)}\leq N, \\
  \displaystyle N\frac{f(x)}{\abs{f(x)}} & \text{else}.\end{cases}
\end{equation}
Then
\begin{equation}\label{eq:fN-key}
  \abs{f_N(x)-f_N(y)} \leq \abs{f(x)-f(y)}\quad\text{for all }x,y\in X,
\end{equation}
and hence
\begin{equation*}
  \Babs{\int_{\partial\gamma}f_N}\leq\Babs{\int_{\partial\gamma}f}\quad
  \text{for every curve }\gamma.
\end{equation*}
From this, it is immediate that $\abs{\nabla f}_{w,p}\in L^p(\mu)$ is also a $p$-weak upper
gradient, in the sense of test plans, of every $f_N$. But $f_N\in L^1(\mu)\cap
L^\infty(\mu)\subset L^p(\mu)$. Hence Theorem \ref{thm:AGS7.4} applies to show that $f_N$
also has a $p$-weak upper gradient in $L^p(\mu)$ in the sense of modulus. Let $\abs{\nabla
f_N}_{p},\abs{\nabla f_N}_{w,p}\in L^p(\mu)$ be the minimal $p$-weak upper gradients of
$f_N\in L^p(\mu)$ in the sense of modulus and test plans,
respectively. Then $\abs{\nabla f_N}_{p}=\abs{\nabla f_N}_{w,p}$ a.e.\ by Theorem
\ref{thm:AGS7.4}, and $\abs{\nabla f_N}_{w,p}\leq\abs{\nabla f}_{w,p}$ a.e.\ by the
minimality of $\abs{\nabla f_N}_{w,p}$, since $\abs{\nabla f}_{w,p}\in L^p(\mu)$ is also a
$p$-weak upper gradient of $f_N$ in the sense of test plans. But then $\abs{\nabla
f}_{w,p}\geq\abs{\nabla f_N}_{w,p}=\abs{\nabla f_N}_p$ is also a $p$-weak upper gradient of
$f_N$ in the sense of modulus, for every $N$. This means that
\begin{equation*}
  \Babs{\int_{\partial\gamma}f_N}\leq\int_\gamma\abs{\nabla f}_{w,p}\qquad
  \forall\ \gamma\in AC([0,1];X)\setminus\Gamma_N
\end{equation*}
for some $\Gamma_N$ with $\operatorname{Mod}_p(\Gamma_N)=0$. Then
$\Gamma:=\bigcup_{N=1}^\infty\Gamma_N$ also satisfies $\operatorname{Mod}_p(\Gamma)=0$ by
Lemma \ref{lem:subadditive}.

Let $\gamma\in AC([0,1];X)\setminus\Gamma\subset AC([0,1];X)\setminus\Gamma_N$ for every
$N$. Choose $N>\max\{\abs{f(\gamma(s))}:s=0,1\}$. Then
\begin{equation*}
  \Babs{\int_{\partial\gamma}f}
  =\Babs{\int_{\partial\gamma}f_N}
  \leq\int_\gamma\abs{\nabla f}_{w,p}.
\end{equation*}
Since this holds for every $\gamma\in AC([0,1];X)\setminus\Gamma$,
where $\operatorname{Mod}_p(\Gamma)=0$, it follows that
$\abs{\nabla f}_{w,p}\in L^p(\mu)$ is a $p$-weak upper gradient of $f$ in the sense of
modulus.
Hence the minimal such object satisfies $\abs{\nabla f}_p\leq \abs{\nabla f}_{w,p}$. The
converse is immediate from \cite[Remark 4.12]{AGS:13}: every $p$-weak upper gradient in the
sense of modulus is a $p$-weak upper gradient in the sense of test plans.
This completes the proof of Corollary \ref{cor:AGS7.4}.
\end{proof}

\begin{remark}
While \cite{AGS:13} mostly discusses $p$-weak upper gradients of functions in $L^p(\mu)$
with the matching integrability $f\in L^p(\mu)$, \cite[Remark 4.4]{AGS:13} points out that
``in principle the integrability of $f$ could be decoupled from the integrability of the
gradient''. Our Corollary \ref{cor:AGS7.4} makes this concrete in a particular example.
\end{remark}

The following theorem shows the role of upper gradients up to some positive scale in
constructing $p$-weak upper gradients without a scale restriction:

\begin{theorem}[\cite{ACDM}, Theorem 27]\label{thm:ACDM27}
Let $(X,\rho,\mu)$ be a complete and separable metric space with a nonnegative Borel measure
$\mu$.
Let $p\in(1,\infty)$. Let $f_n,f:X\to\mathbb R$ be measurable and  $g_n,g\in L^p(\mu)$ be
such that
\begin{enumerate}[\rm(i)]
  \item $g_n$ is an upper gradient of $f_n$ up to scale $\delta_n$, where $\delta_n\searrow
  0$; and
  \item $f_n\to f$ pointwise $\mu$-a.e., and $g_n\rightharpoonup g$ weakly in $L^p(\mu)$.
\end{enumerate}
Then $g$ is a $p$-weak upper gradient of $f$ in the sense of test plans.
\end{theorem}

\begin{remark}
The statement of \cite[Theorem 27]{ACDM} has no conditions on $(X,\rho,\mu)$, but the
standing assumptions for the whole \cite[Section 4]{ACDM} (which contains the said theorem)
are given at the beginning of the section; it is explicitly mentioned there that ``not even
$\sigma$-finiteness is needed for the results of this section.'' The formulation of
\cite[Theorem 27]{ACDM} is slightly more general, but this will not be needed here. A
variant is also given in \cite[Lemma 2.6]{DMS}, where a statement dealing with $p=1$ is
included.
\end{remark}

We also recall the following definition of Poincar\'e spaces.

\begin{definition}
Let $p\in[1,\infty)$ and let $(X,\rho,\mu)$ be a metric measure space.
The space $(X,\rho,\mu)$ is called a \emph{$p$-Poincar\'e space}
with a parameter $\lambda_0\in [1,\infty)$
if there exists a positive constant $C$ such that,
for any integrable functions $f$ on $X$, any upper gradient
$g$ of $f$, and any ball $B\subset X$,
\begin{align*}
\fint_B|f-\ave{f}_B| \, d\mu \le C r_B
\left(\fint_{\lambda_0 B} g^p \, d\mu\right)^{\frac1p},
\end{align*}
where $r_B$ denotes the radius of $B$.
\end{definition}

\subsection{Approximation by bounded subsets}\label{sec:bdApprox}

In several arguments further below, our final results on the full space $X$ will be achieved
by first considering the case of bounded subsets, followed by appropriate limits. In this
section, we collect some tools that facilitate this approximation process. It will be
convenient to use the following notion:

\begin{definition}\label{def:bdExhaustion}
Let $(X,\rho)$ be a metric space. We say that a sequence of subsets $X_n\subset X$ has the
{\em bounded exhaustion property} if each bounded subset $B\subset X$ satisfies $B\subset
X_n$ for all large enough $n$.
\end{definition}

For any nonempty subset $E\subset X$, we define $
\operatorname{diam}(E):=\sup\{\rho(x,y):x,y\in E\}$.
As some of our arguments will need a bounded set, it will be convenient to have the
following approximation result:

\begin{proposition}\label{prop:bdSubsets}
For each doubling metric measure space $(X,\rho,\mu)$, there exists an increasing sequence
of closed bounded subsets $X_n\subset X$ such that
\begin{enumerate}[\rm(i)]
  \item\label{it:XnDb} each $(X_n,\rho|_{X_n},\mu|_{X_n})$ is a doubling metric measure
  space;
  \item\label{it:bdExhaustion} the sequence $(X_n)_{n=1}^\infty$ has the bounded exhaustion
  property (Definition \ref{def:bdExhaustion});
  \item\label{it:Vn} the measure of balls in each $X_n$ satisfies
\begin{equation}\label{eq:bdSubsets}
  \mu(B(x,r)\cap X_n)\sim V(x,\min\{r,\operatorname{diam}(X_n)\}),\quad \forall\ x\in X_n
  \ \text{and}\ r\in(0,\infty);
\end{equation}
\item if $X$ is a complete metric space, then all $X_n$ are complete as well.
\end{enumerate}
Moreover, the implicit constants in properties \eqref{it:XnDb} and \eqref{it:Vn}
of Proposition~\ref{prop:bdSubsets} are uniform in $n$ and only depend on the doubling
constant of $X$.
\end{proposition}

\begin{remark}
If $X$ is also a $p$-Poincar\'e space, one may ask whether the subsets $X_n$ can be chosen
(uniformly in $n$) to be $p$-Poincar\'e spaces as well. A certain converse is true: if such
a sequence of subsets $X_n$ happen to be $p$-Poincar\'e spaces uniformly in $n$, then $X$
itself is a $p$-Poincar\'e space; see \cite[Theorem 3]{K:03}.

We do not know the answer to this first-mentioned question, but we also do not need it for
the present purposes. Although our main results depend on the Poincar\'e inequality, we will
only use this approximation by bounded subsets in a part of the argument, where it does not
play a role.
\end{remark}

\begin{proof}[Proof of Proposition \ref{prop:bdSubsets}]
We will make use of systems of dyadic cubes from \cite{HK:cubes}, elaborating on earlier
construction by \cite{Christ:cubes}. For every $k\in\mathbb Z$, the space $X$ has a
measurable partition $\mathscr D_k$, where $\mathscr D_{k+1}$ is a refinement of $\mathscr
D_k$, and each $Q^k_\alpha\in\mathscr D_k$ has a centre point $x^k_\alpha$
 such that \cite[(2.8) and (2.15)]{HK:cubes}
\begin{equation}
  B(x^k_\alpha,c_1\delta^k)\subset Q^k_\alpha
  \subset\overline{Q^k_\alpha}\subset B(x^k_\alpha,C_1\delta^k)
\end{equation}
for some constants $0<c_1\leq C_1<\infty$ and $\delta\in(0,1)$. Specifically, we will use
the variant in \cite[Proposition 4.3]{HK:cubes}, where a preassigned point $x_0\in X$ has
the property that, for every level $k\in\mathbb Z$, this $x_0$ is the centre point of some
cube $Q^k_0\in\mathscr D_k$.

We will then choose $X_n:=\overline{Q^{-n}_0}$, the corresponding closed cube. If $X$ is
complete, then each $X_n$ is also complete as a closed subset of a complete space. For the
bounded exhaustion property in Proposition~\ref{prop:bdSubsets}\eqref{it:bdExhaustion}, it
suffices to consider $B=B(x_0,r)$, since every bounded set is contained in such a ball for
some large enough $r>0$. Then it is immediate that $B\subset B(x_0,c_1\delta^{-n})\subset
X_n$ as soon as $r\leq c_1\delta^{-n}$, which clearly holds for all large enough $n$.

Let us note that the doubling property of $X_n$ follows, once we prove \eqref{eq:bdSubsets}.
Indeed,
\begin{equation*}
\begin{split}
  \mu(B(x,2r)\cap X_n)
  &\sim V(x,\min\{2r,\operatorname{diam}(X_n)\})\quad\text{by \eqref{eq:bdSubsets}} \\
  &\lesssim V(x,\min\{r,\operatorname{diam}(X_n)\})\quad\text{by doubling in }X \\
  &\sim\mu(B(x,r)\cap X_n)\quad\text{by \eqref{eq:bdSubsets}}.
\end{split}
\end{equation*}

Let us then prove \eqref{eq:bdSubsets}. Since $B(x,r)\cap X_n\subset
B(x,\min\{r,\operatorname{diam}(X_n)\})$, the upper bound $\lesssim$ is evident. For the
lower bound, suppose first that $r\geq 2C_1\delta^{-n}\geq\operatorname{diam}(X_n)$. Then
\begin{equation*}
  B(x_0,c_1\delta^{-n})\subset X_n\subset B(x,\operatorname{diam}(X_n))\subset B(x,r)
\end{equation*}
and hence
\begin{equation*}
\begin{split}
  \mu(B(x,r)\cap X_n)=\mu(X_n)
  &\geq V(x_0,c_1\delta^{-n})
  \gtrsim V(x_0, 2C_1\delta^{-n}) \\
  &\geq V(x_0,\operatorname{diam}(X_n))\sim V(x,\operatorname{diam}(X_n)),
\end{split}
\end{equation*}
which is the required lower bound.

Suppose then that $0<r<2C_1\delta^{-n}$. Then we can find a unique $k>-n$ such that
$2C_1\delta^k\leq r<2C_1\delta^{k-1}$.
Since $X_n=\overline{Q^{-n}_0}$ is the union of some closed cubes of level $k$ (by
\cite[(2.17)]{HK:cubes}), we have $x\in\overline{Q^k_\alpha}$ for some $\alpha$. Then
\begin{equation*}
   B(x^k_\alpha,c_1\delta^k)\subset
   Q^k_\alpha\subset\overline{Q^k_\alpha}\subset\overline{Q^{-n}_0}=X_n
\end{equation*}
and
\begin{equation*}
\begin{split}
  Q^k_\alpha\subset B(x^k_\alpha,C_1\delta^k)
  \subset B(x,2C_1\delta^k)\subset B(x,r)
  &\subset B(x^k_\alpha,r+C_1\delta^k) \\
  &\subset B(x^k_\alpha,C_1(2\delta^{-1}+1)\delta^k).
\end{split}
\end{equation*}
Thus,
\begin{equation*}
  B(x^k_\alpha,c_1\delta^k)\subset Q^k_\alpha\subset B(x,r)\cap X_n,
\end{equation*}
and
\begin{equation*}
  V(x,r)
  \leq V(x^k_\alpha,C_1(2\delta^{-1}+1)\delta^k)
  \lesssim V(x^k_\alpha,c_1\delta^k)
  \leq \mu(B(x,r)\cap X_n).
\end{equation*}
This is the required lower bound in \eqref{eq:bdSubsets} in the remaining case.
This completes the proof.
\end{proof}

The following proposition shows how to build a global $p$-weak upper gradient from local
versions:

\begin{proposition}\label{prop:gluing-ug}
Let $(X,\rho,\mu)$ be a separable and locally finite metric measure space and
$p\in(1,\infty)$. Let $X_n$ be an
increasing sequence of closed subsets of $X$ with the bounded exhaustion property
(Definition \ref{def:bdExhaustion}). Using Convention \ref{con:uppergrad}, let
$f:X\to\mathbb R$ be a function such that each restriction $f^n:=f|_{X_n}$ has a $p$-weak
upper gradient $g^n\in L^p(\mu_n)$, where $\mu_n:=\mu|_{X_n}$. Then $f$ has a $p$-weak upper
gradient $g$ such that
\begin{equation*}
  \Norm{g}{L^p(\mu)}\leq\sup_{n\in\mathbb N}\Norm{g^n}{L^p(\mu_n)}.
\end{equation*}
\end{proposition}

\begin{proof}
Both separability and local finiteness of $(X,\rho,\mu)$ are inherited by the subspaces
$(X_n,\rho|_{X_n},\mu|_{X_n})$. We may then replace each $g^n\in L^p(\mu_n)$ by the
minimal $p$-weak upper gradient of $f^n$ guaranteed by Lemma \ref{lem:minimal-wug} under
these assumptions on $(X_n,\rho|_{X_n},\mu|_{X_n})$.

We will argue that $g^{n}|_{X_{n-1}}$ is a $p$-weak upper gradient of $f^{n-1}$ when $n>1$.
Indeed, since $g^{n}$ is a $p$-weak upper gradient of $f^{n}=f|_{X_n}$, we deduce that
\begin{equation}\label{eq:ug-Xn}
  \Babs{\int_{\partial\gamma}f}=\Babs{\int_{\partial\gamma}f^{n}}
  \leq\int_{\gamma}g^{n}\qquad\forall\,\gamma\in AC([0,1];X_{n})\setminus\Gamma_{n},
\end{equation}
for some $p$-negligible (resp.  $\operatorname{Mod}_p$-negligible) set $\Gamma_{n}$ for
$X_{n}$. Then the smaller set
\begin{equation*}
  \Gamma_{n}':=\Gamma_{n}\cap AC([0,1];X_{n-1})\subset\Gamma_{n}
\end{equation*}
is certainly $p$-negligible (resp. $\operatorname{Mod}_p$-negligible) for $X_{n}$ as well,
and therefore $p$-negligible (resp. $\operatorname{Mod}_p$-negligible) for $X_{n-1}$ by
Lemma \ref{lem:negligible} (resp. Lemma \ref{lem:mod-negli}). Now
\begin{equation*}
  AC([0,1];X_{n-1})\setminus\Gamma_{n}'
  \subset AC([0,1];X_{n})\setminus\Gamma_{n},
\end{equation*}
so \eqref{eq:ug-Xn} holds, in particular, for all $\gamma$ in this family. Since
$f^{n}=f=f^{n-1}$ on $\gamma([0,1])\subset X_{n-1}$ it follows that
\begin{equation*}
  \Babs{\int_{\partial\gamma}f^{n-1}}
  \leq\int_{\gamma}g^{n}
  =\int_{\gamma}g^{n}|_{X_{n-1}}\qquad\forall\,\gamma\in AC([0,1];X_
  {n-1})\setminus\Gamma_{n}',
\end{equation*}
where $\Gamma_{n}'$ is $p$-negligible (resp. $\operatorname{Mod}_p$-negligible) for
$X_{n-1}$. This shows that $g^{n}|_{X_{n-1}}$ is also a $p$-weak upper gradient of
$f^{n-1}$, as claimed.

Since $g^{n-1}$ is the minimal $p$-weak upper gradient of $f^{n-1}$, it follows that
\begin{equation*}
  g^{n-1}\leq g^{n}\quad\text{for a.e. }x\in X_{n-1}.
\end{equation*}
Hence, the sequence of functions $\mathbf{1}_{X_n}g^n\geq 0$ is monotone increasing at a.e.\
$x\in X$. We can then define
\begin{equation}\label{eq:g=sup}
\begin{split}
  g(x)&:=\sup_{n\in\mathbb N}\mathbf{1}_{X_n}(x)g^n(x) \\
  &\phantom{:}=\lim_{n\to\infty}\mathbf{1}_{X_n}(x)g^n(x)\quad\text{at a.e. }x\in X.
\end{split}
\end{equation}
The monotone convergence theorem implies that
\begin{equation*}
  \Norm{g}{L^p(\mu)}^p
  =\lim_{n\to\infty}\int_{X_n}(g^n)^p\,d\mu
  =\lim_{n\to\infty}\Norm{g^n}{L^p(\mu_n)}^p
  =\sup_{n\in\mathbb N}\Norm{g^n}{L^p(\mu_n)}^p.
\end{equation*}

It remains to check that $g$ is a $p$-weak upper gradient of $f$. For each $n$, let again
$\Gamma_n$ be a $p$-negligible (resp. $\operatorname{Mod}_p$-negligible) set in
$C([0,1];X_n)$ such that \eqref{eq:ug-Xn} holds. By Lemma \ref{lem:negligible} (resp. Lemma
\ref{lem:mod-negli}), $\Gamma_n$ is also $p$-negligible (resp.
$\operatorname{Mod}_p$-negligible) in $C([0,1];X)$. Lemma \ref{lem:subadditive} then implies
that $\Gamma:=\bigcup_{n=1}^\infty\Gamma_n\subset C([0,1];X)$ is also $p$-negligible (resp.
$\operatorname{Mod}_p$-negligible).

Now, let $\gamma\in AC([0,1];X)\setminus\Gamma$ be given. Since the image $\gamma([0,1])$ is
a bounded set (indeed, of diameter at most $\ell(\gamma)$), it will be contained in $X_n$
for all large enough $n$ by the bounded exhaustion property. Thus,
\begin{equation*}
  \gamma\in AC([0,1];X_n)\setminus\Gamma
  \subset AC([0,1];X_n)\setminus\Gamma_n.
\end{equation*}
Then \eqref{eq:ug-Xn} and \eqref{eq:g=sup} imply that
\begin{equation*}
  \Babs{\int_{\partial\gamma}f}
  \leq\int_{\gamma}g^n
  =\int_{\gamma}\mathbf{1}_{X_n}g^n
  \leq\int_{\gamma}g.
\end{equation*}
Since $\gamma\in AC([0,1];X)\setminus\Gamma$ was arbitrary, and $\Gamma$ is $p$-negligible
(resp.  $\operatorname{Mod}_p$-negligible), this shows that $g$ is a $p$-weak upper gradient
of $f$, and completes the proof.
\end{proof}

We will next extend Theorem \ref{thm:AGS7.4} and Corollary \ref{cor:AGS7.4}
to the more general function class of locally integrable functions,
for which we adopt the following definition:
\begin{equation}\label{eLoc}
  L^1_{\mathrm{loc}}(\mu):=\Big\{h:X\to\mathbb R\text{ measurable}: \mathbf{1}_B h\in
  L^1(\mu)\text{ for all balls }B\subset X\Big\}.
\end{equation}

\begin{remark}
Our definition \eqref{eLoc} of local integrability agrees with that of
\cite[p.~1905]{FHK:99} and \cite[p.~6]{HK:00}, but differs from the convention used e.g.\ in
\cite[p.~5]{Heinonen:book} and \cite[p.~46]{HKST:book}, where  ``all balls $B\subset X$'' is
replaced by ``some ball $B=B(x,r_x)$ at every $x\in X$''. Definition \eqref{eLoc} as
written is the more useful version for us.
\end{remark}

\begin{corollary}\label{cor:AGS7.4loc}
Theorem \ref{thm:AGS7.4} remains valid for $f\in L^1_{\mathrm{loc}}(\mu)$.
\end{corollary}

\begin{proof}
Our assumptions on the underlying space $(X,\rho,\mu)$ are those of Theorem
\ref{thm:AGS7.4}, i.e., it is complete, separable, and locally finite.
Suppose that $f\in L^1_{\mathrm{loc}}(\mu)$ has a $p$-weak upper gradient in $L^p(\mu)$ in
the sense of test plans, and hence a minimal such object, $\abs{\nabla f}_{w,p}\in
L^p(\mu)$.

We choose an increasing sequence of closed bounded subsets $X_n\subset X$ with the bounded
exhaustion property (Definition \ref{def:bdExhaustion}),
 e.g., $X_n=\overline{B(x_0,n)}$. With the induced metric $\rho_n:=\rho|_{X_n}$ and measure
 $\mu_n:=\mu|_{X_n}$, the triple $(X_n,\rho_n,\mu_n)$ is again a complete, separable, and
 locally finite metric measure space. Thus, each $(X_n,\rho_n,\mu_n)$ satisfies the same
 assumptions as $(X,\rho,\mu)$ in Theorem \ref{thm:AGS7.4}. Since $X_n$ is bounded, the
 assumption $f\in L^1_{\mathrm{loc}}(\mu)$ implies that $f^n:=f|_{X_n}\in L^1(\mu_n)$. It is
 evident that $\abs{\nabla f}_{w,p}|_{X_n}\in L^p(\mu_n)$ is a $p$-weak upper gradient of
 $f^n$ in the sense of test plans. Now, Corollary \ref{cor:AGS7.4} applies to show that
 $f^n\in L^1(\mu_n)$ also has a $p$-weak upper gradient in $L^p(\mu_n)$ in the sense of
 modulus, and hence a minimal such object, $\abs{\nabla f^n}_p\in L^p(\mu_n)$. Moreover, if
 $\abs{\nabla f^n}_{w,p}\in L^p(\mu_n)$ is the minimal $p$-weak upper gradient in the sense
 of test plans, these two minimal objects agree almost everywhere, and the minimality of the
 latter implies that
\begin{equation*}
  \abs{\nabla f^n}_{p}=\abs{\nabla f^n}_{w,p}\leq\abs{\nabla f}_{w,p}|_{X_n}.
\end{equation*}
Summarising, each restriction $f^n=f|_{X_n}$ has a $p$-weak upper gradient in the sense of
modulus, $\abs{\nabla f^n}_{p}\in L^p(\mu_n)$. This is precisely the setting of Proposition
\ref{prop:gluing-ug}, which then implies that the original $f$ has a $p$-weak upper gradient
in the sense of modulus, $g$, with the norm estimate
\begin{equation*}
\begin{split}
  \Norm{g}{L^p(\mu)}
  &\leq\sup_{n\in\mathbb N}\Norm{\abs{\nabla f^n}_p}{L^p(\mu_n)} \\
  &\leq\sup_{n\in\mathbb N}\Norm{\abs{\nabla f}_{w,p}}{L^p(\mu_n)}
  =\Norm{\abs{\nabla f}_{w,p}}{L^p(\mu)}<\infty.
\end{split}
\end{equation*}
Hence, there is a minimal such object, $\abs{\nabla f}_p\in L^p(\mu)$ whose minimality
ensures that
\begin{equation*}
 \Norm{\abs{\nabla f}_p}{L^p(\mu)}
  \leq\Norm{g}{L^p(\mu)}
  \leq\Norm{\abs{\nabla f}_{w,p}}{L^p(\mu)}.
\end{equation*}
This completes the proof of Corollary \ref{cor:AGS7.4loc}.
\end{proof}

\subsection{Sobolev spaces}\label{ss:Sobolev}

In this section, we define the relevant notions of homogeneous Sobolev spaces and norms
featuring in our main results, and collect some of their basic properties.

\begin{definition}\label{def:dotW1p}
Let $p\in[1,\infty)$.
We say that $f\in L^1_{\mathrm{loc}}(\mu)$ belongs to the
\emph{homogeneous Sobolev space
$\dot W^{1,p}(\mu)$} if $f$ has a $p$-weak upper gradient $g\in L^p(\mu)$
in the sense of test plans. In this case, the minimal $p$-weak upper gradient
is denoted by $\abs{\nabla f}_{w,p}$. The homogeneous Sobolev seminorm is defined by
\begin{equation*}
   \Norm{f}{\dot W^{1,p}(\mu)}:=\Norm{\abs{\nabla f}_{w,p}}{L^p(\mu)}.
\end{equation*}
We let $\Norm{f}{\dot W^{1,p}(\mu)}:=\infty$ if $f$ has no $p$-weak upper gradient in
$L^p(\mu)$.
\end{definition}

\begin{remark}
The minimality in Definition \ref{def:dotW1p} is understood with respect to
the $\mu$-a.e. pointwise order. The existence and the uniqueness up to $\mu$-a.e.
equality of the minimal $p$-weak upper gradient follow from Lemma
\ref{lem:minimal-wug}.
\end{remark}

\begin{remark}
Definition \ref{def:dotW1p} is an adaptation of \cite[Definition 2.4 and (2.2)]{DMS}, where
the inhomogeneous space $W^{1,p}(\mu)$ is similarly defined, assuming that $f\in L^p(\mu)$.
Note that \cite[Definition 2.3]{DMS} only defines $p$-weak upper gradients for functions
$f\in L^p(\mu)$, but \cite[Definition 4.8]{AGS:13} allows more general $f$. In place of
$\Norm{f}{\dot W^{1,p}(\mu)}^p$, the notation $\operatorname{Ch}_p(f)$ and the name {\em
Cheeger energy} is used in \cite[(2.2)]{DMS}.
\end{remark}

Let $p\in[1,\infty]$.
Recalling the notion of Haj\l{}asz upper gradients from Definition \ref{def:uppergrad}, for
measurable $f:X\to\mathbb R$ we define
\begin{equation*}
  \Norm{f}{\dot M^{1,p}(\mu)}
  :=\inf\{ \Norm{h}{L^p(\mu)}:h\text{ is a Haj\l{}asz upper gradient of }f\}\in[0,\infty]
\end{equation*}
and
\begin{equation*}
  \dot M^{1,p}(\mu):=\{f\text{ measurable}: \Norm{f}{\dot M^{1,p}(\mu)}<\infty\}.
\end{equation*}

\begin{remark}\label{rem:dotM-basic}
Let $p\in[1,\infty]$ and $f\in\dot M^{1,p}(\mu)$. Then
\begin{enumerate}[\rm(i)]
  \item\label{it:hinLp} $f$ has a Haj\l{}asz upper gradient $h\in L^p(\mu)$;
  \item\label{it:finLploc} $f\in L^p_{\mathrm{loc}}(\mu)$.
\end{enumerate}
\end{remark}

\begin{proof}
For \eqref{it:hinLp}, if the stated conclusion of \eqref{it:hinLp} failed, then the infimum in the definition of
$\Norm{f}{\dot M^{1,p}(\mu)}$ would be taken over a set that only consists of $\infty$, in
which case this infimum would also be $\infty$.

For \eqref{it:finLploc}, let us fix a Haj\l{}asz upper gradient $h\in L^p(\mu)$ and $y_0\in
X$ with $\abs{h(y_0)}<\infty$ such that $\abs{f(x)-f(y_0)}\leq\rho(x,y_0)[h(x)+h(y_0)]$ at
almost every $x\in X$. Then
$$\abs{f}\leq\abs{f(y_0)}+\rho(\cdot,y_0)\big[h(\cdot)+h(y_0)\big]\in
L^p_{\mathrm{loc}}(\mu),$$
since $h\in L^p(\mu)$, and $x\mapsto\rho(x,y_0)$ is locally
bounded.
This completes the proof of both assertions.
\end{proof}

It will be useful to know that the two kinds of Sobolev spaces coincide over a wide class of
metric measure spaces $(X,\rho,\mu)$. This is essentially known for the inhomogeneous spaces
$W^{1,p}(\mu):=L^p(\mu)\cap\dot W^{1,p}(\mu)$ and $M^{1,p}(\mu):=L^p(\mu)\cap\dot
M^{1,p}(\mu)$; see \cite[Theorem 7.4]{AGS:13} for the identification of several variants of
$W^{1,p}(\mu)$, including the one commonly denoted by $N^{1,p}(\mu)$, and \cite[Corollary
10.2.9]{HKST:book} for $N^{1,p}(\mu)=M^{1,p}(\mu)$. For completeness, we give the details
for the homogeneous spaces of relevance to us:

\begin{proposition}\label{prop:MWE}
Let $(X,\rho,\mu)$ be a complete doubling
$p$-Poincar\'e space with $p\in(1,\infty)$.
Then $\dot{W}^{1,p}(\mu)=\dot{M}^{1,p}(\mu)$ with equivalent seminorms.
\end{proposition}
\begin{proof}
Let $f\in L^1_{\rm loc}(\mu)$. If $h\in L^p(\mu)$ is a Haj\l asz
upper gradient of $f$, then \cite[Theorem 1.3]{JSYY:15} shows, after
changing $f$ and $h$ on null sets, that $4h$ is an upper gradient of~$f$.
In particular, it is a $p$-weak upper gradient in the sense of
test plans. Hence
$$
   \Norm{f}{\dot W^{1,p}(\mu)}
   \leq 4\Norm{h}{L^p(\mu)}.
$$
Taking the infimum over all such $h$ gives
$$
   \Norm{f}{\dot W^{1,p}(\mu)}
   \lesssim\Norm{f}{\dot M^{1,p}(\mu)}.
$$

Conversely, suppose that $f\in\dot W^{1,p}(\mu)$.
This means, by definition, that $f\in L^1_{\mathrm{loc}}(\mu)$ has a $p$-weak upper gradient
in $L^p(\mu)$
in the sense of test plans, and then, by Corollary \ref{cor:AGS7.4loc},
also in the sense of modulus. By Lemma \ref{lem:minimal-wug}, $f$ has minimal
$p$-weak upper gradients in both senses, respectively, denoted by $\abs{\nabla f}_{w,p}$
and $\abs{\nabla f}_p$. By Corollary \ref{cor:AGS7.4loc} again, there is coincidence
$\abs{\nabla f}_{w,p}=\abs{\nabla f}_p$ a.e.

We will now proceed with the minimal $p$-weak upper gradient in the sense of modulus,
$\abs{\nabla f}_p\in L^p(\mu)$, using results from \cite{HKST:book},
which are formulated in this setting.
From \cite[Lemma 6.2.2]{HKST:book}, we deduce that there exists a
sequence $\{g_k\}_{k\in\mathbb{N}}$ of the upper gradients of $f$ such that
$g_k \searrow |\nabla f|_p$ and $\||\nabla f|_p-g_k\|_p\to 0$ as $k\to \infty$.
Applying the self-improving property of
the Poincar\'e inequality (see \cite[Theorem 1.0.1]{KZ:08})
and \cite[Theorem 8.1.7]{HKST:book},
we find that there exist $q\in(1,p)$ and
a positive constant $C$ such that, for any $k\in\mathbb{N}$ and for $\mu$-almost every
$x,y\in X$,
\begin{align*}
|f(x)-f(y)|\le C \rho(x,y)\Big(\big[M(g_k^q)(x) \big]^{\frac1q}+
\big[M(g_k^q)(y) \big]^{\frac1q}\Big),
\end{align*}
where $M$ is the Hardy--Littlewood maximal operator.
Thus, $C[M(g_k^q)]^{\frac1q}$ is a Haj\l asz upper
gradient of $f$ for each $k\in\mathbb{N}$ and, by
the Hardy--Littlewood theorem, we further obtain, for any $k\in\mathbb{N}$,
\begin{align*}
\|f\|_{\dot{M}^{1,p}(\mu)}\lesssim \Big\|\big[M(g_k^q)\big]^{\frac1q}\Big\|_{p}\lesssim
\|g_k\|_{p}.
\end{align*}
Letting $k\to\infty$, we obtain
$$\|f\|_{\dot{M}^{1,p}(\mu)}\lesssim \||\nabla f|_p\|_{p}=
\|f\|_{\dot{W}^{1,p}},$$
where the last step used Corollary \ref{cor:AGS7.4loc}.
This completes the proof of Proposition \ref{prop:MWE}.
\end{proof}

\begin{proposition}\label{prop:truncW1p}
Let $p\in[1,\infty)$, $f\in L^1_{\mathrm{loc}}(\mu)$, and let $f_N$ be its standard
truncations \eqref{eq:fN-def}. Then $f\in\dot W^{1,p}(\mu)$ if and only if $f_N\in \dot
W^{1,p}(\mu)$ uniformly in $N\in\mathbb N$, and moreover
\begin{equation*}
  \Norm{f}{\dot W^{1,p}(\mu)}=\sup_{n\in\mathbb N}\Norm{f_N}{\dot W^{1,p}(\mu)}.
\end{equation*}
\end{proposition}

\begin{proof}
``Only if'':  Suppose that $f$ has a $p$-weak upper gradient in the sense of test plans, and
hence a minimal such object, $\abs{\nabla f}_{w,p}\in L^p(\mu)$. From \eqref{eq:fN-key}, it
follows that
\begin{equation*}
  \Babs{\int_{\partial\gamma}f_M}
  \leq\Babs{\int_{\partial\gamma}f}
  \leq\int_\gamma\abs{\nabla f}_{w,p}
\end{equation*}
for all $\gamma\in AC([0,1];X)\setminus\Gamma$, where $\Gamma$ is $p$-negligible. This means
that $\abs{\nabla f}_{w,p}$ is also a $p$-weak upper gradient in the sense of test plans for
$f_M$. Then minimality shows that $\abs{\nabla f_M}_{w,p}\leq\abs{\nabla f}_{w,p}$ a.e., and
taking the $L^p$ norms and then the supremum over $M\in\mathbb N$ proves that
\begin{equation}\label{eq:truncW1p-easy}
  \sup_{M\in\mathbb N}\Norm{f_M}{\dot W^{1,p}(\mu)}
  =\sup_{M\in\mathbb N}\Norm{\abs{\nabla f_M}_{w,p}}{L^p(\mu)}
  \leq\Norm{\abs{\nabla f}_{w,p}}{L^p(\mu)}
  =\Norm{f}{\dot W^{1,p}(\mu)}.
\end{equation}

``If'': Suppose that each $f_N$ has a $p$-weak upper gradient in the sense of test plans,
and hence a minimal such object, $\abs{\nabla f_N}_{w,p}\in L^p(\mu)$. It is immediate to
check that the truncations $f_N$ from \eqref{eq:fN-def} satisfy the tower property
$(f_N)_M=f_M$ for $M\leq N$. Applying the first part of the proof to $f_N$ in place of $f$
shows that $\abs{\nabla f_M}_{w,p}\leq\abs{\nabla f_N}_{w,p}$ a.e.\ whenever $M\leq N$.
Hence the sequence $\abs{\nabla f_N}_{w,p}$ is a.e.\ pointwise increasing in $N$. We may
then define
\begin{equation*}
\begin{split}
  g(x) &:=\sup_{N\in\mathbb N}\abs{\nabla f_N}_{w,p}(x) \\
  &\phantom{:}=\lim_{N\to\infty}\abs{\nabla f_N}_{w,p}(x)\quad\text{at a.e. }x\in X.
\end{split}
\end{equation*}
Then monotone convergence shows that
\begin{equation*}
  \Norm{g}{L^p(\mu)}=\lim_{N\to\infty}\Norm{\abs{\nabla f_N}_{w,p}}{L^p(\mu)}
  =\lim_{N\to\infty}\Norm{f_N}{\dot W^{1,p}(\mu)}
  \leq  \sup_{N\in\mathbb N}\Norm{f_N}{\dot W^{1,p}(\mu)}.
\end{equation*}

Finally, we check that $g$ is a $p$-weak upper gradient of $f$ in the sense of test plans.
Let $\Gamma_N$ be a $p$-negligible set of curves related to the $p$-weak upper gradient
property of $\abs{\nabla f_N}_{w,p}$, and $\Gamma:=\bigcup_{N=1}^\infty\Gamma_N$. Then
$\Gamma$ is also $p$-negligible.

Let $\gamma\in AC([0,1];X)\setminus\Gamma\subset AC([0,1];X)\setminus\Gamma_N$ for every
$N$. Choose
$$N>\max\{\abs{f(\gamma(s))}:s=0,1\}.$$ Then
\begin{equation*}
  \Babs{\int_{\partial\gamma}f}
  =\Babs{\int_{\partial\gamma}f_N}
  \leq\int_\gamma\abs{\nabla f_N}_{w,p}
  \leq\int_\gamma g.
\end{equation*}
Thus, $g\in L^p(\mu)$ is a $p$-weak upper gradient of $f$ in the sense of test plans, and
\begin{equation}\label{eq:truncW1p-harder}
  \Norm{f}{\dot W^{1,p}(\mu)}\leq\Norm{g}{L^p(\mu)}
    \leq  \sup_{N\in\mathbb N}\Norm{f_N}{\dot W^{1,p}(\mu)}.
\end{equation}
A combination of \eqref{eq:truncW1p-easy} and \eqref{eq:truncW1p-harder}
then completes the proof.
\end{proof}

\subsection{Test plans for BV theory}\label{ss:planBV}

The remaining preliminaries are dedicated to some replacements of the Sobolev spaces $\dot
W^{1,p}$ and related notions that will be needed to obtain the correct end-point results at
integrability $p=1$. We will first discuss the relevant notions of test plans in this
section. In Definition~\ref{def:plan} below, case \eqref{it:infty-plan} is just the natural
limiting case $q=\infty$ of Definition~\ref{def:fzplan}, while case \eqref{it:bary-plan} is
a new variant that is relevant for the BV theory. We will compare these two classes
carefully below.

\begin{definition}\label{def:plan}
A Borel probability $\pi\in\mathscr P(C([0,1];X))$ with the bounded compression property
\eqref{it:C(pi)} is called
\begin{enumerate}[\rm(i)]
  \item\label{it:infty-plan} an \emph{$\infty$-test plan} if
  \begin{equation*}
    \pi(AC^\infty([0,1];X))=1,\qquad
    [\gamma\mapsto{\operatorname{Lip}}(\gamma)]\in L^\infty(C([0,1];X),\pi);
  \end{equation*}
  \item\label{it:bary-plan} a \emph{bary-plan} if it has a \emph{barycenter} $b_\pi\in
  L^\infty(\mu)$
such that, for every bounded continuous function $g$ on $X$,
\begin{align}\label{eq:baryc}
\int_{AC([0,1];X)}\Bigg(\int_{\gamma}g \Bigg)\,d\pi =
\int_X gb_\pi\,d\mu.
\end{align}
\end{enumerate}
A Borel set $\Gamma\subset C([0,1];X)$ is said to be {\em $1$-negligible} (resp.\
\emph{bary-negligible}) if $\pi(\Gamma)=0$ for every $\infty$-test plan (resp.\ bary-plan)
$\pi$. A property is said to hold for {\em $1$-a.e.}\ (resp.\ \emph{bary-a.e.}) curve if it
holds for all curves outside a $1$-negligible (resp.\ bary-negligible) set.
\end{definition}

\begin{remark}
The notion of $\infty$-test plans is from \cite[Definition 5.5]{AD:14}.
Bary-plans were introduced under the name ``$1$-plans'' in \cite[Definition 2.2]{DMS}, where
more general barycenters $b_\pi\in L^q(\mu)$ are also considered. Since we will only need
$b_\pi\in L^\infty(\mu)$ here, we omit numerical labels from the name ``bary-plan'' for
brevity.
\end{remark}

The following basic observation is from \cite[Lemma 3.15]{AILP}.

\begin{lemma}\label{lem-plans}
Every $\infty$-test plan $\pi$ on a metric measure space $(X,\rho,\mu)$ is a bary-plan.
Moreover,
\begin{align}\label{bpi}
\|b_\pi\|_{L^\infty(\mu)}\le C(\pi)\|\gamma\mapsto{\operatorname{Lip}}(\gamma)
\|_{L^\infty(\pi)}<\infty,
\end{align}
where $b_\pi$ is as in \eqref{eq:baryc} and $C(\pi)$ is as in \eqref{it:C(pi)}.
\end{lemma}

\begin{remark}\label{rem:plans}
By Lemma \ref{lem-plans} and the definitions, every bary-negligible set is $1$-negligible.
\end{remark}

\begin{remark}\label{rem:same-neg}
The converse of Lemma \ref{lem-plans} fails, i.e., bary-plans are strictly more general than
$\infty$-test plans. Note that, unlike $\infty$-test plans,
bary-plans are not assumed to be concentrated on $AC^\infty([0,1];X)$.
But even if we added this requirement, the resulting
class of bary-plans concentrated on $AC^\infty([0,1];X)$ still strictly contains the one of
$\infty$-test plans;
see Example \ref{ex-plans}.

However, the negligible sets described by $\infty$-test plans
and bary-plans concentrated on
$AC^\infty([0,1];X)$ are the same (see,
for instance, \cite[p.\,40]{D14}). Indeed, fix a $1$-negligible set $A$.
For any $n\in\mathbb{N}$, let
$\Gamma_n:=\{\gamma\in AC^\infty([0,1];X):
{\operatorname{Lip}}(\gamma)\le n\}$.
Clearly, $AC^\infty([0,1];X)= \bigcup_{n=1}^\infty\Gamma_n$. For any
bary-plan $\pi$ concentrated on $AC^\infty([0,1];X)$ and for any
$n\in\mathbb{N}$,
define
\begin{align*}
\pi_n:=
\begin{cases}
\frac{\pi|_{\Gamma_n}}{\pi(\Gamma_n)} & \mbox{if}\ \pi(\Gamma_n)>0,\\
0 & \mbox{otherwise}.
\end{cases}
\end{align*}
It is easy to check that $\pi_n$ is an $\infty$-test plan and hence $\pi_n(A)=0$.
This further implies that
\begin{align*}
\pi(A)\le \sum_{n=1}^{\infty}\pi(A\cap \Gamma_n) = \sum_{n=1}^{\infty}
\pi(\Gamma_n)\pi_n(A)=0.
\end{align*}
Therefore, $A$ is also negligible with respect to all such bary-plans.
Thus, every $1$-negligible set $A$ is also bary-negligible, which,
combined with Remark \ref{rem:plans}, further implies the desired conclusion.
\end{remark}

\begin{example}\label{ex-plans}
The condition that a bary-plan is concentrated on
$AC^{\infty}([0,1];X)$ does not imply that it is an
$\infty$-test plan.

Indeed, consider the circle $\mathbb T=\mathbb R/\mathbb Z$, endowed
with its geodesic distance and normalized Haar measure $\mu$. For
$n\in\mathbb N$ and $a\in\mathbb T$, define
$$
    \gamma_{n,a}(t):=a+nt\pmod 1,
    \qquad t\in[0,1].
$$
Let $\eta_n$ be the pushforward of $\mu$ under the map
$\Phi_n:a\mapsto\gamma_{n,a}$, i.e., $\eta_n:= (\Phi_n)_\sharp\mu
$ and define
\begin{align}\label{eq:pi-eta}
\pi:=\sum_{n=1}^{\infty}2^{-n}\eta_n.
\end{align}
Translation invariance of $\mu$ gives
$(e_t)_\#\eta_n=\mu$ for any $t\in[0,1]$.
Indeed, for any Borel set $B\subseteq\mathbb T$, we have
\begin{align*}
(e_t)_\#\eta_n(B)
&=\eta_n\bigl(\{\gamma:\gamma(t)\in B\}\bigr)=\mu\bigl(\{a\in\mathbb T:a+nt\in B\}\bigr) \\
&=\mu(B-nt)=\mu(B),
\end{align*}
where the last equality follows from the translation invariance of the
Haar measure $\mu$. Consequently,
$$
(e_t)_\#\pi=\sum_{n=1}^{\infty}2^{-n}(e_t)_\#\eta_n =\mu.
$$
It is easy to verify that $|\gamma_{n,a}'|=n$ for any $a\in\mathbb T$
and hence, for any bounded nonnegative Borel function $g$ on $\mathbb T$,
\begin{align*}
\int\Big(\int_\gamma g\Big)\,d\pi(\gamma)
&=\sum_{n=1}^{\infty}2^{-n}\int_{\mathbb T}
\Big(\int_{\gamma_{n,a}}g \Big)\,d\mu(a)\quad\text{by \eqref{eq:pi-eta} and
$\eta_n=(\Phi_n)_\sharp\mu$}\\
&=\sum_{n=1}^{\infty}2^{-n}n\int_{\mathbb T}
\Big[\int_{0}^1 g(a+nt)\,dt \Big]\,d\mu(a)\quad \text{by $|\gamma_{n,a}'|=n$}\\
&=\sum_{n=1}^{\infty}2^{-n}n\int_{\mathbb T}g\,d\mu
    =2\int_{\mathbb T}g\,d\mu\quad\text{by Fubini's theorem.}
\end{align*}
Thus, $\pi$ is a bary-plan with barycenter $b_\pi=2$. Since
$\operatorname{Lip}(\gamma_{n,a})=n$, the measure $\pi$ is concentrated
on $AC^{\infty}([0,1];\mathbb T)$. However, for any
$M\in\mathbb N$,
$$
    \pi\bigl(\{\gamma:
    \operatorname{Lip}(\gamma)\in(M,\infty)\}\bigr)
    =\sum_{n=M+1}^{\infty}2^{-n}
    =2^{-M}.
$$
Consequently,
$\operatorname{Lip}(\gamma)\notin L^\infty(\pi)$, and hence $\pi$ is
not an $\infty$-test plan.
\end{example}

\subsection{BV spaces}\label{ss:BV}

In this final section of the preliminaries, we define and compare various notions of bounded
variation spaces on a metric measure space $(X,\rho,\mu)$. For this, we first introduce an
auxiliary notion on the unit interval, following \cite[(5.14)]{AD:14}:

\begin{definition}\label{def:BVunit}
Given $g:[0,1]\to\mathbb R$, we say that $g\in \dot{\rm BV}(0,1)$ if $g\in L^1(0,1)$ and the
distributional derivative $Dg$ is a finite measure, and we have the additional boundary
regularity
\begin{align}\label{eq:s9}
|g(1)-g(0)|\le |Dg|(0,1)<\infty,
\end{align}
where $\abs{Dg}$ is the total variation measure of $Dg$.
\end{definition}

We will use Definition \ref{def:BVunit} with $g=f\circ\gamma$, where $f\in
L^1_{\mathrm{loc}}(\mu)$ and $\gamma$ ranges over some family of curves in $C([0,1];X)$.
While a condition like \eqref{eq:s9} for a fixed $\gamma$ would depend on the choice of the
representative of $f$ within its equivalence class, it will become a meaningful condition
when required for a family of curves excluding some exceptional set.

Based on either $\infty$-test plans or bary-plans from Definition \ref{def:plan}, and the
corresponding notions of negligible from the same definition, we introduce two versions of
BV classes as follows:

\begin{definition}\label{def:BVboth}
Let $(X,\rho,\mu)$ be a metric measure space and $f\in L^1_{\mathrm{loc}}(\mu)$.
Suppose that, for some classes of curves $\Gamma\subset AC([0,1];X)$ and probability
measures $\Pi\subset\mathscr P(C([0,1];X))$, the following conditions are satisfied:
\begin{enumerate}[\rm(i)]
  \item\label{it:BV(0,1)} $f\circ\gamma\in\dot{\rm BV}(0,1)$ for every $\gamma\in
  AC([0,1];X)\setminus\Gamma$, and
  \item\label{it:exists-nu} there is a finite Borel measure $\nu$ on $X$ such that, for all
  $\pi\in\Pi$ and Borel sets $A\subset X$,
\begin{equation}\label{eq:BVboth}
  \int_{AC([0,1];X)}\abs{D(f\circ\gamma)}(\gamma^{-1}(A))\,d\pi(\gamma)  \leq K(\pi)\nu(A).
\end{equation}
\end{enumerate}
Then we say that $f$ has
\begin{enumerate}[\rm(i)]
  \item\label{it:BVw} {\em weak bounded variation}, and write $f\in\dot{\rm BV}_w(\mu)$ if
\begin{equation*}
   \Gamma\text{ is $1$-negligible},\quad\Pi=\{\text{all $\infty$-test plans}\},\quad
   K(\pi)=C(\pi)\Norm{\gamma\mapsto\operatorname{Lip}(\gamma)}{L^\infty(\pi)},
\end{equation*}
where $C(\pi)$ is the compression constant of $\pi$ as in \eqref{it:C(pi)};
  \item\label{it:BV*} {\em bounded variation}, and write $f\in\dot{\rm BV}_*(\mu)$ if
\begin{equation*}
   \Gamma\text{ is bary-negligible},\qquad\Pi=\{\text{all bary-plans}\},\qquad
   K(\pi)=\Norm{b_\pi}{L^\infty(\mu)},
\end{equation*}
where $b_\pi$ is the barycenter of $\pi$ as in \eqref{eq:baryc}.
\end{enumerate}
\end{definition}

\begin{remark}\label{rem:BV*}
In Definition \ref{def:BVboth}, part \eqref{it:BVw} is from \cite[Section 5.3]{AD:14}, while
part \eqref{it:BV*} is essentially from \cite[Definition 2.5(a)]{DMS}, up to the following
details:

First, in \cite[Definition 2.5(a)]{DMS}, it is not explicitly required that
$f\circ\gamma$ satisfy the boundary regularity condition
\eqref{eq:s9}, which appears in \cite[(5.14)]{AD:14}. Since \cite[p.\,1859]{DMS} states that
its
definition of BV is taken from \cite{AD:14}, we understand also \cite[Definition
2.5(a)]{DMS} as incorporating this condition.

Second, in \cite[Definition 2.5(a)]{DMS}, condition \eqref{eq:BVboth} is only assumed for
open sets $A\subset X$.
However, the left-hand side of \eqref{eq:BVboth}, say $\eta_\pi(A)$, defines a finite Borel
measure $\eta_\pi$ on $X$; hence \eqref{eq:BVboth} says that $\eta_\pi(A)\leq K(\pi)\nu(A)$.
In a metric measure space, every Borel measure $\lambda$ is outer regular, i.e., for all
Borel sets $A\subset X$, we have
\begin{equation*}
  \lambda(A)=\inf\{\lambda(G):A\subset G\subset X,\ G\text{ open}\};
\end{equation*}
see \cite[Proposition 3.3.37]{HKST:book}. Then, assuming \eqref{eq:BVboth} for all open
sets, and taking the infimum of both sides over all open sets containing a given Borel set,
we immediately deduce \eqref{eq:BVboth} for all Borel sets.
\end{remark}

\begin{lemma}\label{lem:smallest-nu}
Under the assumptions of Definition \ref{def:BVboth}\eqref{it:exists-nu}, there is a minimal
Borel measure $\nu_0$ that satisfies the same condition.
\end{lemma}

\begin{proof}
As in Remark \ref{rem:BV*}, denoting the left-hand side of \eqref{eq:BVboth} by
$\eta_\pi(A)$, this defines a Borel measure $\eta_\pi$, and \eqref{eq:BVboth} says that
$\eta_\pi(A)\leq K(\pi)\nu(A)$ for all Borel sets $A$. If $K(\pi)=0$, this implies that
$\eta_\pi\equiv 0$. Hence, we can define another Borel measure by
$\mu_\pi(A):=\eta_\pi(A)/K(\pi)$, taking $0/0:=0$, and the assumption of Definition
\ref{def:BVboth}\eqref{it:exists-nu} is further equivalent to $\mu_\pi(A)\leq\nu(A)$ for all
Borel sets $A$ and all $\pi\in\Pi$.
Let
\begin{equation*}
  \nu_0(A):=\sup\sum_{n=1}^\infty\mu_{\pi_n}(A_n),
\end{equation*}
where the supremum is taken over all countable Borel partitions $\{A_n\}_{n=1}^\infty$ of
$A$ and
all sequences $\{\pi_n\}_{n=1}^\infty\subset\Pi$. Then it is straightforward to verify that
$\nu_0$ defines a finite Borel measure, and $\nu_0(A)\leq\nu(A)$ for any Borel measure $\nu$
that
satisfies \eqref{eq:BVboth}; cf. \cite[(2.8)]{AILP}. This completes the proof.
\end{proof}

Lemma \ref{lem:smallest-nu} allows us to make the following definition:

\begin{definition}\label{def:BVboth2}
In the setting of the respective part of Definition \ref{def:BVboth}, the smallest measure
$\nu$ that satisfies \eqref{eq:BVboth} is denoted by
\begin{enumerate}[\rm(i)]
  \item $\abs{Df}_w$ for $f\in\dot{\rm BV}_w(\mu)$ and other parameters as in
  Definition~\ref{def:BVboth}\eqref{it:BVw};
  \item $\abs{Df}_*$ for $f\in\dot{\rm BV}_*(\mu)$ and other parameters as in
  Definition~\ref{def:BVboth}\eqref{it:BV*}.
\end{enumerate}
We also denote
\begin{equation*}
  \Norm{f}{\dot{\rm BV}_w(\mu)}
  :=\begin{cases} \abs{Df}_w(X) & f\in \dot{\rm BV}_w(\mu), \\ \infty &
  \text{otherwise},\end{cases}\qquad
  \Norm{f}{\dot{\rm BV}_*(\mu)}
  :=\begin{cases} \abs{Df}_*(X) & f\in \dot{\rm BV}_*(\mu), \\ \infty &
  \text{otherwise}.\end{cases}
\end{equation*}
\end{definition}

Next, we study the relation of these two BV spaces.

\begin{lemma}\label{lem:tv-c}
Let $(X,\rho,\mu)$ be a metric measure space.
If $f\in \dot{\rm BV}_*(\mu)$, then $f\in \dot{\rm BV}_w(\mu)$
with $|Df|_w\le |Df|_*$.
\end{lemma}

\begin{proof}
Suppose that $f\in\dot{\rm BV}_*(\mu)$, hence the conditions of Definition
\ref{def:BVboth}\eqref{it:BV*} hold. By Remark \ref{rem:plans}, the bary-negligible set
$\Gamma$ is also $1$-negligible. By Lemma \ref{lem-plans}, the set $\Pi$ of all bary-plans
contains the set $\Pi_\infty$ of all $\infty$-test plans. Moreover, by the same lemma the
constant $K(\pi)=\Norm{b_\pi}{L^\infty(\mu)}$ is dominated by
$C(\pi)\Norm{\gamma\mapsto\operatorname{Lip}(\gamma)}{L^\infty(\pi)}$ for
$\pi\in\Pi_\infty$. Hence also the conditions of Definition \ref{def:BVboth}\eqref{it:BVw}
hold with the same measure $\nu$, and thus $f\in\dot{\rm BV}_w(\mu)$. Taking
$\nu=\abs{Df}_*$, and noting that $\abs{Df}_w$ is the smallest measure that satisfies these
conditions, it follows that $\abs{Df}_w\leq\abs{Df}_*$. This completes the proof of Lemma
\ref{lem:tv-c}.
\end{proof}

To obtain a comparison in the other direction, we need yet another notion of the bounded
variation
from \cite{M:03}. This is defined with the help of (locally) Lipschitz functions.
Given an open set $A\subset X$, for any $u\in\operatorname{Lip}_{\mathrm{loc}}(A)$ and $x\in
A$, let
\begin{equation}\label{eq:lipu}
  \operatorname{lip}u(x)
  :=\liminf_{r\searrow 0}\frac{1}{r}\sup_{y\in B(x,r)}|u(y)-u(x)|
  =\liminf_{r\searrow 0}\frac{1}{r}\sup_{y\in \bar B(x,r)}|u(y)-u(x)|,
\end{equation}
where the identity between the last two forms is elementary.

\begin{remark}\label{rem:lipu}
If $f\in\operatorname{Lip}_{\mathrm{loc}}(X)$,
then $\operatorname{lip} f$ is an upper gradient of $f$;
see, e.g., \cite[Lemma 6.2.6]{HKST:book}.
\end{remark}

We now come to a definition of bounded variation through relaxation from \cite{M:03}.

\begin{definition}\label{def:BV-Mir}
Let $(X,\rho,\mu)$ be a complete doubling metric measure space and $f\in
L_{\mathrm{loc}}^1(\mu)$.
For any open set $A\subset X$, let
\begin{align}\label{eq:Df(A)}
  |Df|(A):=\inf_{\{f_n\}_{n\in\mathbb N}}
  \liminf_{n\to\infty}\int_A\operatorname{lip}f_n(x)\,d\mu(x),
\end{align}
where the infimum is taken over all sequences
$\{f_n\}_{n\in\mathbb N}\subset\operatorname{Lip}_{\mathrm{loc}}(A)$ such that
$f_n\to f$ in $L_{\mathrm{loc}}^1(A)$. For an arbitrary set $E\subset X$, let
$$
|Df|(E):=\inf\left\{|Df|(G): E\subset G\subset X
\ \text{and }G\text{ is open}\right\}.
$$
We say that $f$ has {\em relaxed bounded variation} if $|Df|(X)<\infty$, and define
\begin{align}\label{BV}
\dot{\rm BV}(\mu):=\left\{f\in L_{\mathrm{loc}}^1(\mu):
|f|_{\dot{\rm BV}(\mu)}:=|Df|(X)<\infty\right\}.
\end{align}
By convention, we let $|f|_{\dot{\rm BV}(\mu)}=\infty$ if
$f\notin\dot{\rm BV}(\mu)$. \end{definition}

\begin{remark}\label{rem:BV-Mir}
The set function $|Df|$ defined in Definition \ref{def:BV-Mir} is an outer measure whose
restriction to the Borel
$\sigma$-algebra is a regular Borel measure; see \cite{M:03}. We use the same symbol $|Df|$
for this Borel measure.
\end{remark}

\begin{remark}
Many of the more recent papers, like \cite[Section 5]{AD:14}, deal with the inhomogeneous BV
class with the global integrability assumption $f\in L^1(\mu)$ and take the infimum in
\eqref{eq:Df(A)} over sequences $f_n\to f$ in $L^1(A)$; cf.\ \cite[(5.1)]{AD:14}. However,
the original definition of \cite[(7)]{M:03} allows $f\in L^1_{\mathrm{loc}}(\mu)$ and takes
the infimum over the same sequences as in \eqref{eq:Df(A)}.
\end{remark}

We now record a useful localization property that we will need. It is the BV analogue of
Proposition \ref{prop:truncW1p}.

\begin{proposition}\label{p-local}
Let $(X,\rho,\mu)$ be a complete doubling metric measure space,
$\{X_n\}_{n\in\mathbb N}$ be an increasing sequence of closed bounded subsets
of $X$ with the bounded exhaustion property (Definition \ref{def:bdExhaustion}),
and let $\mu_n:=\mu|_{X_n}$. For
$f\in L_{\mathrm{loc}}^1(\mu)$, let $f^n:=f|_{X_n}$. Then, for every
$n_0\in\mathbb N$,
\begin{equation}
|Df|(X)=\sup_{n\ge n_0}|Df^n|(X_n),
\end{equation}
where both sides are allowed to be infinite.
\end{proposition}

\begin{proof}
  Fix $n_0\in\mathbb N$.
We first prove
$$
\sup_{n\ge n_0}|Df^n|(X_n)\le |Df|(X).
$$
There is nothing to prove if $|Df|(X)=\infty$. Otherwise, given
$\varepsilon\in(0,\infty)$, choose
$\{h_k\}_{k\in\mathbb N}\subset\operatorname{Lip}_{\mathrm{loc}}(X)$ such that
$h_k\to f$ in $L_{\mathrm{loc}}^1(\mu)$ and
$$
\liminf_{k\to\infty}\int_X\operatorname{lip}_X h_k\,d\mu
\le |Df|(X)+\varepsilon.
$$
Since a complete doubling metric space is proper, each $X_n$ is compact.
Thus, $h_k|_{X_n}\to f^n$ in $L^1(\mu_n)$ and
$$
\operatorname{lip}(h_k|_{X_n})
\le \operatorname{lip} h_k
\quad\text{on }X_n.
$$
It follows that
$$
|Df^n|(X_n)
\le |Df|(X)+\varepsilon
$$
for every $n\ge n_0$. Letting $\varepsilon\searrow 0$ proves the first
inequality.

For the converse, let
$$
M:=\sup_{n\ge n_0}|Df^n|(X_n)
$$
and assume that $M<\infty$. Fix $x_0\in X$ and $m\in\mathbb N$. By the
bounded exhaustion property, there exists $N=N(m)\ge n_0$ such that
$B(x_0,m)\subset X_N$. Hence
$$
|Df|(B(x_0,m))
=
|Df^N|(B(x_0,m))
\le |Df^N|(X_N)
\le M.
$$
Letting $B(x_0,m)\nearrow X$ gives
$$
|Df|(X)=\lim_{m\to\infty}|Df|(B(x_0,m))\le M.
$$
This proves the reverse inequality and completes the proof.
\end{proof}

Proposition \ref{p-local} and the
known inhomogeneous equivalence
yield the following conclusion.

\begin{lemma}\label{lem:tv-w-lower}
Let $(X,\rho,\mu)$ be a complete doubling metric measure space. Then, for
every $f\in L_{\mathrm{loc}}^1(\mu)$,
$$
|Df|(X)\le |Df|_w(X).
$$
\end{lemma}

\begin{proof}
It suffices to consider $f\in\dot{\rm BV}_w(\mu)$. Let
$\{X_n\}_{n\in\mathbb N}$ and $\mu_n$ be as in Proposition
\ref{p-local}, and let $f^n:=f|_{X_n}$. Since $X_n$ is bounded,
$f^n\in L^1(\mu_n)$. Moreover, every $\infty$-test plan on
$(X_n,\rho|_{X_n},\mu_n)$ is an $\infty$-test plan on $(X,\rho,\mu)$ with
the same bounded compression constant as in \eqref{it:C(pi)}. Therefore,
\begin{equation}\label{eq:weak-restriction}
|Df^n|_{w}(X_n)
\le |Df|_{w}(X_n)
\le |Df|_{w}(X).
\end{equation}
The space $X_n$ is complete and separable, $\mu_n$ is finite and $f^n\in L^1(\mu_n)$, and
hence
\cite[Theorem 1.1]{AD:14} gives
$$
|Df^n|(X_n)=|Df^n|_{w}(X_n).
$$
Applying Proposition \ref{p-local} and then \eqref{eq:weak-restriction}, we
obtain
\begin{align*}
|Df|(X)
=\sup_{n\in\mathbb N}|Df^n|(X_n)
=\sup_{n\in\mathbb N}|Df^n|_{w}(X_n)
\le |Df|_{w}(X).
\end{align*}
This completes the proof of Lemma \ref{lem:tv-w-lower}.
\end{proof}

The inhomogeneous equivalence between the relaxed and weak definitions of
BV is due to Ambrosio and Di Marino \cite[Theorem 1.1]{AD:14}.  The
bary-plan formulation used here is taken from Di Marino and Squassina
\cite[Definition 2.5]{DMS}.  The corresponding equivalence for the
homogeneous spaces considered in this paper is recorded next.

\begin{proposition}\label{p-eq}
Let $(X,\rho,\mu)$ be a complete doubling metric measure space. Then
$$
\dot{\rm BV}_w(\mu)=\dot{\rm BV}(\mu)=\dot{\rm BV}_*(\mu)
$$
and $|Df|_w(X)=|Df|(X)=|Df|_*(X)$ for every $f\in L_{\mathrm{loc}}^1(\mu)$.
\end{proposition}

The proof of Proposition \ref{p-eq} is given in Appendix \ref{app:homogeneous-BV-equivalence}.

We emphasize that the complete equivalence in Proposition \ref{p-eq} is
not needed in the applications below.  Those applications only use the
one-sided comparison
\begin{equation}\label{eq:BV-one-sided-comparison}
  |Df|(X)\leq |Df|_w(X)\leq |Df|_*(X),
\end{equation}
which follows directly from Lemmas \ref{lem:tv-w-lower} and
\ref{lem:tv-c}.  For completeness, Appendix
\ref{app:homogeneous-BV-equivalence} proves the reverse comparison and
hence the full homogeneous equivalence.

\section{Frank-type characterizations}\label{s:Frank}

We now come to the main part of the paper dealing with several new characterizations of the
Sobolev and BV spaces that we have presented in Section \ref{sec:preliminaries}.
We begin by setting up some notation. For a function $f\in L^1_{\mathrm{loc}}(\mu)$ [recall
\eqref{eLoc}] and a ball $B=B(x,t)$ of center $x\in X$ and radius $t\in(0,\infty)$, we
define the \emph{mean oscillation}
\begin{equation*}
  m_f(B):=m_f(x,t):=\fint_{B}\abs{f-\ave{f}_B}.
\end{equation*}
Slightly adapting the notation of \cite{DLYYZ}, we define the \emph{lifted maximal operator}
\begin{equation}\label{eq:Msharp(x,t)}
  \mathcal M^{\#}f(x,t)
  :=\sup_{\gfz{B\owns x}{r_B\in[t,2t)}}\frac{m_f(B)}{r_B}, \quad x\in X\ \text{and}\
  t\in(0,\infty).
\end{equation}
From this single definition, we infer the two variants considered in \cite{DLYYZ}: for
any $\theta\in(0,\infty)$, $x,y\in X$, and $t\in(0,\infty)$,
\begin{equation*}
\begin{split}
  \mathcal M^{\#}_\theta f(x,t)
  &:=\mathcal M^{\#}f(x,\theta t),\\
  f_\theta^*(x,y)
  &:=\mathcal M^{\#}_\theta f(x,\rho(x,y))
  =\mathcal M^{\#}f(x,\theta\rho(x,y)).
\end{split}
\end{equation*}
Strictly speaking, \cite{DLYYZ} defines $f_\theta^*$ with a closed interval $[t,2t]$ in
place of $[t,2t)$ in \eqref{eq:Msharp(x,t)}, but this distinction is not essential.

Now we are ready for our first new characterization of the Sobolev space. We state it in
terms of $\dot M^{1,p}(\mu)$, since this space very naturally comes up in the proof, but
recall that it coincides with $\dot W^{1,p}(\mu)$ under the assumptions below, by
Proposition \ref{prop:MWE}.

\begin{theorem}\label{thm:Frank-type}
Let $p\in(1,\infty)$, $\gamma\in\mathbb R\setminus\{0\}$, and $(X,\rho,\mu)$ be a complete
and doubling $p$-Poincar\'e space. If $f\in L^1_{\mathrm{loc}}(\mu)$, then
\begin{subequations}\label{eq:Frank-type}
\begin{align}
  \Norm{f}{\dot M^{1,p}(\mu)}
  &\lesssim\Norm{t^{-1-\gamma}m_f}{L^{p,\infty}(t^{\gamma p-1}dt\,d\mu)}
    \label{eq:Frank-lower} \\
  &\lesssim\sup_{\theta\in(0,1]}
    \Norm{t^{-\gamma}\mathcal M_\theta^{\#}f}{L^{p,\infty}(t^{\gamma p-1}dt\,d\mu)}
   \lesssim\Norm{f}{\dot M^{1,p}(\mu)}.
   \label{eq:Frank-upper}
\end{align}
\end{subequations}
In particular, $f\in\dot M^{1,p}(\mu)$ if and only if either one of the quantities in the
middle of \eqref{eq:Frank-type} is finite.
\end{theorem}

Except for the intermediate step involving $\mathcal M_{\theta}^{\#}$, case $\gamma=-1$ is
due to \cite{Frank:W1p} for $X=\mathbb R^d$, and due to \cite{HK:W1p} for the same class of
spaces as in Theorem \ref{thm:Frank-type}. With $\mathcal M_{\theta}^{\#}$ and for other
values of $\gamma$, the norm equivalence \eqref{eq:Frank-type} is due to \cite[Theorem 1.12
and Corollary 1.13]{DLYYZ} with $\Norm{\operatorname{lip}f}{L^p(\mu)}$ in place of
$\Norm{f}{\dot M^{1,p}(\mu)}$ and under the
a priori assumption that $f\in\operatorname{Lip}(X)$. Related results on $p=1$ are also
contained in \cite{DLYYZ}. While restricted to $p\in(1,\infty)$, the novelty of Theorem
\ref{thm:Frank-type} is obtaining a characterization of the whole space $\dot M^{1,p}(\mu)$,
instead of just an equivalent norm of test functions in this space.

\begin{lemma}\label{lem:Msharpf<Mh}
If $f\in L^1_{\mathrm{loc}}(\mu)$ has a Haj\l{}asz upper gradient $h\in
L^1_{\mathrm{loc}}(\mu)$, then
\begin{equation*}
   \frac{m_f(B)}{r_B}\lesssim\fint_B h
\end{equation*}
for all balls $B$. In particular,
\begin{equation*}
  \mathcal M^{\#}f(x,t)\lesssim Mh(x),\quad
  \mathcal M^{\#}_\theta f(x,t)\lesssim Mh(x),\quad
  f^*_\theta (x,y)\lesssim Mh(x),
\end{equation*}
where $M$ is the Hardy--Littlewood maximal operator.
\end{lemma}

\begin{proof}
The first estimate is contained in the proof of \cite[Proposition 5.1]{HK:W1p}.
The bound for $\mathcal M^{\#}f$ follows from taking the supremum over all balls $B\owns x$
of radius $r_B\in[t,2t)$. The last two bounds follow by substituting $\theta t$ or
$\theta\rho(x,y)$ in place of $t$ into the bound for $\mathcal M^{\#}f$.
This completes the proof of Lemma \ref{lem:Msharpf<Mh}.
\end{proof}

\begin{proof}[Proof of the upper bound \eqref{eq:Frank-upper}]
The first $\lesssim$ in \eqref{eq:Frank-upper} follows from the elementary pointwise bound
$t^{-1}m_f\leq\mathcal M_1^{\#}f$, which is immediate from the definition and also recorded
as \cite[(2.13)]{DLYYZ}.

For the second $\lesssim$, if $h$ is a Haj\l{}asz upper gradient of $f$, it follows that
\begin{equation*}
\begin{split}
   \lambda^p\iint_{\{t^{-\gamma}\mathcal M_\theta^{\#}f>\lambda\}}t^{\gamma p-1}\,dt\,d\mu
   &\leq \lambda^p\iint_{\{t^{-\gamma}Mh>c\lambda\}}t^{\gamma p-1}\,dt\,d\mu
   \quad\text{by Lemma \ref{lem:Msharpf<Mh}} \\
   &\sim\lambda^p\int\Big(\frac{Mh}{c\lambda}\Big)^p\, d\mu
   \sim\int (Mh)^p\,d\mu\sim\Norm{h}{L^p(\mu)}^p
\end{split}
\end{equation*}
by an elementary integration in $t$ in the first $\sim$ and the Hardy--Littlewood maximal
theorem in the last one. Taking the supremum over $\lambda>0$, we obtain $\Norm{ t^{-\gamma}
\mathcal M_\theta^{\#} f
 }{ L^{p,\infty}(t^{\gamma p-1}dt\,d\mu)}
\lesssim\Norm{h}{L^p(\mu)}$. Taking the infimum over all $h$ and the supremum over
$\theta\in(0,1]$, we obtain \eqref{eq:Frank-upper}.
This completes the proof of the upper bound \eqref{eq:Frank-upper}.
\end{proof}

For the lower bound \eqref{eq:Frank-lower}, we use the following idea from \cite{HK:W1p}.

Given $f\in L^1_{\mathrm{loc}}(\mu)$ and $\alpha\in (0,\infty)$,
we consider the approximate identities \cite[(3.2)]{HK:W1p}
\begin{equation}
\Phi_s f(x):=\sum_i \phi_i(x)\ave{f}_{B(x_i,s)}, \quad s\in(0,\infty)\ \text{and}\ x\in X,
\end{equation}
where the balls $B(x_i,\frac12 s)$ are pairwise disjoint, and the functions $\phi_i$ satisfy
\begin{equation*}
  0\leq\phi_i\leq \mathbf{1}_{B(x_i,(1+\alpha)s)},\qquad
  \sum_i\phi_i\equiv 1,\qquad
  \operatorname{Lip}(\phi_i)\lesssim s^{-1}.
\end{equation*}
Then $\Phi_s f\in\operatorname{Lip}_\mathrm{loc}(X)$.

\begin{lemma}\label{lem:lip(Phis)<K}
Let $p\in(1,\infty)$ and $(X,\rho,\mu)$ be a complete doubling
$p$-Poincar\'e space. Let $f\in L^1_{\mathrm{loc}}(\mu)$ and suppose that
there exists a positive constant $K$ such that $\Norm{\operatorname{lip}(\Phi_s
f)}{L^p(\mu)}\leq K$ uniformly in $s>0$. Then $\Norm{f}{\dot M^{1,p}(\mu)}\lesssim K$.
\end{lemma}

\begin{proof}
This is implicitly contained in the proof of \cite[Proposition 8.2]{HK:W1p}. We indicate the
main steps for completeness.

By assumption, the functions $\operatorname{lip}(\Phi_s f)^q$ form a bounded set in the
reflexive space $L^{p/q}(\mu)$, where we denote by $q\in(1,p)$ an exponent guaranteed by
\cite{KZ:08} such that $X$ is a $q$-Poincar\'e space. Hence, there is a sequence $s_j\to 0$
such that $\operatorname{lip}(\Phi_{s_j} f)^q\rightharpoonup g^q\in L^{p/q}(\mu)$ for some
$g\in L^p(\mu)$, where $\rightharpoonup$ denotes weak convergence in $L^{p/q}(\mu)$.

As in \cite[(8.4) and (8.5)]{HK:W1p}, we obtain
\begin{equation}\label{eq:HK8.4}
  \Norm{g}{L^p(\mu)}
  \leq\liminf_{j\to\infty}\Norm{\operatorname{lip}(\Phi_{s_j}f)}{L^p(\mu)} \leq K
\end{equation}
and, at Lebesgue points of $f$,
\begin{equation*}
  \abs{f(x)-f(y)}
  \lesssim\rho(x,y)\left[M_q g(x)+M_q g(y)\right],
\end{equation*}
where $M_q g:=(M\abs{g}^q)^{\frac1q}$ is the rescaled maximal function.

This shows that $c\cdot M_q g$, with some constant $c$, is a Haj\l{}asz upper gradient of
$f$, and hence
\begin{equation*}
  \Norm{f}{\dot M^{1,p}(\mu)}
  \lesssim\Norm{M_q g}{L^p(\mu)}
  \lesssim\Norm{g}{L^p(\mu)}
  \leq K
\end{equation*}
by \eqref{eq:HK8.4} in the last step.
This completes the proof of Lemma \ref{lem:lip(Phis)<K}.
\end{proof}

\begin{proof}[Proof of the lower bound \eqref{eq:Frank-lower}]
We extend the approach of \cite{HK:W1p} from the case $\gamma=-1$.

While \eqref{eq:Frank-lower} is stated in \cite[Corollary 1.13]{DLYYZ} for
$f\in\operatorname{Lip}(X)$, its proof in \cite[Section 2.2]{DLYYZ} (where the key
ingredient is \cite[Proposition 6.3]{HK:W1p}) only requires
$f\in\operatorname{Lip}_{\mathrm{loc}}(X)$. Hence, we have
\begin{equation}\label{eq:1.13Phis}
  \Norm{\operatorname{lip}(\Phi_s f)}{L^p(\mu)}
  \lesssim\Norm{t^{-1-\gamma}m_{\Phi_s f}}{L^{p,\infty}(t^{p\gamma-1}dt\,d\mu)}.
\end{equation}
From \cite[Lemma 7.3]{HK:W1p}, we deduce the pointwise bound
\begin{equation}\label{eq:HK7.3}
  m_{\Phi_s f}(x,t)
  \lesssim\mathcal A_{s,t}( m_f(\cdot,t))(x),
\end{equation}
where
\begin{equation*}
  \mathcal A_{s,t}g(x)
  := \Big[\fint_{B(x, C\max(s,t))} g(z)^q\, d\mu(z)\Big]^{\frac 1q},
\end{equation*}
provided that $X$ is a $q$-Poincar\'e space. Thanks to \cite{KZ:08}, this is the case for
some $q\in(1,p)$. Multiplying both sides of  \eqref{eq:HK7.3} by $t^{-1-\gamma}$, we obtain
\begin{equation}\label{eq:HK7.3b}
  \frac{m_{\Phi_s f}(x,t)}{t^{1+\gamma}}
  \lesssim
  \mathcal A_{s,t}\Big(\frac{m_f(\cdot,t)}{t^{1+\gamma}}\Big)(x).
\end{equation}

We next adapt the proof of \cite[Lemma 7.5]{HK:W1p}. For each fixed $s,t\in(0,\infty)$, it
is immediate from H\"older's inequality that $\mathcal A_{s,t}$ is bounded on $L^u(\mu)$ for
all $u\in[q,\infty)$. Hence, by Fubini's theorem,
\begin{equation*}
  G\mapsto \mathcal A_s G,\quad
  \mathcal A_s G(x,t):=\mathcal A_{s,t}(G(\cdot,t))(x)
\end{equation*}
is bounded on $L^u(\mu\times\nu)$ for any $\sigma$-finite measure $\nu$ on $(0,\infty)$ and
for the same range $u\in[q,\infty)$. Then, by the Marcinkiewicz interpolation theorem (see,
e.g., \cite[Theorem 2.2.3]{HNVW1} for a version that covers the case at hand), the sublinear
operator $\mathcal A_s$ is also bounded  on $L^{u,\infty}(\mu\times\nu)$ for all
$u\in(q,\infty)$. In particular, choosing $u=p$ and $d\nu(t)=t^{\gamma p-1}dt$, it follows
that
\begin{equation}\label{eq:HK7.5}
  \BNorm{\mathcal A_{s,t}\Big(\frac{m_f(\cdot,t)}{t^{1+\gamma}}\Big)}{L^{p,\infty}(t^{\gamma
  p-1}dt\,d\mu)}
  \lesssim\BNorm{\frac{m_f}{t^{1+\gamma}}}{L^{p,\infty}(t^{\gamma p-1}dt\,d\mu)}.
\end{equation}
Combining these estimates, it follows that
\begin{equation}
\begin{split}
  \Norm{\operatorname{lip}(\Phi_s f)}{L^p(\mu)}
  &\lesssim\Norm{t^{-1-\gamma}m_{\Phi_s f}}{L^{p,\infty}(t^{p\gamma-1}dt\,d\mu)}
  \quad\text{by \eqref{eq:1.13Phis}} \\
  &\lesssim\Norm{\mathcal A_{s,t}(t^{-1-\gamma}m_f(\cdot,t))
  }{L^{p,\infty}(t^{p\gamma-1}dt\,d\mu)}
  \quad\text{by \eqref{eq:HK7.3b}} \\
  &\lesssim\Norm{t^{-1-\gamma}m_f}{L^{p,\infty}(t^{p\gamma-1}dt\,d\mu)}=:K
  \quad\text{by \eqref{eq:HK7.5}}.
\end{split}
\end{equation}
Thus, we have verified the assumptions of Lemma \ref{lem:lip(Phis)<K} with $K$ as above. The
said lemma then implies that
\begin{equation*}
  \Norm{f}{\dot M^{1,p}(\mu)}
  \lesssim K:=\Norm{t^{-1-\gamma}m_f}{L^{p,\infty}(t^{p\gamma-1}dt\,d\mu)}.
\end{equation*}
This completes the proof of \eqref{eq:Frank-lower} and hence of Theorem
\ref{thm:Frank-type}.
\end{proof}

\section{BBM-type characterizations}\label{sec:DMS}

An influential point of departure for the particular line of singular nonlocal
characterizations considered here is the work of Bourgain, Brezis, and Mironescu \cite{BBM}.
Their results have a slightly different spirit in the sense of involving limits of Sobolev
norms of lower, fractional smoothness. Such limit relations are not only interesting in
their own right but also serve as tools for obtaining other kinds of characterizations,
e.g., in \cite[Section 4]{BSVY}. Extensions of the \cite{BBM} results to metric measure
spaces have already been explored by a number of authors \cite{DMS,Munnier}, and we quote
the following theorem as an example:

\begin{theorem}[\cite{DMS}, Theorem 1.4]\label{thm:DMS}
Let $p\in[1,\infty)$, and $(X,\rho,\mu)$ be a complete doubling $p$-Poincar\'e space.
If $f\in L^p(\mu)$, then
\begin{subequations}\label{eq:DMS}
\begin{align}
  {\operatorname{Ch}}_p(f)
  &\lesssim
  \liminf_{s\nearrow 1}(1-s)\Norm{\Delta f}{L^{p}(\rho^{-ps}V^{-1})}^p
    \label{eq:DMS-lower} \\
  &\leq
  \limsup_{s\nearrow 1}(1-s)\Norm{\Delta f}{L^{p}(\rho^{-ps}V^{-1})}^p
   \lesssim{\operatorname{Ch}}_p(f),
   \label{eq:DMS-upper}
\end{align}
\end{subequations}
where $\Delta f(x,y):=\abs{f(x)-f(y)}$ for any $x,y\in X$ and
\begin{align*}
{\operatorname{Ch}}_p(f):=
\begin{cases}\Norm{f}{\dot W^{1,p}(\mu)}^p & \text{for}\ p\in(1,\infty), \\
|Df|_*(X) & \text{for}\ p=1.\end{cases}
\end{align*}
\end{theorem}

\begin{remark}
Under the additional assumption of Ahlfors regularity, the case $p\in(1,\infty)$ of Theorem
\ref{thm:DMS} was already obtained in \cite[Theorem 2]{Munnier}.
\end{remark}

 Our present interest in this type of characterizations is practical: having them as tools
 for other characterizations in the same way as in \cite[Section 4]{BSVY}. With this in
 mind, while Theorem \ref{thm:DMS} is an elegant statement in its own right, it turns out
 that a somewhat technical modification is necessary for our needs:

\begin{proposition}\label{prop:DMS}
Let $p\in(1,\infty)$ and $(X,\rho,\mu)$ be a complete doubling metric measure space. If
$f\in L^1_{\mathrm{loc}}(\mu)$, then
\begin{subequations}
\begin{equation}\label{eq:DMS-lower-v2}
\begin{split}
  \Norm{f}{\dot W^{1,p}(\mu)}^p \lesssim
  \sup_{E,N}\liminf_{s\nearrow 1}(1-s)\Norm{\mathbf{1}_{E\times E}\Delta
  f_N}{L^p(\rho^{-ps}V^{-1})}^p.
\end{split}
\end{equation}
If, in addition, $X$ is a $p$-Poincar\'e space, then
\begin{equation}\label{eq:DMS-upper-v2}
  \sup_{E,N}\limsup_{s\nearrow 1}(1-s)\Norm{\mathbf{1}_{E\times E}\Delta
  f_N}{L^p(\rho^{-ps}V^{-1})}^p
  \lesssim\Norm{f}{\dot W^{1,p}(\mu)}^p.
\end{equation}
\end{subequations}
where $f_N$ is the usual truncation \eqref{eq:fN-def}, and the supremum is taken over all
bounded measurable sets $E\subset X$ and over all $N\in\mathbb N$.
\end{proposition}

\begin{remark}\label{rem:loc-global}
Note that $\Delta f_N\leq \Delta f$ by \eqref{eq:fN-key}, and obviously $\mathbf{1}_{E\times
E}\leq 1$. Moreover, the assumption $f\in L^1_{\mathrm{loc}}(\mu)$ is more general than
$f\in L^p(\mu)$. Hence \eqref{eq:DMS-lower-v2} is a strengthening of \eqref{eq:DMS-lower} in
the range $p\in(1,\infty)$. Moreover, this technical modification is necessary to obtain a
two-sided estimate when we only assume that $f\in L^1_{\mathrm{loc}}(\mu)$, instead of $f\in
L^p(\mu)$ as in Theorem \ref{thm:DMS}. Indeed, already in $\mathbb R^d$, the quantity
$\Norm{\Delta f}{L^p(\rho^{-ps}V^{-1})}$ is the seminorm of the homogeneous Besov space
$\dot B^s_{p,p}$, which is incomparable with the homogeneous Sobolev space $\dot W^{1,p}$ of
different smoothness. This is most easily seen for $p=2$ in the Fourier domain:
\begin{equation*}
  \Norm{f}{\dot W^{1,2}(\mathbb R^d)}^2
  \sim\int_{\mathbb R^d}\abs{\xi}^2\abs{\hat f(\xi)}^2\, d\xi,\qquad
  \Norm{f}{\dot B^s_{2,2}(\mathbb R^d)}^2
  \sim_s\int_{\mathbb R^d}\abs{\xi}^{2s}\abs{\hat f(\xi)}^2\, d\xi.
\end{equation*}
For $f\in \dot W^{1,2}(\mathbb R^d)$, it can easily happen that $f\notin \dot
B^s_{2,2}(\mathbb R^d)$ for all $s<1$; hence, the middle terms in \eqref{eq:DMS} can be
$\infty$ while the right-hand side is finite, making \eqref{eq:DMS-upper} impossible. (While
it is equally easy to find $f\in\dot B^s_{2,2}(\mathbb R^d)\setminus\dot W^{1,2}(\mathbb
R^d)$ for any {\em individual} $s\in(0,1)$, there is no contradiction with
\eqref{eq:DMS-lower}, where the finiteness of the right-hand side requires that $f\in\dot
B^s_{2,2}(\mathbb R^d)$ for a {\em sequence} of $s=s_j\nearrow 1$, with a control on the
respective norms.)
\end{remark}

The proof of the upper bound \eqref{eq:DMS-upper-v2} is postponed to Section
\ref{sec:Nguyen},
and we now turn to the proof of the lower bound \eqref{eq:DMS-lower-v2}.
This will be achieved by combining the methods of \cite{DMS} from their proof of Theorem
\ref{thm:DMS}
 with the techniques of approximating by bounded subsets from Section
 \ref{sec:bdApprox}.
Indeed, to obtain a lower bound for the supremum over all bounded sets, it is enough to
consider
some particular sets, and choosing $E=X_n$ as the approximating bounded subsets
constructed in Proposition \ref{prop:bdSubsets} turns out to be convenient:

\begin{lemma}\label{lem:gnt}
Let $(X,\rho,\mu)$ be a doubling metric measure space and $f\in L^1_{\mathrm{loc}}(\mu)$.
Let $X_n$ be a sequence of bounded subsets of $X$ as constructed in Proposition
\ref{prop:bdSubsets},
and denote
\begin{equation*}
  B_n(z,t):=B(z,t)\cap X_n,\quad
  \mu_n:=\mu|_{X_n},\quad
  V_n(z,t):=\mu_n(B_n(z,t)).
\end{equation*}
Let
\begin{equation}\label{eq:gnt}
   g_t^n(z):=\frac{1}{V_n(z,t)^2}\iint_{B_n(z,t)\times B_n(z,t)}\frac{\Delta f}{t}.
\end{equation}
Then, for some $n_0$ that only depends on the space $X$, and for all $n\geq n_0$ and
$p\in[1,\infty)$,
\begin{equation}
 \Norm{\mathbf{1}_{X_n\times X_n}\Delta f}{L^p(\rho^{-ps}V^{-1})}^p
   \gtrsim \frac{ps}{2^{ps}}\int_0^1\frac{\Norm{g_t^n}{L^p(\mu_n)}^p}{t^{p(s-1)+1}}\,dt,
   \quad s\in(0,1).
\label{eq:weakLp>limsup}
\end{equation}
\end{lemma}

\begin{proof}
Starting with a generic $E$, we first note that
\begin{equation}\label{eq:intKtE}
\begin{split}
  \iint_{E\times E}\frac{(\Delta f)^p}{\rho^{ps}V}
  &=\iint_{E\times E}\frac{(\Delta f)^p}{V}ps\int_{\rho}^\infty\frac{dt}{t^{ps+1}} \\
  &=ps\int_0^\infty\Big[\iint_{\{\rho<t\}}
  \frac{\mathbf{1}_{E\times E}(\Delta f)^p}{V}\Big]\frac{dt}{t^{ps+1}} \\
  &=:ps\int_0^\infty\mathcal K_t^E\frac{dt}{t^{ps+1}}
  =\frac{ps}{2^{ps}}\int_0^\infty\mathcal K_{2t}^E\frac{dt}{t^{ps+1}}.
\end{split}
\end{equation}

From \cite[Lemma 3.1(i--ii)]{DMS} followed by H\"older's inequality, we infer that
\begin{equation}
\begin{split}
   \mathcal K_{2t}^E\gtrsim\mathcal S_t^E
   &:=\int_X\frac{1}{V(z,t)^2}\Big[\iint_{B(z,t)\times B(z,t)} \mathbf{1}_{E\times
   E}(\Delta f)^p \Big]\,d\mu(z) \\
   &\geq
   \int_X\Big[\frac{1}{V(z,t)^2} \iint_{B(z,t)\times B(z,t)} \mathbf{1}_{E\times E}\Delta f
   \Big]^p\, d\mu(z);
\end{split}\label{eq:DMS3.1}
\end{equation}
strictly speaking, \cite[Lemma 3.1]{DMS} is formulated for $E=X$ only,
but an inspection of the proof shows that the parts (i) and (ii) apply to a generic function $F$
on $X\times X$
in place of $\Delta f$, so we can in particular take $F=\mathbf{1}_{E\times E}\Delta f$.

Now, in \eqref{eq:DMS3.1}, we choose $E=X_n$, one of the bounded subsets constructed in
Proposition \ref{prop:bdSubsets}. If $z\in X_n$ and $t\leq\operatorname{diam}(X_n)$, then
\begin{equation}
   V(z,t)=V(z,\min\{t,\operatorname{diam}(X_n)\})\sim
   \mu(B(z,t)\cap X_n)\quad\text{by \eqref{eq:bdSubsets}}
   \label{eq:bdSubsets2}
\end{equation}
and hence
\begin{equation}
\begin{split}
   \mathcal K_{2t}^{X_n}
   &\overset{\text{\eqref{eq:DMS3.1}}}{\gtrsim}
   \int_{X_n}\Big[\frac{1}{V(z,t)^2}\iint_{B(z,t)\times B(z,t)}
   \mathbf{1}_{X_n\times X_n}\Delta f\Big]^p\,d\mu(z) \\
   &\overset{\text{\eqref{eq:bdSubsets2}}}{\sim}
    \int_{X_n}\Big[\frac{1}{\mu(B(z,t)\cap X_n)^2}
   \iint_{B(z,t)\times B(z,t)}
    \mathbf{1}_{X_n\times X_n}\Delta f\Big]^p\,d\mu(z)     \\
   & \overset{\text{\eqref{eq:gnt}}}{=}\int_{X_n} [t g^n_t(z)]^p\,d\mu_n(z).
\end{split}\label{eq:KtXn}
\end{equation}

Now recall that the sets $X_n$ from Proposition \ref{prop:bdSubsets} have the bounded
exhaustion property (Definition \ref{def:bdExhaustion}). If $X$ itself is bounded, it means
that $X_n=X$ for $n\geq n_0$.
In this case, \eqref{eq:bdSubsets2} and \eqref{eq:KtXn} hold for all $t\in(0,\infty)$. If
$X$ is unbounded, then $\operatorname{diam}(X_n)\to\operatorname{diam}(X)=\infty$. Thus, for
some $n_0$, it follows that all $n\geq n_0$ satisfy $\operatorname{diam}(X_n)\geq 1$ and
hence \eqref{eq:bdSubsets2} and \eqref{eq:KtXn} hold in particular for
$t\in(0,1)\subset(0,\operatorname{diam}(X_n))$. Thus, in either case, we conclude that
\begin{equation*}
  \mathcal K_{2t}^{X_n}
  \gtrsim (t\Norm{g_t^n}{L^p(\mu_n)})^p,\qquad t\in(0,1),\quad n\geq n_0.
\end{equation*}
Substituting this into \eqref{eq:intKtE}, we obtain
\begin{equation*}
\begin{split}
  \Norm{\mathbf{1}_{X_n\times X_n}\Delta f}{L^p(\rho^{-ps}V^{-1})}^p
  &\geq\frac{ps}{2^{ps}}\int_0^1\mathcal K_{2t}^{X_n}\frac{dt}{t^{ps+1}} \\
  &\gtrsim\frac{ps}{2^{ps}}\int_0^1\Norm{g_t^n}{L^p(\mu_n)}^p \frac{dt}{t^{p(s-1)+1}}
\end{split}
\end{equation*}
for $n\geq n_0$, and this completes the proof.
\end{proof}

\begin{lemma}\label{lem:2liminf}
Let $h:(0,1)\to[0,\infty]$ be a measurable function and $p\in[1,\infty)$. Then
\begin{equation*}
  \liminf_{s\nearrow 1}(1-s)ps\int_0^1\frac{h(t)\,dt}{t^{p(s-1)+1}}
  \geq \liminf_{t\searrow 0}h(t).
\end{equation*}
\end{lemma}

\begin{proof}
Let us denote the $\liminf$ on the right by $L$. The claim is trivial if $L=0$, so suppose
that $L\in(0,\infty]$.
By the definition of the $\liminf$, given $L'\in(0,L)$, we can find $\delta\in(0,1)$ such that
$h(t)>L'$ for all $t\in(0,\delta)$. Hence
\begin{equation*}
  \int_0^1\frac{h(t)\, dt}{t^{p(s-1)+1}}
  \geq \int_0^\delta L' t^{p(1-s)-1}\,dt
  =L'\frac{\delta^{p(1-s)}}{p(1-s)}
\end{equation*}
and
\begin{equation*}
  \liminf_{s\nearrow 1}(1-s)ps\int_0^1\frac{h(t)\,dt}{t^{p(s-1)+1}}
  \geq \liminf_{s\nearrow 1} sL'\delta^{p(1-s)}=L'.
\end{equation*}
The proof is finished by taking the limit $L'\to L$.
\end{proof}

\begin{lemma}\label{lem:weakLp>liminf}
Let $(X,\rho,\mu)$ be a doubling metric measure space.
Let $\gamma\in\mathbb R$, $p\in(1,\infty)$, and let
$f\in L^1_{\mathrm{loc}}(\mu)$ be such that
\begin{equation}\label{eq:L=supE}
  L:=
  \sup_E\liminf_{s\nearrow 1}(1-s)^{\frac1p}
  \Norm{\mathbf{1}_{E\times E}\Delta f}{L^p(\rho^{-ps}V^{-1})}<\infty,
\end{equation}
where the supremum is taken over all bounded measurable sets $E\subset X$.
Let the assumptions and notation of Lemma \ref{lem:gnt} be in force.
Then, for every $n\geq n_0$, there are a sequence $t_k\searrow 0$ and $g^n\in L^p(\mu_n)$ such
that the functions \eqref{eq:gnt} satisfy
\begin{equation*}
  g^n_{t_k}\rightharpoonup g^n\quad\text{weakly in }L^p(\mu_n)
\end{equation*}
and
\begin{equation*}
  \Norm{g^n}{L^p(\mu_n)}\lesssim L.
\end{equation*}
\end{lemma}

\begin{proof}
Combining several results from above, we find that
\begin{equation*}
\begin{split}
  \liminf_{t\searrow 0}\Norm{g^n_t}{L^p(\mu)}^p
  &\leq\liminf_{s\nearrow 1}(1-s)ps\int_0^1\frac{\Norm{g^n_t}{L^p(\mu)}^p}{t^{p(s-1)+1}}\,dt
  \quad\text{by Lemma \ref{lem:2liminf}} \\
&  \lesssim \liminf_{s\nearrow 1}(1-s)
   \Norm{ \mathbf{1}_{X_n\times X_n}\Delta f }{ L^p(\rho^{-ps}V^{-1})}^p
  \quad\text{by \eqref{eq:weakLp>limsup}} \\
& \leq L^p\quad\text{by \eqref{eq:L=supE}}.
\end{split}
\end{equation*}

From the definition of $\liminf$, we deduce the existence of a sequence $t_k\searrow 0$ with
\begin{equation*}
   \Norm{g^n_{t_k}}{L^p(\mu)}
 \lesssim L.
\end{equation*}
The bounded sequence in the reflexive Banach space $L^p(\mu_n)$ (this is where we use
$p\in(1,\infty)$) has a weakly convergent subsequence $g^n_{t_k}\rightharpoonup g^n$, which
we continue to denote by the same symbol. The upper bound on $\Norm{g^n}{L^p(\mu_n)}$ is then a
standard property of weak limits.
This completes the proof of Lemma \ref{lem:weakLp>liminf}.
\end{proof}

The functions of the form $g^n_{t_k}$ are upper gradients, up to a certain scale, of
mollifications of $f$; we will apply the following lemma with $X_n$ in place of $X$.

\begin{lemma}[\cite{DMS}]\label{lem:DMS1869}
Let $(X,\rho,\mu)$ be a doubling metric measure space.
Given $f\in L^1_{\mathrm{loc}}(X)$, consider the functions
\begin{equation}\label{eq:ft-gt}
  f_t(x):=\fint_{B(x,t)}f,\qquad
  g_t(x):=\fint_{B(x,t)}\fint_{B(x,t)}\frac{\Delta f}{t},\qquad t>0.
\end{equation}
Then, for a constant $c$ that only depends on the doubling properties of $X$,
\begin{enumerate}[\rm(i)]
  \item\label{it:ft->f} $f_t\to f$ pointwise $\mu$-a.e.\ as $t\searrow 0$;
  \item\label{it:gt-uppergrad} $c\cdot g_{2t}$ is an upper gradient of $f_t$ up to scale
  $\frac12 t$.
\end{enumerate}
\end{lemma}

\begin{proof}
Lemma~\ref{lem:DMS1869}\eqref{it:ft->f} is the well-known Lebesgue differentiation theorem;
see e.g. \cite[Theorem 1.8]{Heinonen:book}. The proof of
Lemma~\ref{lem:DMS1869}\eqref{it:gt-uppergrad} is contained in \cite[page 1869]{DMS}.
This completes the proof of Lemma \ref{lem:DMS1869}.
\end{proof}

Now, we are ready to prove \eqref{eq:DMS-lower-v2}.
\begin{proof}[Proof of \eqref{eq:DMS-lower-v2}]
We assume that the right-hand side of \eqref{eq:DMS-lower-v2} is finite, since otherwise
there is nothing to prove.
We first work under the stronger assumption that $f\in L^\infty(\mu)$. In this case, $f_N=f$
for all sufficiently large $N$, and we are reduced to proving that
\begin{equation}\label{eq:DMS-lower-v3}
  \Norm{f}{\dot W^{1,p}(\mu)}
  \lesssim\sup_E\liminf_{s\nearrow 1}(1-s)^{\frac1p}
  \Norm{\mathbf{1}_{E\times E}\Delta f}{L^p(\rho^{-ps}V^{-1})}=:L.
\end{equation}

Let $X_n$ be a sequence of bounded subsets of $X$ as constructed in Proposition
\ref{prop:bdSubsets}, with measure $\mu_n:=\mu|_{X_n}$ and balls $B_n(x,t):=B(x,t)\cap X_n$
of volume $V_n(x,t):=\mu_n(B_n(x,t))$.
With functions $f^n:=f|_{X_n}$, we define
\begin{equation}\label{eq:fnt-gnt}
  f^n_t:=\fint_{B_n(x,t)}f^n,\quad
  g^n_t:=\fint_{B_n(x,t)}\fint_{B_n(x,t)}\frac{\Delta f^n}{t},\quad t>0,
\end{equation}
i.e., the same as \eqref{eq:ft-gt} with $f^n$ in place of $f$. Note that the definition of
$g^n_t$ in \eqref{eq:fnt-gnt} is consistent with \eqref{eq:gnt}. Combining the results of
Lemmas \ref{lem:weakLp>liminf} and \ref{lem:DMS1869}, we then
conclude that, for each $n\geq n_0$,
there are a sequence $t_k\searrow 0$ and a function $g^n\in L^p(\mu_n)$ such that
\begin{enumerate}[\rm(i)]
  \item $c\cdot g_{t_k}^n$ is an upper gradient of $f^n_{\frac12 t_k}$ up to scale $\frac14
  t_k$, where $\frac14 t_k\searrow 0$;
  \item $f^n_{\frac12 t_k}\to f^n$ pointwise $\mu_n$-a.e., and $g^n_{t_k}\rightharpoonup
  g^n$ weakly in $L^p(\mu_n)$.
\end{enumerate}
These are precisely the assumptions of Theorem \ref{thm:ACDM27},
which guarantees that $c\cdot g^n\in L^p(\mu_n)$ is a $p$-weak upper gradient of $f^n=f|_{X_n}$ in
the sense of test plans, for each $n\geq n_0$. Thus, we are in a position to apply
Proposition \ref{prop:gluing-ug}, which guarantees that the original function $f$ has a
$p$-weak upper gradient $g$ in the sense of test plans such that
\begin{equation*}
\begin{split}
  \Norm{g}{L^p(\mu)}  &\leq\sup_{n\geq n_0}\Norm{g^n}{L^p(\mu_n)}
  \quad\text{by Proposition \ref{prop:gluing-ug}} \\
  &\lesssim L
  \quad\text{by Lemma \ref{lem:weakLp>liminf}}.
\end{split}
\end{equation*}
By Definition \ref{def:dotW1p}, this means that $f\in \dot W^{1,p}(\mu)$ and
\begin{equation*}
  \Norm{f}{\dot W^{1,p}(\mu)}\lesssim L,
\end{equation*}
which completes the proof under the additional assumption that $f\in L^\infty(\mu)$.

Let us then consider a general $f\in L^1_{\mathrm{loc}}(\mu)$. Its truncations satisfy
$f_N\in L^\infty(\mu)$, a case that we already considered, and hence we can apply
\eqref{eq:DMS-lower-v3} to $f_N$ in place of $f$:
\begin{equation*}
  \Norm{f_N}{\dot W^{1,p}(\mu)}
  \lesssim\sup_E\liminf_{s\nearrow 1}(1-s)^{\frac1p}
  \Norm{\mathbf{1}_{E\times E}\Delta f_N}{L^p(\rho^{-ps}V^{-1})}=:L_N.
\end{equation*}
We can now apply Proposition \ref{prop:truncW1p} to conclude that
\begin{equation*}
  \Norm{f}{\dot W^{1,p}(\mu)}
  =\sup_{N\in\mathbb N}\Norm{f_N}{\dot W^{1,p}(\mu)}
  \lesssim\sup_{N\in\mathbb N}L_N
  =\sup_{E,N}\liminf_{s\nearrow 1}(1-s)^{\frac1p}
  \Norm{\mathbf{1}_{E\times E}\Delta f_N}{L^p(\rho^{-ps}V^{-1})}.
\end{equation*}
Taking the $p$-th power gives \eqref{eq:DMS-lower-v2} and concludes the proof
of the said inequality.
\end{proof}

We have hence proved the lower bound \eqref{eq:DMS-lower-v2} of Proposition \ref{prop:DMS}.
The upper bound \eqref{eq:DMS-upper-v2} will be proved in the following section.

\section{Nguyen-type characterizations}\label{sec:Nguyen}

In this section, we complete the proof of Proposition \ref{prop:DMS} and also use it to
obtain another characterization of the Sobolev space, combining some elements of both
Theorem \ref{thm:Frank-type} and Proposition \ref{prop:DMS}: as in Theorem
\ref{thm:Frank-type}, the characterization is in terms of an $L^{p,\infty}$ norm; as in
Proposition \ref{prop:DMS}, it involves the differences $\Delta f(x,y)=\abs{f(x)-f(y)}$.

\begin{theorem}\label{thm:Nguyen-type}
Let $p\in(1,\infty)$, $\gamma\in\mathbb{R}\setminus\{0\}$,
and $(X,\rho,\mu)$ be a complete doubling metric measure space.
If $f\in L^1_{\mathrm{loc}}(\mu)$, then
\begin{subequations}\label{eq:Nguyen-type}
\begin{align}
  \Norm{f}{\dot W^{1,p}(\mu)}
  &\lesssim\Norm{\rho^{-1-\gamma}\Delta f}{L^{p,\infty}(\rho^{p\gamma}V^{-1})}
    \label{eq:Nguyen-lower} \\
  &\lesssim\sup_{\theta\in(0,1]}
    \Norm{\rho^{-\gamma} f_\theta^*}{L^{p,\infty}(\rho^{\gamma p}V^{-1})}
   \lesssim\Norm{f}{\dot M^{1,p}(\mu)}.
   \label{eq:Nguyen-upper}
\end{align}
\end{subequations}
If, in addition, $X$ is a $p$-Poincar\'e space, then all four quantities in
\eqref{eq:Nguyen-type} are comparable; thus, $f\in\dot M^{1,p}(\mu)$ if and only if either of
the two quantities in the middle of \eqref{eq:Nguyen-type} is finite.
\end{theorem}

\begin{remark}\label{rem:Nguyen-finish}
Under the additional assumptions of the last claim of Theorem \ref{thm:Nguyen-type}, we have
the equivalence $\Norm{f}{\dot W^{1,p}(\mu)}\sim\Norm{f}{\dot M^{1,p}(\mu)}$ by Proposition
\ref{prop:MWE}; hence this last claim will immediately follow from \eqref{eq:Nguyen-type}
and the said equivalence, and we only need to prove \eqref{eq:Nguyen-type}.
\end{remark}

\begin{remark}
Except for the intermediate step involving $f_\theta^*$, case $X=\mathbb R^d$ is due to
Nguyen \cite[Theorem 2]{N:06} for $\gamma=-1$ and due to Brezis et al.\ \cite[Theorem
1.3]{BSVY} for all $\gamma\in\mathbb R\setminus\{0\}$ as in Theorem \ref{thm:Nguyen-type}.
For general $p$-Poincar\'e spaces but under stronger a priori assumptions on $f$, the case
$\gamma=-1$ was obtained by Di Marino and Squassina \cite[Theorem 1.5]{DMS} for $f\in
L^p(\mu)$ (sufficient for characterizing the inhomogeneous Sobolev space $W^{1,p}(\mu)$, but
not $\dot W^{1,p}(\mu)$) and the case of general $\gamma\in\mathbb R\setminus\{0\}$ by Dai
et al. \cite[Corollary 1.14]{DLYYZ} for $f\in\operatorname{Lip}(X)$.
\end{remark}

\begin{remark}
Although a norm equivalence is only achieved in a $p$-Poincar\'e space,
it is interesting to note that a comparison involving two different Sobolev norms is
achieved
in a much greater generality. This shares some of spirit of the recent results of
\cite{HXZ},
where \cite[Theorem 1.3]{HXZ} sandwiches a weak-type quantity like those
in the middle of \eqref{eq:Nguyen-type} between two different Sobolev type norms,
while \cite[Corollary 1.5]{HXZ} achieves an equivalence under additional assumptions on the
space.
However, in contrast to Theorem \ref{thm:Nguyen-type},
which provides a characterization of the whole {\em space} $\dot M^{1,p}(\mu)$,
\cite{HXZ} only compares the {\em norms} of Lipschitz functions of bounded support.
\end{remark}

Another variant of Theorem \ref{thm:Nguyen-type} in general metric measure spaces is as
follows. Although we do not assume any direct functional relationship (like Ahlfors
regularity) between  the distance $\rho$ and the volume $V$, we can nevertheless replace the
pair of distance factors $(\rho^{-\gamma},\rho^{p\gamma})$ on the function and the weight in
Theorem \ref{thm:Nguyen-type} by a corresponding pair of volume factors
$(V^{-\gamma},V^{p\gamma})$ in Theorem \ref{thm-Nguyen-type-v2}.

\begin{theorem}\label{thm-Nguyen-type-v2}
Let $p\in(1,\infty)$, $\gamma\in\mathbb{R}\setminus\{0\}$,
and $(X,\rho,\mu)$ be a complete and
doubling metric measure space. If $f\in L^1_{\mathrm{loc}}(\mu)$, then
\begin{subequations}\label{eq:Nguyen-type-v2}
\begin{align}
  \Norm{f}{\dot W^{1,p}(\mu)}
  &\lesssim\Norm{\rho^{-1}V^{-\gamma}\Delta f}{L^{p,\infty}(V^{\gamma p-1})}
    \label{eq:Nguyen-lower-v2} \\
  &\lesssim\sup_{\theta\in(0,1]}
    \Norm{V^{-\gamma} f_\theta^*}{L^{p,\infty}(V^{\gamma p-1})}
   \lesssim\Norm{f}{\dot M^{1,p}(\mu)}.
   \label{eq:Nguyen-upper-v2}
\end{align}
\end{subequations}
If, in addition, $X$ is a $p$-Poincar\'e space,
then all four quantities in \eqref{eq:Nguyen-type-v2} are comparable;
thus, $f\in\dot M^{1,p}(\mu)$ if and only if either of the two quantities
in the middle of \eqref{eq:Nguyen-type-v2} is finite.
\end{theorem}

\begin{remark}
As in Theorem \ref{thm:Nguyen-type} (cf.\ Remark \ref{rem:Nguyen-finish}), the last claim
will immediately follow from \eqref{eq:Nguyen-type-v2} and the equivalence of the two
Sobolev norms (Proposition \ref{prop:MWE}).
\end{remark}

\begin{remark}
In the Ahlfors regular space $\mathbb R^d$, Theorem \ref{thm-Nguyen-type-v2} reduces to
Theorem \ref{thm:Nguyen-type} and therefore has the same history. In general, the two
theorems are different.
As an a priori estimate for test functions $f\in\operatorname{Lip}(X)$, the norm equivalence
\eqref{eq:Nguyen-type-v2} is previously due to Dai et al. \cite[Theorem 1.8 and Corollary
1.9]{DLYYZ}.
\end{remark}

\begin{proof}[Proof of the upper bound \eqref{eq:Nguyen-upper}]
For the first inequality, we use the following estimate from \cite[Lemma 2.1]{DLYYZ}:
\begin{equation}\label{eq:swap}
  \frac{\Delta
  f}{\rho}\lesssim\sum_{j=0}^\infty2^{-j}(f_{2^{-j}}^*+f_{2^{-j}}^*\circ\operatorname{swap}),\qquad
  \operatorname{swap}(x,y):=(y,x),
\end{equation}
where a simple change of variables shows that
\begin{equation}\label{eq:rid-of-swap}
  \Norm{\rho^{-\gamma}(f_{2^{-j}}^*\circ\operatorname{swap})}{
     L^{p,\infty}(\rho^{p\gamma}V^{-1})}
   \sim\Norm{\rho^{-\gamma}f_{2^{-j}}^*}{
     L^{p,\infty}(\rho^{p\gamma}V^{-1})},
\end{equation}
since $\rho\circ\operatorname{swap}=\rho$ and $V\circ\operatorname{swap}\sim V$ by doubling.

Using an $r$-norm equivalent to $\Norm{\ }{L^{p,\infty}}$ with some $r\in(0,1]$, as provided
by the Aoki--Rolewicz theorem (actually, for $p\in(1,\infty)$, we can take $r=1$), we obtain
from the $r$-triangle inequality that
\begin{equation}\label{eq:ar}
\begin{split}
  \BNorm{\rho^{-\gamma}\frac{\Delta f}{\rho}}{L^{p,\infty}(\rho^{p\gamma}V^{-1})}
  &\lesssim
  \Big[\sum_{j=0}^\infty 2^{-jr}\Norm{\rho^{-\gamma}f_{2^{-j}}^*}{
     L^{p,\infty}(\rho^{p\gamma}V^{-1})}^r \Big]^{\frac1r} \\
  &\leq\sup_{\theta\in(0,1]}\Norm{\rho^{-\gamma}f_{\theta}^*}{
     L^{p,\infty}(\rho^{p\gamma}V^{-1})}\Big(\sum_{j=0}^\infty 2^{-jr}\Big)^{\frac1r},
\end{split}
\end{equation}
which proves the first inequality in \eqref{eq:Nguyen-upper}.

For the second inequality, we recall from Lemma \ref{lem:Msharpf<Mh} that
$f_\theta^*(x,y)\lesssim Mh(x)$, when $h$ is a Haj\l{}asz upper gradient of $f$. Thus,
\begin{equation*}
\begin{split}
  \lambda^p &\iint_{\{\rho^{-\gamma}f_\theta^*>\lambda\}}\rho^{\gamma p}V^{-1} \\
  &\leq\lambda^p\int_X \int_{ \{ y: \rho(x,y)^{\gamma}<c Mh(x) / \lambda \}}
  \rho(x,y)^{\gamma p}V(x,y)^{-1}\,d\mu(y)\,d\mu(x) \\
  &\lesssim\lambda^p\int_X\Big[\frac{Mh(x)}{\lambda}\Big]^p \,d\mu(x)
  =\Norm{Mh}{L^p(\mu)}^p\lesssim\Norm{h}{L^p(\mu)}^p
\end{split}
\end{equation*}
by a straightforward estimate using a dyadic annular decomposition in the first $\lesssim$,
and
by the Hardy--Littlewood maximal theorem in the last step. Taking the supremum over
$\lambda>0$ and $\theta\in(0,1]$ and then the infimum over all Haj\l{}asz upper gradients $h$ of
$f$ proves the second estimate in \eqref{eq:Nguyen-upper} and completes the proof of
\eqref{eq:Nguyen-upper}.
\end{proof}

\begin{proof}[Proof of the upper bound \eqref{eq:Nguyen-upper-v2}]
The first inequality of \eqref{eq:Nguyen-upper-v2} follows from an argument
similar to \eqref{eq:ar} with $\rho^{-\gamma}$ and $\rho^{p\gamma}V^{-1}$ replaced,
respectively, by $V^{-\gamma}$ and $V^{p\gamma-1}$.

For the second inequality of \eqref{eq:Nguyen-upper-v2}, we need the following
elementary estimates in doubling metric measure spaces, proved in \cite[Lemma 2.8 and
p.~33]{DLYYZ:JFA} and stated as below in \cite[Lemma 3.2]{DLYYZ}:
for any given $\alpha\in\mathbb{R}\setminus\{0\}$ and
for any $x\in X$ and $R\in(0,\infty)$,
\begin{equation}\label{eq:ele}
\begin{split}
\int_{\{y\in X:[V(x,y)]^{-\alpha}>R\}}[V(x,y)]^{\alpha-1}\,d\mu(y)\lesssim R^{-1},
\end{split}
\end{equation}
where the implicit positive constant depends only on $\alpha$ and the doubling constant
of $X$. From this and Lemma \ref{lem:Msharpf<Mh}, it follows that
\begin{equation}
\begin{split}
\lambda^p &\iint_{\{V^{-\gamma}f_\theta^*>\lambda\}}V^{\gamma p-1} \\
&\leq\lambda^p\int_X \int_{ \{ y:c [V(x,y)]^{-\gamma}Mh(x) >\lambda \}}
V(x,y)^{\gamma p-1}\,d\mu(y)\,d\mu(x) \\
&\lesssim\lambda^p\int_X\Big[\frac{Mh(x)}{\lambda}\Big]^p\, d\mu(x)
=\Norm{Mh}{L^p(\mu)}^p\lesssim\Norm{h}{L^p(\mu)}^p,
\end{split}
\end{equation}
where $h$ is a Haj\l{}asz upper gradient of $f$ and the last step
used the Hardy--Littlewood maximal theorem.
Taking the supremum over $\lambda>0$ and $\theta\in(0,1]$
and then the infimum over all Haj\l{}asz upper gradients $h$ of $f$
proves the second estimate in \eqref{eq:Nguyen-upper-v2},
which completes the proof of \eqref{eq:Nguyen-upper-v2}.
\end{proof}

We then turn to the more challenging lower bounds \eqref{eq:Nguyen-lower}
and \eqref{eq:Nguyen-lower-v2}.
We adapt the Euclidean argument of \cite[Section 4]{BSVY},
where these estimates in the case of $X=\mathbb R^d$ were reduced
to the lower bound in a characterization of $\dot W^{1,p}(\mathbb R^d)$ from \cite{BBM}.
Our substitute for this is Proposition \ref{prop:DMS}, whose lower bound
\eqref{eq:DMS-lower-v2} we already proved in Section \ref{sec:DMS}.
While reducing the lower bounds in Theorems \ref{thm:Nguyen-type} and
\ref{thm-Nguyen-type-v2} to the lower  bound \eqref{eq:DMS-lower-v2}  of Proposition
\ref{prop:DMS}, we will at the same time obtain a proof of the upper bound
\eqref{eq:DMS-upper-v2}  of Proposition \ref{prop:DMS} that we postponed until now.

Let us first see how to control the middle terms of \eqref{eq:DMS-lower} by the right-hand
side of \eqref{eq:Nguyen-lower}. In the relevant application, we will take $F=\Delta f_N$ in
the lemma below:

\begin{lemma}\label{lem:BSVY963}
Let $(X,\rho,\mu)$ be a doubling metric measure space.
Let $p\in[1,\infty)$, $\gamma\in\mathbb{R}$,
$F\in L^\infty(\mu\times\mu)$, and $E\subset X$ be a bounded set. Then
\begin{equation*}
  \limsup_{s\nearrow 1}(1-s)^{\frac1p}\Norm{\mathbf{1}_{E\times
  E}F}{L^{p}(\rho^{-ps}V^{-1})}
  \lesssim\Norm{\rho^{-1-\gamma}F}{L^{p,\infty}(\rho^{p\gamma}V^{-1})},
\end{equation*}
where the implicit positive constant depends only on $p$ and $\gamma$.
\end{lemma}

\begin{proof}
Adapting the argument of \cite[page 963]{BSVY}, we will make use of the Lorentz space
duality \cite[(4-2)]{BSVY}
\begin{equation}\label{eq:BSVY4-2}
  \int GH\,d\nu
  \leq q'\Norm{G}{L^{q,\infty}(\nu)}\Norm{H}{L^{q',1}(\nu)},\qquad q\in(1,\infty).
\end{equation}
We consider the following two cases for $\gamma\in\mathbb{R}$.

\emph{Case 1: $\gamma\in(0,\infty)$}. Given $s\in(0,1)$, we can write
\begin{equation}
\begin{split}
  \iint_{E\times E} &
     \frac{|F|^p}{\rho^{ps}V}
  =\iint_{E\times E}
     \frac{|F|^p}{\rho^{p(s+\gamma)}}\frac{\rho^{p\gamma}}{V}    \\
  &=\iint_{X\times X}\Big(\frac{|F|}{\rho^{1+\gamma}}\Big)^{p\frac{s+\gamma}{1+\gamma}}
  \cdot\mathbf{1}_{E\times E}|F|^{p\frac{1-s}{1+\gamma}}
  \cdot\frac{\rho^{p\gamma}}{V} \\
  &\leq\frac{1+\gamma}{1-s}
  \BNorm{\Big(\frac{|F|}{\rho^{1+\gamma}}\Big)^{p\frac{s+\gamma}{1+\gamma} }}{
  L^{\frac{1+\gamma}{s+\gamma},\infty} (\rho^{p\gamma}V^{-1})}
  \BNorm{\mathbf{1}_{E\times E}|F|^{p\frac{1-s}{1+\gamma}} }{
  L^{\frac{1+\gamma}{1-s},1}( \rho^{p\gamma} V^{-1} ) } \\
  &\qquad\text{by \eqref{eq:BSVY4-2} with }q=\frac{1+\gamma}{s+\gamma}\in(1,\infty) \\
  &\leq\frac{1+\gamma}{1-s}
  \BNorm{ \frac{F}{\rho^{1+\gamma}} }{   L^{p,\infty} (\rho^{p\gamma}V^{-1})
  }^{p\frac{s+\gamma}{1+\gamma}}
  \Norm{F}{L^\infty(\mu\times\mu)}^{p\frac{1-s}{1+\gamma}}
  \Big(\iint_{E\times E}\rho^{p\gamma}V^{-1}\Big)^{\frac{1-s}{1+\gamma}},
\end{split}\label{eq:BSVY963}
\end{equation}
where we used $\Norm{\mathbf{1}_A}{L^{r,1}(\nu)}=\nu(A)^{\frac1r}$ in the last step.

Concerning the last integral in \eqref{eq:BSVY963}, we make the following observations.
Since $E$ is a bounded set, we have $E\subset B(x_0,R)$ for some $x_0\in X$ and $R>0$. Then,
for each $x,y\in E$, we have $y\in B(x,2R)$. Hence,
\begin{equation*}
\begin{split}
  \int_E\frac{\rho(x,y)^{p\gamma}}{V(x,y)}\,d\mu(y)
  &\leq\sum_{k=0}^\infty\int_{2^{-k}R\leq
  \rho(x,y)<2^{1-k}R}\frac{(2^{1-k}R)^{p\gamma}}{V(x,2^{-k}R)}\,d\mu(y) \\
  &\leq\sum_{k=0}^\infty\frac{2^{(1-k)p\gamma}R^{p\gamma}}{V(x,2^{-k}R)}V(x,2^{1-k}R)
  \lesssim\sum_{k=0}^\infty 2^{-kp\gamma}R^{p\gamma}\lesssim R^{p\gamma},
\end{split}
\end{equation*}
and hence
\begin{equation*}
  \iint_{E\times E}\frac{\rho(x,y)^{p\gamma}}{V(x,y)}\,d\mu(y)\,d\mu(x)
  \lesssim\int_{B(x_0,R)}R^{p\gamma}\,d\mu(x)=R^{p\gamma}V(x_0,R)<\infty.
\end{equation*}
Thus, the last two terms on the right of \eqref{eq:BSVY963} are finite. By inspection of the
exponents, these terms converge to $1$ as $s\nearrow 1$. Hence,
\begin{equation*}
  \limsup_{s\nearrow 1}(1-s)\iint_{E\times E}
     \frac{|F|^p}{\rho^{ps}V}
   \leq(1+\gamma)\BNorm{\frac{F}{\rho^{1+\gamma}}}{L^{p,\infty}(\rho^{p\gamma}V^{-1})}^p,
\end{equation*}
which completes the proof. 

\emph{Case 2: $\gamma\in(-\infty,0]$}. We fix $s\in (0,1)$ and write
\begin{align*}
\iint_{E\times E} \frac{|F|^p}{\rho^{ps}V}
=\iint_{E\times E} \Big(\frac{|F|}{\rho^{1+\gamma}}
\Big)^{\frac{p(1+s)}{2}}\Big(\mathbf{1}_{E\times E}|F|
\rho^{1-\gamma}\Big)^{\frac{p(1-s)}{2}}\frac{\rho^{p\gamma}}{V}.
\end{align*}
With the Lorentz space duality \eqref{eq:BSVY4-2}, this is further estimated by
\begin{align}
\le\frac{2}{1-s}\Big\|\Big(\frac{|F|}{\rho^{1+\gamma}}
\Big)^{\frac{p(1+s)}{2}}\Big\|_{L^{\frac{2}{1+s},\infty}(\rho^{p\gamma}V^{-1})}
 \Big\|\Big(\mathbf{1}_{E\times E}|F|\rho^{1-\gamma}
 \Big)^{\frac{p(1-s)}{2}}\Big\|_{L^{\frac{2}{1-s},1}(\rho^{p\gamma}V^{-1})}.
\end{align}
Note that
\begin{align}
\Big\|\Big(\frac{|F|}{\rho^{1+\gamma}}
\Big)^{\frac{p(1+s)}{2}}\Big\|_{L^{\frac{2}{1+s},\infty}(\rho^{p\gamma}V^{-1})}
=\Big\|\frac{F}{\rho^{1+\gamma}}
\Big\|_{L^{p,\infty}(\rho^{p\gamma}V^{-1})}^{\frac{p(1+s)}{2}}
\end{align}
and
\begin{align}\label{eq:term2}
\Big\|\Big(\mathbf{1}_{E\times E}|F|\rho^{1-\gamma}
\Big)^{\frac{p(1-s)}{2}}\Big\|_{L^{\frac{2}{1-s},1}(\rho^{p\gamma}V^{-1})} 
\le \|F\|_{L^\infty(\mu\times\mu)}^{\frac{p(1-s)}{2}}
\Big\|\mathbf{1}_{E\times E}\rho^{t}
\Big\|_{L^{\frac{2}{1-s},1}(\rho^{p\gamma}V^{-1})},
\end{align}
where
\begin{equation*}
  t:=\frac{p(1-s)(1-\gamma)}{2}>0
\end{equation*}
since $\gamma\in(-\infty,0]$ and $s\in(0,1)$. If we can show that $\limsup_{s\nearrow 1}$ of
the last term on the right of \eqref{eq:term2} is bounded by a constant, then it follows
from the previous estimates that
\begin{equation*}
  \limsup_{s\nearrow 1}\, (1-s)\iint_{E\times E} \frac{\abs{F}^p}{\rho^{ps}V}\lesssim
  \BNorm{ \frac{F}{ \rho^{1+\gamma} } }{L^{p,\infty}(\rho^{p\gamma}V^{-1})}^p,
\end{equation*}
which completes the proof. So we turn to this remaining task.

Recall that $E\subset B(x_0,R)$ and hence
$\rho^{t}\le (2R)^{t}$ on $E\times E$. Thus,
\begin{align}\label{eq:need-bounded}
\Big\|\mathbf{1}_{E\times E}\rho^{t}
\Big\|_{L^{\frac{2}{1-s},1}(\rho^{p\gamma}V^{-1})} 
=
\int_{0}^{(2R)^{t}}\Big[\iint_{\{(x,y)\in E\times E:
\rho(x,y)^{t}>\lambda\}}\frac{\rho^{p\gamma}}{V}
\Big]^{\frac{1-s}{2}}\,d\lambda.
\end{align}
Given $\lambda\in (0,\infty)$, choose $k_\lambda\in\mathbb{Z}$
such that $2^{-k_\lambda}R\le \lambda^{\frac{1}{t}}<
2^{-k_{\lambda}+1}R$.

For any given $\lambda\in(0,(2R)^{t})$
and $x\in E$,
\begin{equation}\label{eq:est0}
\begin{split}
&\int_{\{y\in
E:\rho(x,y)>\lambda^{\frac{1}{t}}\}}\frac{\rho(x,y)^{p\gamma}}{V(x,y)}\,d\mu(y)\\
&\quad\le\sum_{k=0}^{k_\lambda}\int_{2^{-k}R\le
\rho(x,y)<2^{1-k}R}\frac{(2^{-k}R)^{p\gamma}}{
V(x,2^{-k}R)}\,d\mu(y) \\
&\quad\lesssim \sum_{k=0}^{k_\lambda}(2^{-k}R)^{p\gamma}
\lesssim
\begin{cases}
 (2^{-k_\lambda}R)^{p\gamma}
\sim\lambda^{\frac{p\gamma}{t}} & \text{if }\gamma<0, \\
\displaystyle
1+k_\lambda\leq 1+\frac{1}{t}\log_2\frac{(2R)^t}{\lambda} & \text{if }\gamma=0. \end{cases}
\end{split}
\end{equation}

\emph{Subcase 2a: $\gamma\in(-\infty,0)$}.
If $\gamma<0$, then \eqref{eq:need-bounded} and \eqref{eq:est0} further imply that
\begin{equation*}
\Big\|\mathbf{1}_{E\times E}\rho^{t}
\Big\|_{L^{\frac{2}{1-s},1}(\rho^{p\gamma}V^{-1})}
\le \Big[C V(x_0,R)\Big]^{\frac{1-s}{2}} \int_{0}^{(2R)^{t}}
\lambda^{\frac{\gamma}{1-\gamma}}\,d\lambda,
\end{equation*}
where
\begin{equation*}
  \int_{0}^{(2R)^{t}}\lambda^{\frac{\gamma}{1-\gamma}}\,d\lambda
  =(1-\gamma)((2R)^t)^{\frac{1}{1-\gamma}}
  =(1-\gamma)(2R)^{\frac{p(1-s)}{2}}.
\end{equation*}
Then clearly
\begin{equation*}
  \limsup_{s\nearrow 1}\Big\|\mathbf{1}_{E\times E}\rho^{t}
\Big\|_{L^{\frac{2}{1-s},1}(\rho^{p\gamma}V^{-1})}
\leq 1-\gamma,
\end{equation*}
as we wanted.

\emph{Subcase 2b: $\gamma=0$.}
If $\gamma=0$, then $t=\frac{p(1-s)}{2}$, and \eqref{eq:need-bounded} and \eqref{eq:est0}
imply that
\begin{align*}
&\Big\|\mathbf{1}_{E\times E}\rho^{\frac{p(1-s)}{2}}
\Big\|_{L^{\frac{2}{1-s},1}(V^{-1})}\\
&\quad\le \int_{0}^{(2R)^{\frac{p(1-s)}{2}}}\Big[
1+\frac{2}{p(1-s)}\log_2 \frac{(2R)^{\frac{p(1-s)}{2}}}{\lambda}\Big]^{\frac{1-s}{2}}
\,d\lambda[CV(x_0,R)]^{\frac{1-s}{2}}\\
&\quad=(2R)^{\frac{p(1-s)}{2}}\int_{0}^{1}\Big[
1+\frac{2}{p(1-s)}\log_2 \frac{1}{\lambda}\Big]^{\frac{1-s}{2}}
\,d\lambda[CV(x_0,R)]^{\frac{1-s}{2}}.
\end{align*}
By dominated convergence, we obtain
\begin{align*}
\limsup_{s\nearrow 1}\Big\|\mathbf{1}_{E\times E}\rho^{\frac{p(1-s)}{2}}
\Big\|_{L^{\frac{2}{1-s},1}(V^{-1})}\le 1.
\end{align*}
This is the required limit, which completes the proof of Lemma \ref{lem:BSVY963}.
\end{proof}

A counterpart of Lemma \ref{lem:BSVY963} related to \eqref{eq:Nguyen-type-v2}
is as follows.

\begin{lemma}\label{lem:BSVY105}
Let $(X,\rho,\mu)$ be a doubling metric measure space.
Let $p\in[1,\infty)$, $\gamma\in\mathbb{R}$,
$F\in L^\infty(\mu\times\mu)$, and $E\subset X$ be a bounded measurable set. Then
\begin{equation*}
  \limsup_{s\nearrow 1}(1-s)^{\frac1p}\Norm{\mathbf{1}_{E\times
  E}F}{L^{p}(\rho^{-ps}V^{-1})}
  \lesssim\Norm{\rho^{-1}V^{-\gamma}F}{L^{p,\infty}(V^{p\gamma-1})},
\end{equation*}
where the implicit positive constant depends only on $p,\gamma$, and the doubling constant
of $X$.
\end{lemma}

\begin{proof}
The proof of the present lemma is a technical modification of Lemma \ref{lem:BSVY963}. For
completeness, we give the details. As before, we consider different cases of
$\gamma\in\mathbb{R}$.

\emph{Case 0: $\gamma=0$}. For $\gamma=0$, the statement of Lemma \ref{lem:BSVY105} reduces
to the corresponding case of Lemma \ref{lem:BSVY963}, which we have already proved.

\emph{Case 1: $\gamma\in(0,\infty)$}. Let the symbol $D$ denote the upper dimension
of $X$, i.e., a positive constant such that, for all $x\in X$ and $0<r\le R<\infty$,
\begin{align}\label{eq:db}
\frac{V(x,R)}{V(x,r)}\lesssim \Big(\frac{R}{r}\Big)^D
\end{align}
with the implicit positive constant independent of $x,r$, and $R$.
Given $s\in(0,1)$, using the Lorentz space duality \eqref{eq:BSVY4-2} with $q:=
\frac{1+\gamma D}{s+\gamma D}\in (1,\infty)$, we obtain
\begin{equation}\label{eq:BSVY-v2}
\begin{split}
&\iint_{E\times E}\frac{|F|^p}{\rho^{ps}V}
=\iint_{E\times E}\Big(\frac{|F|}{\rho^{s}V^\gamma}\Big)^pV^{p\gamma-1}    \\
&\quad=\iint_{X\times X}\Big(\frac{|F|}{\rho V^\gamma}\Big)^{\frac{p(s+D\gamma)}{1+D\gamma}}
\mathbf{1}_{E\times E}\Big(\frac{|F|
\rho^{D\gamma}}{V^\gamma}\Big)^{\frac{p(1-s)}{1+D\gamma}}
V^{p\gamma-1}\\
&\quad\le \frac{1+D\gamma}{1-s}
\BNorm{\Big(\frac{|F|}{\rho V^\gamma}\Big)^{\frac{p(s+D\gamma)}{1+D\gamma}}}{
L^{\frac{1+D\gamma}{s+D\gamma},\infty} (V^{p\gamma-1})} 
\BNorm{\mathbf{1}_{E\times E}\Big(\frac{|F|
\rho^{D\gamma}}{V^\gamma}\Big)^{\frac{p(1-s)}{1+D\gamma}}}{
L^{\frac{1+D\gamma}{1-s},1} (V^{p\gamma-1})}\\
&\quad\le \frac{1+D\gamma}{1-s}
\BNorm{\frac{F}{\rho V^\gamma}}{
L^{p,\infty}
(V^{p\gamma-1})}^{\frac{p(s+D\gamma)}{1+D\gamma}}\Norm{F}{L^\infty(\mu\times\mu)}
^{\frac{p(1-s)}{1+D\gamma}} 
\BNorm{\mathbf{1}_{E\times
E}\Big(\frac{\rho^{D\gamma}}{V^\gamma}\Big)^{\frac{p(1-s)}{1+D\gamma}}}{
L^{\frac{1+D\gamma}{1-s},1} (V^{p\gamma-1})}.
\end{split}
\end{equation}
Now, we estimate the last $L^{\frac{1+D\gamma}{1-s},1}$-norm.
Since $E$ is bounded, we infer that $E\subset B(x_0,R)$ for some $x_0\in X$
and $R\in(0,\infty)$. Clearly, for any $x,y\in E$, $B(x,\rho(x,y))\subset
B(x_0,3R)$. Applying this and \eqref{eq:db}, we find that, for any $x,y\in E$,
\begin{align}\label{eq:pvm}
\frac{\rho(x,y)^{D\gamma}}{V(x,y)^\gamma}\lesssim
\frac{R^{D\gamma}}{V(x_0,R)^\gamma}=:M.
\end{align}
Then we obtain
\begin{equation}\label{eq:lot1}
\begin{split}
&\BNorm{\mathbf{1}_{E\times
E}\Big(\frac{\rho^{D\gamma}}{V^\gamma}\Big)^{\frac{p(1-s)}{1+D\gamma}}}{
L^{\frac{1+D\gamma}{1-s},1} (V^{p\gamma-1})}\\
&\quad=\int_{0}^{\infty}\Big[\iint_{E\times E}
\mathbf{1}_{\{(\frac{\rho^{D\gamma}}{V^\gamma})^{\frac{p(1-s)}{1+D\gamma}}>\lambda
\}}V^{p\gamma-1}\Big]^{\frac{1-s}{1+D\gamma}}\,d\lambda\\
&\quad\le \int_{0}^{M^{\frac{p(1-s)}{1+D\gamma}}}
\Big[\iint_{E\times
E}V(x,y)^{p\gamma-1}\,d\mu(y)\,d\mu(x)\Big]^{\frac{1-s}{1+D\gamma}}\,d\lambda,
\end{split}
\end{equation}
where the last step used \eqref{eq:pvm}.
Note that, for any $x,y\in E$, $V(x,y)\le V(x_0,3R)$,
which, combined with \eqref{eq:ele} with $\alpha:=p\gamma$ and the
doubling condition of $\mu$, further implies that
\begin{align*}
\iint_{E\times E}V^{p\gamma-1}\le
\int_E \int_{\{y\in X:V(x,y)\le V(x_0,3R)\}}V(x,y)^{p\gamma-1}\,d\mu(y)\,d\mu(x)
\lesssim V(x_0,R)^{p\gamma+1}.
\end{align*}
This, together with \eqref{eq:lot1}, yields
\begin{align*}
\BNorm{\mathbf{1}_{E\times
E}\Big(\frac{\rho^{D\gamma}}{V^\gamma}\Big)^{\frac{p(1-s)}{1+D\gamma}}}{
L^{\frac{1+D\gamma}{1-s},1} (V^{p\gamma-1})}\le
M^{\frac{p(1-s)}{1+D\gamma}}\Big[cV(x_0,R)^{p\gamma+1}\Big]^{\frac{1-s}{1+D\gamma}},
\end{align*}
which clearly converges to $1$ as $s\nearrow 1$.
Substituting this estimate into \eqref{eq:BSVY-v2}, we obtain
\begin{equation*}
  \limsup_{s\nearrow 1}(1-s)\iint_{E\times E}
     \frac{|F|^p}{\rho^{ps}V}
   \leq(1+D\gamma)\BNorm{\frac{F}{\rho V^\gamma}}{
L^{p,\infty} (V^{p\gamma-1})}^p,
\end{equation*}
which completes the proof in this case.

\emph{Case 2: $\gamma\in(-\infty,0)$}.
Given $s\in(0,1)$, applying \eqref{eq:BSVY4-2} with
$q:=\frac{2}{1+s}\in(1,\infty)$,
we conclude that
\begin{equation}
\begin{split}
&\iint_{E\times E}\frac{|F|^p}{\rho^{ps}V}
=\iint_{X\times X}\Big(\frac{|F|}{\rho V^\gamma}\Big)^{\frac{p(1+s)}{2}}
\mathbf{1}_{E\times E}\Big(\frac{|F| \rho}{V^\gamma}\Big)^{\frac{p(1-s)}{2}}
V^{p\gamma-1}\\
&\quad\le \frac{2}{1-s}
\BNorm{\Big(\frac{|F|}{\rho V^\gamma}\Big)^{\frac{p(1+s)}{2}}}{
L^{\frac{2}{1+s},\infty} (V^{p\gamma-1})}
\BNorm{\mathbf{1}_{E\times E}\Big(\frac{|F| \rho}{V^\gamma}\Big)^{\frac{p(1-s)}{2}}}{
L^{\frac{2}{1-s},1} (V^{p\gamma-1})}\\
&\quad\le \frac{2}{1-s}
\BNorm{\frac{F}{\rho V^\gamma}}{
L^{p,\infty} (V^{p\gamma-1})}^{\frac{p(1+s)}{2}}\Norm{F}{L^\infty(\mu\times\mu)}
^{\frac{p(1-s)}{2}}
\BNorm{\mathbf{1}_{E\times E}\Big(\frac{\rho}{V^\gamma}\Big)^{\frac{p(1-s)}{2}}}{
L^{\frac{2}{1-s},1} (V^{p\gamma-1})}.
\end{split}
\end{equation}
Since $\gamma\in(-\infty,0)$, for all $x,y\in E\subset B(x_0,R)$,
it follows from the doubling condition that
\begin{align*}
\frac{\rho(x,y)}{V(x,y)^\gamma}
\lesssim RV_0^{-\frac1D}V(x,y)^{\frac1D-\gamma},
\qquad
V_0:=V(x_0,R),
\end{align*}
where the implicit positive constants depend only on the doubling constant and $\gamma$.
Denoting
\begin{equation*}
  t:=\Big(\frac 1D-\gamma\Big)\frac{p(1-s)}{2}=\frac{p(1-\gamma D)(1-s)}{2D}>0,
\end{equation*}
it follows that
\begin{align*}
  &\BNorm{\mathbf{1}_{E\times E}\Big(\frac{\rho}{V^\gamma}\Big)^{\frac{p(1-s)}{2}}}{
  L^{\frac{2}{1-s},1} (V^{p\gamma-1})}
  \leq (cRV_0^{-\frac{1}{D}})^{\frac{p(1-s)}{2}}
  \BNorm{\mathbf{1}_{E\times E} V^t}{L^{\frac{2}{1-s},1}(V^{p\gamma-1})} \\
  &\quad=(cRV_0^{-\frac{1}{D}})^{\frac{p(1-s)}{2}}\int_{0}^{(cV_0)^t}
  \Big(\iint_{E\times E}\mathbf{1}_{\{V>\lambda^{\frac 1t}\}}
  V^{p\gamma-1}\Big)^{\frac{1-s}{2}}\,d\lambda.
\end{align*}
Recalling that $\gamma\in(-\infty,0)$ and using \eqref{eq:ele}, it follows that
\begin{equation*}
\begin{split}
  \int_0^{(cV_0)^t} &\Big(\iint_{E\times E} \mathbf{1}_{\{V>\lambda^{\frac 1t}\}}
  V^{p\gamma-1}\Big)^{\frac{1-s}{2}}\,d\lambda   \\
  &\leq \int_0^{(cV_0)^t} \Big[\int_{B(x_0,R)}\int_{\{y:V(x,y)>\lambda^{\frac 1t}\}}
  V(x,y)^{p\gamma-1}d\mu(y)\,d\mu(x)
  \Big]^{\frac{1-s}{2}}\,d\lambda   \\
  &\leq \int_0^{(cV_0)^t} \big( cV_0
  \lambda^{\frac{p\gamma}{t}}\big)^{\frac{1-s}{2}}d\lambda
  = (cV_0)^{\frac{1-s}{2}} \int_0^{(cV_0)^t} \lambda^{\frac{\gamma D}{1-\gamma D}}d\lambda
  \\
  &=(cV_0)^{\frac{1-s}{2}}(1-\gamma D)(cV_0)^{\frac{t}{1-\gamma D}}
 =(1-\gamma D)(cV_0)^{\frac{1-s}{2}(1+\frac{p}{D})}.
\end{split}
\end{equation*}
Substituting back, it follows that
\begin{align*}
\limsup_{s\nearrow 1}(1-s)\iint_{E\times E}
     \frac{|F|^p}{\rho^{ps}V}
   \leq2(1-D\gamma)\BNorm{\frac{F}{\rho V^\gamma}}{
L^{p,\infty} (V^{p\gamma-1})}^p,
\end{align*}
since all other factors above are some finite quantities raised to the power $1-s\to 0$.
This completes the proof of Lemma \ref{lem:BSVY105}.
\end{proof}

Combining Lemmas \ref{lem:BSVY963} and \ref{lem:BSVY105}, we obtain the following:

\begin{corollary}\label{cor:DMS}
Let $(X,\rho,\mu)$ be a doubling metric measure space.
Let $p\in[1,\infty)$, $\gamma\in\mathbb R$, and $\phi\in\{\rho,V\}$. Then, for all functions
$f\in L^1_{\mathrm{loc}}(\mu)$, truncation levels $N\in\mathbb N$, and bounded measurable
sets $E\subset X$, we have
\begin{equation*}
  \limsup_{s\nearrow 1}(1-s)^{\frac1p}\Norm{\mathbf{1}_{E\times E}\Delta
  f_N}{L^p(\rho^{-ps}V^{-1})}
  \lesssim\Norm{\rho^{-1}\phi^{-\gamma}\Delta f}{L^{p,\infty}(\phi^{p\gamma}V^{-1})},
\end{equation*}
where $f_N$ is the usual truncation \eqref{eq:fN-def} and the implicit
positive constant is independent of $f$, $N$, and $E$.
\end{corollary}

\begin{proof}
It is evident that $f_N\in L^\infty(\mu)$, and hence $\Delta f_N\in L^\infty(\mu\times\mu)$.
By Lemma \ref{lem:BSVY963} if $\phi=\rho$ or Lemma \ref{lem:BSVY105} if $\phi=V$, applied to
$F:=\Delta f_N$ in each case, we obtain a bound like the claimed one but with $f_N$ in place
of $f$ on the right. The proof is completed by \eqref{eq:fN-key}, which shows that $\Delta
f_N\leq\Delta f$.
\end{proof}

We are now ready to complete the proofs of Proposition \ref{prop:DMS} and Theorems
\ref{thm:Nguyen-type} and \ref{thm-Nguyen-type-v2}:

\begin{proof}[Proof of \eqref{eq:DMS-upper-v2}, \eqref{eq:Nguyen-lower}, and
\eqref{eq:Nguyen-lower-v2}]
Denoting by $\sup_{E,N}$ the supremum over all bounded measurable subsets $E\subset X$ and
$N\in\mathbb N$, we have, for both choices of $\phi\in\{\rho,V\}$,
\begin{equation*}
\begin{split}
  \Norm{f}{\dot W^{1,p}(\mu)}
  &\lesssim\sup_{E,N}\liminf_{s\nearrow 1} (1-s)^{\frac1p}\Norm{\mathbf{1}_{E\times E}\Delta
  f_N}{L^p(\rho^{-ps}V^{-1})}
  \quad\text{by \eqref{eq:DMS-lower-v2}} \\
  &\leq\sup_{E,N}\limsup_{s\nearrow 1} (1-s)^{\frac1p}\Norm{\mathbf{1}_{E\times E}\Delta
  f_N}{L^p(\rho^{-ps}V^{-1})}
  \quad\text{trivially} \\
  &\lesssim\Norm{\rho^{-1}\phi^{-\gamma}\Delta f}{L^{p,\infty}(\phi^{p\gamma}V^{-1})}
  \quad\text{by Corollary \ref{cor:DMS}} \\
  &\lesssim\Norm{f}{\dot M^{1,p}(\mu)}\quad\text{by \eqref{eq:Nguyen-upper} (if $\phi=\rho$)
  or
    \eqref{eq:Nguyen-upper-v2} (if $\phi=V$)} \\
  &\sim\Norm{f}{\dot W^{1,p}(\mu)}\quad\text{by Proposition \ref{prop:MWE}}.
\end{split}
\end{equation*}
The first three lines give a proof of both \eqref{eq:Nguyen-lower} (taking $\phi=\rho$)
and \eqref{eq:Nguyen-lower-v2} (taking $\phi=V$). Under the additional assumption that $X$
is a $p$-Poincar\'e space, the last four lines give a proof of \eqref{eq:DMS-upper-v2}.
This completes the proof of all three inequalities, and hence of Proposition
\ref{prop:DMS} and Theorems \ref{thm:Nguyen-type} and \ref{thm-Nguyen-type-v2}.
\end{proof}

\section{Further variants of Sobolev characterizations}\label{sec:further-variants}

The characterizations established above involve either pointwise differences
of $f$ or scale-dependent mean oscillations. In this section, we present
further variants that combine these two viewpoints: to a pair $(x,y)$ we
associate the ball $B(x,\rho(x,y))$ and measure the oscillation of $f$ on this
ball. This leads to characterizations of Sobolev seminorms in terms of the
following weak-type norm:
\begin{align}\label{eq:wn-mf}
\Big\|\frac{m_f}{\rho \cdot\phi^\gamma}\Big\|_{L^{p,\infty}(\phi^{\gamma p}V^{-1})},
\end{align}
where $\phi\in\{\rho,V\}$ and, for any $x,y\in X$,
$m_f(x,y)$ is as in \eqref{eq-def:mf}.
Indeed, if $\phi=\rho$, then the weak-type norm \eqref{eq:wn-mf}
reduces to the one in the right-hand side of \eqref{eq:Frank-lower}
by a relatively standard argument.

Recall that a metric measure space $(X,\rho,\mu)$
is said to satisfy the \emph{reverse doubling} condition if
there exist constants $a,C_0\in (1,\infty)$ such that, for any $x\in X$ and
$r\in(0,\frac{{\rm diam}(X)}{a})$,
\begin{align}\label{rdb}
V(x,ar)\ge C_0V(x,r);
\end{align}
see, for instance, \cite[Definition 1.1(ii)]{YZ:11}.

\begin{remark}\label{rem:Poincare>RD}
Every doubling $p$-Poincar\'e space is reverse doubling. Indeed, a $p$-Poincar\'e space
is connected by \cite[Proposition 4.2]{BB:book}, and a connected doubling space is reverse
doubling by \cite[Corollary 3.8]{BB:book}.
\end{remark}

We first record a localized estimate that will be used below for both choices
of $\phi$.  The localization is important when $X$ has finite diameter: the
reverse doubling condition is then only available below a fixed scale, whereas
the remaining large scales make no contribution to the BBM limit.

\begin{lemma}\label{lem:BSVY161}
Let $(X,\rho,\mu)$ be a doubling and reverse doubling metric measure space,
let $p\in[1,\infty)$, $\gamma\in\mathbb R$, and $\phi\in\{\rho,V\}$.
If $f\in L^\infty(\mu)$ and $E\subset X$ is a bounded measurable set, then
\begin{equation}\label{eq:mean-to-difference-BBM}
 \limsup_{s\nearrow 1}(1-s)^{\frac1p}
 \Norm{\mathbf{1}_{E\times E}\Delta f}{L^p(\rho^{-ps}V^{-1})}
 \lesssim
 \Big\|\frac{m_f}{\rho\phi^\gamma}\Big\|_
 {L^{p,\infty}(\phi^{\gamma p}V^{-1})},
\end{equation}
where the implicit positive constant is independent of $f$ and $E$.
\end{lemma}

\begin{proof}
From \cite[Lemma 2.2]{DLYYZ} (see also the proof of
\cite[Lemma 2.1]{DLYYZ}),
\begin{equation}\label{eq:swap2}
 |\Delta f|\lesssim\sum_{j=0}^\infty
 (F_j+F_j\circ\operatorname{swap}),
 \qquad \operatorname{swap}(x,y):=(y,x),
\end{equation}
where, with $a$ as in \eqref{rdb},
\begin{equation*}
 F_j(x,t):=m_f(x,a^{-j+1}t),\qquad
 F_j(x,y):=F_j(x,\rho(x,y)),\qquad j\in\mathbb Z_+.
\end{equation*}
Suppose first that $\operatorname{diam}(X)=\infty$.  If
$E\subset B(x_0,R)=:B_0$, then $y\in B(x,2R)$ for any $x,y\in E$.
Let $k_R:=\lceil\log_a(2R)\rceil$,
where $\lceil u\rceil$ for any $u\in\mathbb R$
denotes the smallest integer greater than or equal to $u$.
For any $j\in\mathbb Z_+$ and
$s\in(0,1)$, decomposition into the annuli
$B(x,a^k)\setminus B(x,a^{k-1})$, followed by a change of indices, gives
\begin{equation}\label{eq:annular-decomposition-mf}
\begin{aligned}
&\iint_{E\times E}\frac{F_j(x,y)^p}{\rho(x,y)^{ps}V(x,y)}
 \,d\mu(y)\,d\mu(x)\\
&\quad\lesssim
 \int_E\sum_{k=-\infty}^{k_R}
 \int_{B(x,a^k)\setminus B(x,a^{k-1})}
 \frac{F_j(x,y)^p}{\rho(x,y)^{ps}V(x,y)}
 \,d\mu(y)\,d\mu(x)\\
&\quad\lesssim
 \int_E\sum_{k=-\infty}^{k_R}
 \frac{F_j(x,a^k)^p}{a^{kps}}\,d\mu(x)\\
&\quad\lesssim
 a^{-jps}\int_E\sum_{k=-\infty}^{k_R+1}
 \frac{m_f(x,a^k)^p}{a^{kps}}\,d\mu(x)\\
&\quad\lesssim
 a^{-jps}\int_E\sum_{k=-\infty}^{k_R+1}
 \int_{B(x,a^{k+1})\setminus B(x,a^k)}
 \frac{m_f(x,a^k)^p}{a^{kps}}
 \frac{d\mu(y)}{V(x,a^k)}\,d\mu(x)\\
&\quad\lesssim
 a^{-jps}\int_E\int_{3a^2B_0}
 \frac{m_f(x,y)^p}{\rho(x,y)^{ps}V(x,y)}
 \,d\mu(y)\,d\mu(x).
\end{aligned}
\end{equation}
Here the penultimate estimate follows from \eqref{rdb}; in particular,
\begin{equation*}
 \mu(B(x,a^{k+1})\setminus B(x,a^k))
 \gtrsim V(x,a^k).
\end{equation*}
Taking the $p$-th root in \eqref{eq:annular-decomposition-mf}, using
\eqref{eq:swap2}, Minkowski's inequality, and $V(x,y)\sim V(y,x)$,
and then summing in $j$, we obtain
\begin{align}\label{eq:s3}
 \Norm{\mathbf{1}_{E\times E}\Delta f}{L^p(\rho^{-ps}V^{-1})}
 \lesssim
 \Norm{\mathbf{1}_{3a^2B_0\times3a^2B_0}m_f}
 {L^p(\rho^{-ps}V^{-1})},
\end{align}
where we also used that $\sum_{j=0}^\infty a^{-js}$ is bounded uniformly
for $s\in[\frac{1}{2},1)$.

We next suppose that $D:=\operatorname{diam}(X)\in(0,\infty)$ and let
$\delta:=D/a^3$.  On the set $\{(x,y):\rho(x,y)<\delta\}$, all the radii
in \eqref{eq:annular-decomposition-mf} are less than $D/a$.
Consequently, restricting the annular sums in
\eqref{eq:annular-decomposition-mf} to $\rho(x,y)<\delta$ yields
\begin{align}\label{eq:s3-local}
 &\Norm{\mathbf{1}_{E\times E}\mathbf{1}_{\{\rho<\delta\}}\Delta f}
 {L^p(\rho^{-ps}V^{-1})}
 \lesssim \Norm{m_f}{L^p(\rho^{-ps}V^{-1})}.
\end{align}
Here $X\times X$ may be used on the right because $X$ is bounded.

For the complementary set, doubling and
$X\subset B(x,2D)$ imply that
\begin{equation*}
 V(x,\delta)\gtrsim V(x,2D)=\mu(X).
\end{equation*}
Since $|\Delta f|\leq2\|f\|_{L^\infty(\mu)}$, we obtain
\begin{align*}
 \iint_{E\times E}\mathbf{1}_{\{\rho\geq\delta\}}
 \frac{|\Delta f|^p}{\rho^{ps}V}\,d\mu(y)\,d\mu(x)\lesssim
 \|f\|_{L^\infty(\mu)}^p\delta^{-ps}\mu(E).
\end{align*}
It follows that
\begin{align}\label{eq:large-scales-vanish}
 \limsup_{s\nearrow 1}(1-s)^{\frac1p}
 \Norm{\mathbf{1}_{E\times E}\mathbf{1}_{\{\rho\geq\delta\}}\Delta f}
 {L^p(\rho^{-ps}V^{-1})}=0.
\end{align}
The case $D=0$ is trivial.

If $\phi=\rho$, we apply Lemma \ref{lem:BSVY963} with $F=m_f$;
if $\phi=V$, we apply Lemma \ref{lem:BSVY105} with $F=m_f$.
For infinite diameter we apply the appropriate lemma on
$3a^2B_0\times3a^2B_0$ and use \eqref{eq:s3}.  For finite diameter we
apply it on $X\times X$ and use \eqref{eq:s3-local} and
\eqref{eq:large-scales-vanish}.  These estimates give
\eqref{eq:mean-to-difference-BBM} in both cases.
This completes the proof of Lemma \ref{lem:BSVY161}.
\end{proof}

\begin{corollary}\label{cor:mf-reverse-doubling}
Let $p\in(1,\infty)$, $\gamma\in\mathbb R\setminus\{0\}$, and let
$(X,\rho,\mu)$ be a complete doubling and reverse doubling metric measure
space. If $f\in L^1_{\mathrm{loc}}(\mu)$, then, for any
$\phi\in\{\rho,V\}$,
\begin{equation}\label{eq:mf-reverse-doubling}
 \Norm{f}{\dot W^{1,p}(\mu)}
 \lesssim
 \Big\|\frac{m_f}{\rho\phi^\gamma}\Big\|_
 {L^{p,\infty}(\phi^{\gamma p}V^{-1})}
 \lesssim
 \Norm{f}{\dot M^{1,p}(\mu)}.
\end{equation}
\end{corollary}

\begin{proof}
The second inequality in \eqref{eq:mf-reverse-doubling} follows directly from
Lemma \ref{lem:Msharpf<Mh}, inequality \eqref{eq:ele} if $\phi=V$ or its
analogue if $\phi=\rho$, and the Hardy--Littlewood maximal theorem.

To prove the first inequality, let $f_N$ be the standard truncation
\eqref{eq:fN-def}. By \eqref{eq:fN-key}, for any ball $B$, we have
$ m_{f_N}(B)\le 2m_f(B).$
Applying Lemma \ref{lem:BSVY161} to $f_N\in L^\infty(\mu)$, we obtain, for
any bounded measurable set $E\subset X$,
\begin{align*}
 &\limsup_{s\nearrow 1}(1-s)^{\frac1p}
 \Norm{\mathbf{1}_{E\times E}\Delta f_N}{L^p(\rho^{-ps}V^{-1})} \\
 &\quad\lesssim
 \Big\|\frac{m_{f_N}}{\rho\phi^\gamma}\Big\|_
 {L^{p,\infty}(\phi^{\gamma p}V^{-1})}
 \lesssim
 \Big\|\frac{m_f}{\rho\phi^\gamma}\Big\|_
 {L^{p,\infty}(\phi^{\gamma p}V^{-1})}.
\end{align*}
The right-hand side is independent of $E$ and $N$. Proposition
\ref{prop:DMS}, together with $\liminf\leq\limsup$, therefore yields that
\begin{align*}
 \Norm{f}{\dot W^{1,p}(\mu)}
 \lesssim\sup_{E,N}\liminf_{s\nearrow 1}(1-s)^{\frac1p}
 \Norm{\mathbf{1}_{E\times E}\Delta f_N}{L^p(\rho^{-ps}V^{-1})}
 \lesssim
 \Big\|\frac{m_f}{\rho\phi^\gamma}\Big\|_
 {L^{p,\infty}(\phi^{\gamma p}V^{-1})}.
\end{align*}
This proves the first inequality in \eqref{eq:mf-reverse-doubling} and
completes the proof.
\end{proof}

\begin{theorem}\label{thm:pd-Nguyen-type}
Let $p\in(1,\infty)$, $\gamma\in\mathbb{R}\setminus\{0\}$,
and $(X,\rho,\mu)$ be a complete doubling $p$-Poincar\'e space. Let
$\phi\in\{V,\rho\}$. If $f\in L^1_{\mathrm{loc}}(\mu)$, then
\begin{align}\label{eq:pd-Nguyen-type-Poincare}
 \Norm{f}{\dot W^{1,p}(\mu)}
 \sim\Big\|\frac{m_f}{\rho \phi^\gamma}\Big\|_{L^{p,\infty}(\phi^{\gamma p}V^{-1})}
 \sim\Norm{f}{\dot M^{1,p}(\mu)}.
\end{align}
In particular, $f\in\dot W^{1,p}(\mu)$ if and only if the middle
quantity in \eqref{eq:pd-Nguyen-type-Poincare} is finite.
\end{theorem}

\begin{proof}
By Remark \ref{rem:Poincare>RD}, the space $X$ is reverse doubling, so Corollary
\ref{cor:mf-reverse-doubling} applies. Combining
\eqref{eq:mf-reverse-doubling} with Proposition \ref{prop:MWE} completes the
proof of \eqref{eq:pd-Nguyen-type-Poincare} and hence the theorem.
\end{proof}

\section{BV characterizations}\label{sec:BV}

In this section, we work with the characterizations of the homogeneous $\dot{\rm BV}$ space.
Recall that $\dot{\rm BV}(\mu)$ is defined by relaxation in Definition
\ref{def:BV-Mir}. Although Proposition \ref{p-eq} identifies this space
with the two curve-based homogeneous BV spaces, the proof below only uses
the one-sided comparison \eqref{eq:BV-one-sided-comparison}.

It is known that the reverse doubling property
\eqref{rdb} implies that there exist a
positive constant $s\in (0,\infty)$
and a positive $C\in(0,1]$ such that,
for any $x\in X$, $r\in(0,\frac{{\rm diam}( X)}{3})$,
and $\lambda\in[1,\frac{{\rm diam}( X)}{3r})$,
\begin{align}\label{RD}
V(x,\lambda r)\ge C\lambda^{s}V(x,r).
\end{align}
In what follows, we denote by $d\in(0,\infty)$ the \emph{lower homogeneous dimension} of
$\mu$,
which is the maximal value of the quantity $s$ such that \eqref{RD} holds
(cf.\ \cite[(1.14)]{DLYYZ}).

\begin{theorem}\label{thm:BV-characterizations}
Let $(X,\rho,\mu)$ be a complete doubling $1$-Poincar\'e space and
$f\in L^1_{\mathrm{loc}}(\mu)$.
For both $F\in\{m_f,\Delta f\}$ on $X\times X$,
the following statements hold.
\begin{enumerate}[\rm(i)]
\item If $\gamma\in(-\infty,-1)\cup(0,\infty)$, then
\begin{equation}\label{eq:cBV-rho}
\begin{split}
|f|_{\dot{\rm{BV}}(\mu)}
&\lesssim\Norm{\rho^{-1-\gamma}F}{L^{1,\infty}(\rho^\gamma V^{-1})}\\
&\lesssim\sup_{\theta\in(0,1]}
\Norm{t^{-\gamma}\mathcal M_\theta^{\#}f}
{L^{1,\infty}(t^{\gamma-1}dt\,d\mu)}
\lesssim|f|_{\dot{\rm{BV}}(\mu)}.
\end{split}
\end{equation}
\item If $\gamma\in(-\infty,-\frac1d)\cup(0,\infty)$, then
\begin{equation}\label{eq:cBV-volume}
\begin{split}
|f|_{\dot{\rm{BV}}(\mu)}
&\lesssim\Norm{\rho^{-1}V^{-\gamma}F}{L^{1,\infty}(V^{\gamma-1})}\\
&\lesssim\sup_{\theta\in(0,1]}
\Norm{V^{-\gamma}f^*_{\theta}}{L^{1,\infty}(V^{\gamma-1})}
\lesssim|f|_{\dot{\rm{BV}}(\mu)}.
\end{split}
\end{equation}
\end{enumerate}
Consequently, in either part, $f\in\dot{\rm{BV}}(\mu)$ if and only if
any one of the corresponding intermediate quantities is finite.
\end{theorem}

As observed at the beginning of the preliminaries,
the underlying metric space $(X,\rho)$ of a
complete doubling metric measure space $(X,\rho,\mu)$ is proper and hence
locally compact and Hausdorff.
We denote by $\mathscr{M}(X)$
the set of all complex Radon measures on $X$, which is known as the dual space of
$C_0(X)$, the closure of $C_{\rm c}(X)$
(the set of all continuous functions with compact support on $X$) in the uniform norm. The
norm of $\nu\in\mathscr M(X)$ is given by
\begin{equation*}
  \Norm{\nu}{\mathscr M(X)}:=\abs{\nu}(X),
\end{equation*}
where $\abs{\nu}$ is the total variation measure of the complex measure $\nu$.

\begin{lemma}\label{lem:weakL1>liminf}
Let $\gamma\in\mathbb R$, and let $(X,\rho,\mu)$ be a complete doubling
and reverse doubling metric measure space. Let $f\in L^\infty(\mu)$ be such that
\begin{equation}\label{eq:Gammaf}
  \Gamma_f:=\min\bigg\{\Big\|\frac{F}{\rho
  \cdot\phi^\gamma}\Big\|_{L^{1,\infty}(\phi^{\gamma}V^{-1})}:
  F\in\{m_f,\Delta f\}\ \text{and}\ \phi\in\{\rho,V\}\bigg\}<\infty.
\end{equation}
Let the assumptions and notation of Lemma \ref{lem:gnt} be in force.
Then, for every $n\geq n_0$, there are a sequence $t_k\searrow 0$ and $g^n\in
\mathscr{M}(X_n)$
such that the functions \eqref{eq:gnt} satisfy
$g^n_{t_k}\mu_n\overset{*}{\rightharpoonup} g^n$ in the weak$^*$ topology of
$\mathscr{M}(X_n)$
and
\begin{equation*}
  \Norm{g^n}{\mathscr{M}(X_n)}\lesssim\Gamma_f,
\end{equation*}
where the implicit constant is independent of $n$ and $f$.
\end{lemma}

\begin{proof}
Combining several results from above, we find that
\begin{equation*}
\begin{split}
  \liminf_{t\searrow 0}\Norm{g^n_t}{L^1(\mu)}
  &\leq\liminf_{s\nearrow 1}(1-s)s\int_0^1\frac{\Norm{g^n_t}{L^1(\mu)}\, dt}{t^{s}}
  \quad\text{by Lemma \ref{lem:2liminf}} \\
&  \lesssim \limsup_{s\nearrow 1}(1-s)
   \Norm{ \mathbf{1}_{X_n\times X_n}\Delta f }{ L^1(\rho^{-s}V^{-1})}
  \quad\text{by \eqref{eq:weakLp>limsup}} \\
&\lesssim
\begin{cases} \Norm{\rho^{-1}\phi^{-\gamma}m_f}{L^{1,\infty}(\phi^\gamma V^{-1})} &
\text{by Lemma \ref{lem:BSVY161}}, \\
  \Norm{\rho^{-1}\phi^{-\gamma}\Delta f}{L^{1,\infty}(\phi^\gamma V^{-1})} &
  \text{by }
  \begin{cases} \text{Lemma \ref{lem:BSVY963} if $\phi=\rho$, or} \\
  \text{Lemma \ref{lem:BSVY105} if $\phi=V$}.\end{cases}
  \end{cases}
\end{split}
\end{equation*}
Taking the minimum of these upper bounds, the left-hand side is bounded by $\Gamma_f$.

From the definition of the $\liminf$, we deduce the existence of a sequence $t_k\searrow 0$ with
\begin{equation*}
   \Norm{g^n_{t_k}}{L^1(\mu)}
 \lesssim\Gamma_f.
\end{equation*}
It is well known that $L^1(\mu_n)\hookrightarrow \mathscr{M}(X_n)$ and
the bounded sequence in the dual space of a separable
Banach space $C_0(X_n)$ has a weak$^*$ convergent subsequence
$g^n_{t_k}\overset{*}{\rightharpoonup} g^n$, which we continue to denote by the same symbol.
The upper bound on $\Norm{g^n}{\mathscr{M}(X_n)}$ is then a standard property of weak
limits.
This completes the proof of Lemma \ref{lem:weakL1>liminf}.
\end{proof}

The following proposition is the $p=1$ case of the semicontinuity result \cite[Lemma
2.6]{DMS}. It refers to the notion $\abs{Df}_*$ from Definition \ref{def:BVboth2}.

\begin{lemma}\label{prop:semics}
Let $(X,\rho,\mu)$ be a complete doubling metric measure space.
Let $f_n,f\in L^1_{\mathrm{loc}}(\mu)$, $g_n\in L^1_{\mathrm{loc}}(\mu)$, and let
$g$ be a nonnegative Radon measure such that
\begin{enumerate}[\rm(i)]
  \item $g_n$ is an upper gradient of $f_n$ up to scale $\delta_n$, where $\delta_n\searrow
  0$; and
  \item $f_n\to f$ in $L^1_{\mathrm{loc}}(\mu)$ and
  $g_n\mu\overset{*}{\rightharpoonup}g$ in the sense of Radon measures, namely
  $$
  \int_X\varphi g_n\,d\mu\longrightarrow\int_X\varphi\,dg
  \quad\text{for every }\varphi\in C_{\rm c}(X).
  $$
\end{enumerate}
Then $|Df|_*\le g$.
\end{lemma}
We also recall the following lower semicontinuity of the relaxed total
variation; see \cite[Proposition 3.6]{M:03}.
\begin{lemma}\label{lem:lsc}
Let $(X,\rho,\mu)$ be a complete doubling metric measure space, and let
$A\subset X$ be open. If $f_j,f\in L^1_{\mathrm{loc}}(\mu)$ and
$f_j\to f$ in $L^1_{\mathrm{loc}}(\mu)$, then
\begin{equation}\label{eq:lsc-BV}
  |Df|(A)\leq \liminf_{j\to\infty}|Df_j|(A),
\end{equation}
where the total variations are defined by \eqref{eq:Df(A)}.
\end{lemma}

\begin{proof}
If the right-hand side of \eqref{eq:lsc-BV} is infinite, there is nothing
to prove. Otherwise, we pass to a subsequence along which the lower limit is
attained. After discarding finitely many terms, $|Df_j|(A)<\infty$, and
hence $f_j|_A$ has locally finite total variation on $A$. Applying
\cite[Proposition 3.6]{M:03} on the open set $A$ gives
\eqref{eq:lsc-BV}.
This completes the proof of Lemma \ref{lem:lsc}.
\end{proof}

Now, we are in the position of proving Theorem \ref{thm:BV-characterizations}.

\begin{proof}[Proof of Theorem \ref{thm:BV-characterizations}]
The upper bound in \eqref{eq:cBV-rho} for $\gamma\in(-\infty,-1)$
has been proved in \cite[(6.3)]{DLYYZ}. For $\gamma\in(0,\infty)$,
it is indicated in \cite[Section 4.1]{DLYYZ} that, for any $f\in
\operatorname{Lip}_{\mathrm{loc}}(X)$,
\begin{align}\label{eq:s4}
\sup_{\theta\in(0,1]}\Norm{t^{-\gamma}\mathcal
M_\theta^{\#}f}{L^{1,\infty}(t^{\gamma-1}dt\,d\mu)}
   \lesssim\Norm{{\operatorname{lip}}f}{L^1(\mu)};
\end{align}
recall that $\operatorname{lip}f$ is defined in \eqref{eq:lipu}.
Given $f\in \dot{\rm{BV}}(\mu)$ and any sequence $\{f_n\}_{n\in\mathbb{N}}
$ in  $\operatorname{Lip}_{\mathrm{loc}}(X)$ such that $f_n\to f$ in
$L^1_{\mathrm{loc}}(\mu)$,
by \eqref{eq:s4} and Fatou's lemma, we find that
\begin{align*}
&\liminf_{n\to\infty}\Norm{{\operatorname{lip}}f_n}{L^1(\mu)}\\
&\quad\gtrsim\liminf_{n\to\infty}\sup_{\theta\in(0,1]}\Norm{t^{-\gamma}
\mathcal M_\theta^{\#}f_n}{L^{1,\infty}(t^{\gamma-1}dt\,d\mu)}\\
&\quad\ge
\sup_{\theta\in(0,1],\lambda\in(0,\infty)}\lambda\int_X\int_{0}^{\infty}\liminf_{n\to\infty}
{\bf 1}_{\{t^{-\gamma}
\mathcal M_\theta^{\#}f_n(x,t)>\lambda\}}\,t^{\gamma-1}\,dt\,d\mu(x)\\
&\quad\ge \sup_{\theta\in(0,1],\lambda\in(0,\infty)}\lambda\int_X\int_{0}^{\infty}
{\bf 1}_{\{t^{-\gamma}\liminf_{n\to\infty}
\mathcal M_\theta^{\#}f_n(x,t)>\lambda\}}\,t^{\gamma-1}\,dt\,d\mu(x).
\end{align*}
Combined with the fact that $f_n\to f$ in $L^1_{\mathrm{loc}}(\mu)$, this is
further bounded from below by
\begin{align*}
\ge\sup_{\theta\in(0,1]}\Norm{t^{-\gamma}\mathcal
M_\theta^{\#}f}{L^{1,\infty}(t^{\gamma-1}dt\,d\mu)}.
\end{align*}
Taking the infimum over all such sequences $\{f_n\}_n$, and recalling from \eqref{eq:Df(A)}
and \eqref{BV} that
\begin{equation*}
  \abs{f}_{\dot{\rm BV}(\mu)}:=\inf_{\{f_n\}_{n\in\mathbb
  N}}\liminf_{n\to\infty}\Norm{\operatorname{lip}f_n}{L^1(\mu)},
\end{equation*}
where the infimum is taken precisely over the said sequences, we obtain the desired upper
bound in this case.

The upper bound in \eqref{eq:cBV-volume} for
$\gamma\in(-\infty,-\frac1d)$ has been proved in \cite[(6.1)]{DLYYZ}.
For $\gamma\in(0,\infty)$, the following a priori estimate for test functions
$f\in \operatorname{Lip}_{\mathrm{loc}}(X)$ is contained in \cite[Theorem 1.8]{DLYYZ}:
\begin{equation}\label{eq:s5}
  \sup_{\theta\in(0,1]}\Norm{V^{-\gamma}f^{*}_\theta}{L^{1,\infty}(V^{\gamma-1})}
  \lesssim\Norm{{\operatorname{lip}}f}{L^1(\mu)}.
\end{equation}
(Strictly speaking, \cite[Theorem 1.8]{DLYYZ} is stated for $f\in \operatorname{Lip}(X)$
only,
but an inspection of the proof of \eqref{eq:s5} in \cite[Section 4.2]{DLYYZ} shows that
$f\in \operatorname{Lip}_{\mathrm{loc}}(X)$ is enough; indeed, one only needs that
$\operatorname{lip}f$
is an upper gradient of $f$ so as to have access to the Poincar\'e inequality.)
The proof of \eqref{eq:cBV-volume} for $\gamma\in(0,\infty)$ is now concluded in the
same way as the proof of \eqref{eq:cBV-rho} in the same range, using \eqref{eq:s5} in place
of
\eqref{eq:s4}.

Next, we prove the lower bounds in \eqref{eq:cBV-rho} and
\eqref{eq:cBV-volume} for both choices of $F$.
We first consider $f\in L^\infty(\mu)$ as we did in the proof of \eqref{eq:Nguyen-lower}.
Let the assumptions and notation of Lemma \ref{lem:gnt} be in force.
For any $n\in\mathbb{N}$ and $t\in(0,\infty)$, let $f^n:=f|_{X_n}$ and let $f^n_t,
g^n_t$ be as in \eqref{eq:fnt-gnt}.
By Remark \ref{rem:Poincare>RD}, the space $X$ is reverse doubling. Using Lemmas
\ref{lem:weakL1>liminf} and \ref{lem:DMS1869}, we conclude that,
for each $n\geq n_0$, there are a sequence $t_k\searrow 0$ and a
measure $g^n\in\mathscr M (X_n)$ such that
\begin{enumerate}[\rm(i)]
  \item $c\cdot g_{t_k}^n$ is an upper gradient of $f^n_{\frac12 t_k}$ up to scale $\frac14
  t_k$, where $\frac14 t_k\searrow 0$,
  \item $f^n_{\frac12 t_k}\to f^n$ pointwise $\mu_n$-a.e., and
  $g^n_{t_k}\mu_n\overset{*}{\rightharpoonup} g^n$ in the weak$^*$ topology of
  $\mathscr{M}(X_n)$,
\end{enumerate}
and, with $\Gamma_f$ as in \eqref{eq:Gammaf},
\begin{equation}\label{eq:gn<Gamma}
    \Norm{g^n}{\mathscr M(X_n)}
    \lesssim\Gamma_f.
\end{equation}
By its definition in \eqref{eq:fnt-gnt}, each function $f^n_t$ is an average of
$f^n=f|_{X_n}$,
and hence
$$
\Norm{f^n_{\frac12t_k}}{L^\infty(\mu_n)}\leq  \|f^n\|_{L^\infty(\mu_n)}\leq
\|f\|_{L^\infty(\mu)}.
$$
Since $\mu_n(X_n)<\infty$, the pointwise convergence in (ii) and the dominated convergence
theorem
give
$f^n_{\frac12t_k}\longrightarrow f^n$
in $L^1(\mu_n).$
Moreover, since $X_n$ is closed, bounded, and compact, it follows that
$C_{\rm c}(X_n) = C_0(X_n)$ and hence
$c g^n_{t_k}\mu_n\overset{*}{\rightharpoonup}c g^n$
in the sense of Radon measures.
Consequently, the assumptions of Lemma \ref{prop:semics} hold
on $(X_n,\rho|_{X_n},\mu_n)$, and hence
\begin{equation}\label{eq:Dfn<gn}
  |Df^n|_*\le c g^n\quad\text{for every }n\ge n_0.
\end{equation}
It then follows that
\begin{equation*}
\begin{split}
    |Df^n|(X_n) &\leq |Df^n|_w(X_n)\quad\text{by Lemma \ref{lem:tv-w-lower}} \\
  &\leq |Df^n|_*(X_n)\quad\text{by Lemma \ref{lem:tv-c}} \\
  &\le c\|g^n\|_{\mathscr M(X_n)}\quad\text{by \eqref{eq:Dfn<gn}} \\
  &\lesssim\Gamma_f\quad\text{by \eqref{eq:gn<Gamma}}.
\end{split}
\end{equation*}
Hence, Proposition \ref{p-local} implies that $f\in\dot{\rm BV}(\mu)$
and gives
\begin{align*}
|Df|(X)
=\sup_{n\ge n_0}|Df^n|(X_n)
\lesssim\Gamma_f.
\end{align*}
This completes the proof under the assumption that $f\in L^\infty(\mu)$.

Let now $f\in L^1_{\mathrm{loc}}(\mu)$, and consider the standard truncations
$f_N\in L^\infty(\mu)$ defined in \eqref{eq:fN-def}. Applying the bounded
case proved above to $f_N$, we obtain
$|f_N|_{\dot{\rm BV}(\mu)}
\lesssim\Gamma_{f_N}.$
For the quantities involving $\Delta f$, the estimate \eqref{eq:fN-key} gives
$\Delta f_N\le \Delta f$. For the quantities involving $m_f$,  we have, for
every ball $B$,
$m_{f_N}(B)\le2m_f(B)$.
Consequently, in all four cases,
$\Gamma_{f_N}\lesssim\Gamma_f,$
and hence
$|Df_N|(X)\lesssim\Gamma_f.$

Since $f_N\to f$ in $L^1_{\mathrm{loc}}(\mu)$, Lemma \ref{lem:lsc} gives
\begin{align*}
|Df|(X)
&\le \liminf_{N\to\infty}|Df_N|(X)
\lesssim\Gamma_f.
\end{align*}
This completes the proof of Theorem \ref{thm:BV-characterizations}.
\end{proof}

\appendix

\section{Equivalence of homogeneous BV spaces}\label{app:homogeneous-BV-equivalence}

This appendix proves Proposition \ref{p-eq}.  The argument also supplies
the boundary-regularity detail needed to compare the bary-plan definition
with the relaxed total variation.  As explained after Proposition
\ref{p-eq}, the applications in the main text only use the one-sided
comparison \eqref{eq:BV-one-sided-comparison}.

For checking the boundary regularity \eqref{eq:s9}, we first record a
one-dimensional sufficient condition motivated by
\cite[p.~4169]{AD:14}.

\begin{lemma}\label{lem:ADM5.8aux}
Suppose that $g\in L^1(0,1)$ has a finite measure $Dg$ as its
distributional derivative. Suppose further that $0$ and $1$ are weak
Lebesgue points of $g$ in the sense that
\begin{equation}\label{eq:weakLeb}
 \liminf_{t\searrow 0}\frac{1}{t}\int_0^t\abs{g(s)-g(0)}\,ds
 =0
 =\liminf_{t\searrow 0}\frac{1}{t}\int_0^t
 \abs{g(1-s)-g(1)}\,ds.
\end{equation}
Then $g$ satisfies the boundary regularity \eqref{eq:s9}.
\end{lemma}

\begin{proof}
Let $\phi\in C_{\rm c}^\infty(0,1)$ be nonnegative with
$\int_0^1\phi=1$. For $\varepsilon,\delta\in(0,1)$, let
\begin{equation*}
 \psi^0_\varepsilon(t)
 :=\frac{1}{\varepsilon}\phi\left(\frac{t}{\varepsilon}\right),
 \qquad
 \psi^1_\delta(t)
 :=\frac{1}{\delta}\phi\left(\frac{1-t}{\delta}\right).
\end{equation*}
Both functions have integral one. Hence,
\begin{align*}
 \abs{g(0)-g(1)}
 &\leq \Babs{\int_0^1\psi^0_\varepsilon(t)
   \bigl[g(t)-g(0)\bigr]\,dt}\\
 &\quad+\Babs{\int_0^1\psi^1_\delta(t)
   \bigl[g(t)-g(1)\bigr]\,dt}\\
 &\quad+\Babs{\int_0^1
   \bigl[\psi^0_\varepsilon(t)-\psi^1_\delta(t)\bigr]g(t)\,dt}\\
 &=:I(\varepsilon)+II(\delta)+III(\varepsilon,\delta).
\end{align*}
Define
\begin{equation*}
 \Psi_{\varepsilon,\delta}(t)
 :=\int_0^t\bigl[\psi^0_\varepsilon(s)
                 -\psi^1_\delta(s)\bigr]\,ds.
\end{equation*}
Since the two functions have the same integral,
$\Psi_{\varepsilon,\delta}\in C_{\rm c}^\infty(0,1)$ and
$\Norm{\Psi_{\varepsilon,\delta}}{\infty}\leq1$. Integration by parts
therefore gives
\begin{equation*}
 III(\varepsilon,\delta)
 =\Babs{\int_0^1\Psi_{\varepsilon,\delta}'g}
 =\Babs{\int_{(0,1)}\Psi_{\varepsilon,\delta}\,dDg}
 \leq\abs{Dg}(0,1).
\end{equation*}
Moreover,
\begin{align*}
 I(\varepsilon)
 &\leq\Norm{\phi}{\infty}\frac{1}{\varepsilon}
 \int_0^\varepsilon\abs{g(t)-g(0)}\,dt,\\
 II(\delta)
 &\leq\Norm{\phi}{\infty}\frac{1}{\delta}
 \int_0^\delta\abs{g(1-t)-g(1)}\,dt.
\end{align*}
By \eqref{eq:weakLeb}, there are sequences
$\varepsilon_k\searrow 0$ and $\delta_k\searrow 0$ along which the respective
right-hand sides tend to zero. Letting $k\to\infty$ yields
\begin{equation*}
 \abs{g(0)-g(1)}\leq\abs{Dg}(0,1),
\end{equation*}
which is \eqref{eq:s9} and hence completes the proof of Lemma \ref{lem:ADM5.8aux}.
\end{proof}

\begin{remark}\label{rem:ADM5.8aux}
Lemma \ref{lem:ADM5.8} below
can be used to fix an oversight around \cite[Lemma 5.8]{AD:14}. That is, under some
assumptions on a function $f$ and curves $\gamma$, the said lemma asserts that $0$ (and
implicitly $1$, by symmetry) is a Lebesgue point of each $g=f\circ\gamma$; this is a
condition like \eqref{eq:weakLeb} but with $\lim$ in place of $\liminf$. However, an
inspection of their proof shows that only the weaker version \eqref{eq:weakLeb} is
established there. On the other hand, the actual purpose of \cite[Lemma 5.8]{AD:14} is to
verify the boundary regularity property \eqref{eq:s9}, and Lemma \ref{lem:ADM5.8aux} shows
that the obtained lower limits \eqref{eq:weakLeb} are enough for this.
\end{remark}

We first complete the boundary argument suggested by
\cite[Lemma 5.8]{AD:14}.  In contrast to an $\infty$-test plan, a
bary-plan need not have an essentially bounded Lipschitz constant.  The
metric-speed factor is therefore retained and its integrability is
obtained from the barycenter.

\begin{lemma}\label{lem:ADM5.8}
Let $f_n,g_n\in L^1_{\mathrm{loc}}(\mu)$, where $f_n\to f$ in
$L^1_{\mathrm{loc}}(\mu)$, each $g_n$ is nonnegative, and $g_n$ is an upper
gradient of $f_n$.  Then $0$ and $1$ are weak Lebesgue points, in the
sense of \eqref{eq:weakLeb}, of $f\circ\gamma$ for bary-a.e. curve
$\gamma\in AC([0,1];X)$.
\end{lemma}

\begin{proof}
By symmetry, it is enough to consider the point $0$.  Fix a bary-plan
$\pi$ and a bounded open set $U\subset X$.  For $t\in(0,1)$, let
\begin{equation*}
 H_U(t,\gamma)
 :=\mathbf{1}_{\{\gamma([0,t])\subset U\}}\frac{1}{t}
   \int_0^t\abs{f(\gamma(s))-f(\gamma(0))}\,ds,
\end{equation*}
and let $H_U^n$ be defined in the same way with $f_n$ in place of $f$.
The triangle inequality, Tonelli's theorem, and the bounded compression
property give
\begin{equation}\label{eq:appendix-H-error}
 \int H_U(t,\gamma)\,d\pi(\gamma)
 \leq \int H_U^n(t,\gamma)\,d\pi(\gamma)
 +2C(\pi)\int_U\abs{f_n-f}\,d\mu.
\end{equation}
Indeed, the two error terms are estimated using
$(e_s)_\sharp\pi\leq C(\pi)\mu$ at $s\in[0,t]$ and at $s=0$,
respectively.

The upper-gradient property yields
\begin{align*}
 H_U^n(t,\gamma)
 &\leq \mathbf{1}_{\{\gamma([0,t])\subset U\}}\frac{1}{t}
 \int_0^t\int_0^s g_n(\gamma(r))\abs{\gamma'}(r)\,dr\,ds\\
 &\leq \int_0^t\mathbf{1}_U(\gamma(r))g_n(\gamma(r))
 \abs{\gamma'}(r)\,dr.
\end{align*}
The barycenter identity, extended from bounded continuous
functions to nonnegative Borel functions, gives
\begin{align*}
 \int_{AC([0,1];X)}
 \Big(\int_\gamma \mathbf{1}_U g_n\Big)\,d\pi(\gamma)
 =\int_U g_n b_\pi\,d\mu
 \leq \Norm{b_\pi}{L^\infty(\mu)}\int_Ug_n\,d\mu<\infty.
\end{align*}
Thus, for fixed $n$,
$\int_\gamma\mathbf{1}_Ug_n<\infty$ for $\pi$-a.e.\ $\gamma$.
For any such $\gamma$, the preceding estimate and the absolute
continuity of the integral imply that
\begin{equation*}
  \lim_{t\searrow 0}H_U^n(t,\gamma)=0,
  \qquad
  H_U^n(t,\gamma)\leq\int_\gamma\mathbf{1}_Ug_n.
\end{equation*}
Hence, with $n$ fixed, the dominated convergence theorem implies
\begin{equation}\label{eq:appendix-Hn-zero}
 \lim_{t\searrow 0}\int H_U^n(t,\gamma)\,d\pi(\gamma)=0.
\end{equation}
Combining Fatou's lemma with \eqref{eq:appendix-H-error} and
\eqref{eq:appendix-Hn-zero}, we obtain
\begin{align*}
 \int\liminf_{t\searrow 0}H_U(t,\gamma)\,d\pi(\gamma)
 &\leq 2C(\pi)\int_U\abs{f_n-f}\,d\mu.
\end{align*}
Letting $n\to\infty$ shows that the lower limit vanishes for
$\pi$-a.e. curve under consideration.  If $\gamma(0)\in U$, the continuity
of $\gamma$ gives $\gamma([0,t])\subset U$ when $t$ is sufficiently
small.  A countable exhaustion of $X$ by bounded open sets therefore
shows that $0$ is a weak Lebesgue point of $f\circ\gamma$ for
$\pi$-a.e. $\gamma$.  The point $1$ follows by reversing the
parameterization.  Since $\pi$ was arbitrary, the exceptional set is
bary-negligible. This completes the proof of Lemma \ref{lem:ADM5.8}.
\end{proof}

As a final auxiliary result for comparing $\abs{Df}$ and $\abs{Df}_*$, we record the
following
basic lemma. The case of curves $\gamma$ parametrized by arc-length is
contained e.g.\ in the proof of \cite[Lemma 6.2.6]{HKST:book}; in the lack of direct
reference for the case at hand, we give the proof for completeness.

\begin{lemma}\label{lem:chain-rule}
If $f\in\operatorname{Lip}_{\mathrm{loc}}(X)$ and $\gamma\in AC([0,1];X)$, then
$f\circ\gamma\in AC([0,1];\mathbb R)$ and
\begin{equation*}
  \abs{(f\circ\gamma)'}(t)\leq
  (\operatorname{lip}f\circ\gamma)(t)\abs{\gamma'}(t)\quad\text{at a.e. }t\in[0,1].
\end{equation*}
\end{lemma}

\begin{proof}
Let $L:=\Norm{f|_{\gamma([0,1])}}{\operatorname{Lip}}<\infty$ by the assumption that
$f\in\operatorname{Lip}_{\mathrm{loc}}(X)$. Let $h\in L^1(0,1)$ be such that $\gamma$
satisfies \eqref{eq:gamma-AC}. Then, for all $0\leq s<t\leq 1$,
\begin{equation*}
  \abs{(f\circ\gamma)(s)-(f\circ\gamma)(t)}
  \leq L\rho(\gamma(s),\gamma(t))
  \leq L\int_s^t h(u)du,
\end{equation*}
and hence $f\circ\gamma$ satisfies \eqref{eq:gamma-AC} with $Lh\in L^1(0,1)$ in place of
$h$.

For $s,t\in[0,1]$ with $s\neq t$, we have
\begin{equation*}
  \frac{\abs{(f\circ\gamma)(s)-(f\circ\gamma)(t) }}{\abs{s-t}}
  =\frac{\abs{f(\gamma(s))-f(\gamma(t)) }}{\rho(\gamma(s),\gamma(t))}
  \frac{\rho(\gamma(s),\gamma(t))}{\abs{s-t}}
  =:I\cdot II.
\end{equation*}
For a.e.\ $t\in[0,1]$, as $s\to t$, we have $II\to\abs{\gamma'}(t)$ by
\eqref{eq:metric-speed}. On the other hand,
\begin{equation*}
  I\leq\frac{1}{r}\sup_{y\in\bar
  B(\gamma(t),r)}\abs{f(y)-f(\gamma(t))}\Big|_{r=\rho(\gamma(s),\gamma(t))},
\end{equation*}
and the lower limit of this as $s\to t$ (and hence $r=\rho(\gamma(s),\gamma(t))\to 0$) is
equal to $\operatorname{lip}f(\gamma(t))$ by definition \eqref{eq:lipu}. It follows that
\begin{equation*}
  \liminf_{s\to t}I\cdot II
  =\liminf_{s\to t}I\cdot \lim_{s\to t} II
  \leq\operatorname{lip}f(\gamma(t))\cdot\abs{\gamma'}(t).
\end{equation*}
Since we already checked that $f\circ\gamma$ is absolutely continuous, we deduce from
\cite[Theorem 1.1.2]{AGS:book} that the metric speed $\abs{(f\circ\gamma)'}(t)=\lim_{s\to
t}I\cdot II$ exists at a.e.\ $t\in[0,1]$, and hence must agree with the lower limit above.
This completes the proof.
\end{proof}

We next establish the comparison that is not needed in the main
applications but completes the homogeneous equivalence.

\begin{lemma}\label{lem:tv-star}
Let $(X,\rho,\mu)$ be a complete doubling metric measure space.  If
$f\in\dot{\rm BV}(\mu)$, then $f\in\dot{\rm BV}_*(\mu)$ and
$|Df|_*\leq |Df|$ as Borel measures.
\end{lemma}

\begin{proof}
Fix an open set $A\subset X$ with $|Df|(A)<\infty$ and
fix $\varepsilon\in(0,\infty)$.  By the definition of $|Df|(A)$, there is
a sequence $\{f_n\}_{n\in\mathbb N}\subset\operatorname{Lip}_{\mathrm{loc}}(A)$
such that $f_n\to f$ in $L^1_{\mathrm{loc}}(A)$ and, after passing to a
subsequence,
\begin{equation}\label{eq:bv-recovery}
 \lim_{n\to\infty}\int_A\operatorname{lip}f_n\,d\mu
 \leq |Df|(A)+\varepsilon.
\end{equation}
Let $\pi$ be a bary-plan and let $U\subset A$ be open and bounded with
$\overline U\subset A$.  The bounded compression property gives
\begin{align*}
 &\int_{AC([0,1];X)}\int_{\gamma^{-1}(U)}
 \abs{f_n(\gamma(t))-f(\gamma(t))}\,dt\,d\pi(\gamma)\\
 &\qquad\leq C(\pi)\int_U\abs{f_n-f}\,d\mu\longrightarrow0\quad\text{as}\ n\to\infty.
\end{align*}
After passing to a further subsequence,
$f_n\circ\gamma\to f\circ\gamma$ in
$L^1(\gamma^{-1}(U))$ for $\pi$-a.e. $\gamma$ as $n\to\infty$.
For any such curve $\gamma$, we regard $f_n\circ\gamma$ as a function on the open set
$\gamma^{-1}(A)$. Applying Lemma \ref{lem:chain-rule}, with $A$ in place of $X$, to compact
subintervals of the components of $\gamma^{-1}(U)$ and then exhausting these components
gives
\begin{equation*}
  |D(f_n\circ\gamma)|(\gamma^{-1}(U))
  \leq\int_{\gamma^{-1}(U)}
  \operatorname{lip}f_n(\gamma(t))\abs{\gamma'}(t)\,dt.
\end{equation*}
The lower semicontinuity of the one-dimensional total variation and the preceding estimate
imply
\begin{align*}
 &\int_{AC([0,1];X)}
 |D(f\circ\gamma)|(\gamma^{-1}(U))\,d\pi(\gamma)\\
 &\quad\leq\liminf_{n\to\infty}
 \int_{AC([0,1];X)} |D(f_n\circ\gamma)|(\gamma^{-1}(U))\,d\pi(\gamma)
 \quad\text{by lower semicontinuity}\\
 &\quad\leq\liminf_{n\to\infty}
 \int_{AC([0,1];X)}\int_{\gamma^{-1}(U)}
  \operatorname{lip}f_n(\gamma(t))\abs{\gamma'}(t)\,dt\,d\pi(\gamma)
 \quad\text{by the preceding estimate}\\
 &\quad=\liminf_{n\to\infty}\int_U
 \operatorname{lip}f_n\,b_\pi\,d\mu\leq\Norm{b_\pi}{L^\infty(\mu)}
 \bigl[|Df|(A)+\varepsilon\bigr].
\end{align*}
Letting $U\nearrow A$ and then $\varepsilon\searrow 0$ gives
\begin{equation}\label{eq:appendix-bary-variation}
 \int_{AC([0,1];X)}
 |D(f\circ\gamma)|(\gamma^{-1}(A))\,d\pi(\gamma)
 \leq\Norm{b_\pi}{L^\infty(\mu)}|Df|(A).
\end{equation}
In particular, the left-hand side is finite when $A=X$, so
$f\circ\gamma$ has finite one-dimensional variation for $\pi$-a.e.
$\gamma$.

It remains to verify the boundary regularity in \eqref{eq:s9}.
Choose a global recovery sequence
$\{h_n\}_{n\in\mathbb N}\subset\operatorname{Lip}_{\mathrm{loc}}(X)$ for
$|Df|(X)$.  By Remark \ref{rem:lipu}, the function
$g_n:=\operatorname{lip}h_n$ is an upper gradient of $h_n$.  Lemma
\ref{lem:ADM5.8}, applied to $h_n$, $g_n$, and $f$, shows that both
endpoints are weak Lebesgue points of $f\circ\gamma$ for bary-a.e.
$\gamma$.  Lemma \ref{lem:ADM5.8aux} therefore gives
\begin{equation*}
 \abs{f(\gamma(1))-f(\gamma(0))}
 \leq |D(f\circ\gamma)|(0,1)
\end{equation*}
for bary-a.e. $\gamma$.
Consequently, $f$ satisfies both parts of case \eqref{it:BV*} of Definition \ref{def:BVboth}
with $\nu=|Df|$.  Hence $f\in\dot{\rm BV}_*(\mu)$, and
\eqref{eq:appendix-bary-variation} yields $|Df|_*\leq|Df|$.
This completes the proof of Lemma \ref{lem:tv-star}.
\end{proof}

\begin{proof}[Proof of Proposition \ref{p-eq}]
Lemmas \ref{lem:tv-c} and \ref{lem:tv-star} give
$|Df|_w\leq |Df|_*\leq |Df|.$
On the other hand, Lemma \ref{lem:tv-w-lower} gives
$|Df|(X)\leq |Df|_w(X).$
The three total masses are therefore equal.  The same inequalities,
together with the convention that the relevant seminorm is infinite
outside its domain, identify the three homogeneous BV spaces and
then complete the proof of the proposition.
\end{proof}

\end{document}